\documentclass[reqno,11pt]{amsart}
\usepackage{geometry}
\usepackage{mathtools,amssymb,amsthm,mathrsfs,color,lineno,paralist,graphicx,float}
\usepackage[colorlinks,
linkcolor=blue,
anchorcolor=green,
citecolor=blue, 
]{hyperref}
\usepackage{tikz}
\usepackage{graphicx}
\usepackage{longtable}
\usepackage{array}
\usepackage{booktabs}
\usetikzlibrary{shapes.geometric, arrows.meta, positioning, fit, backgrounds, shadows}

\usepackage{enumitem}
\usepackage[T1]{fontenc}
\usepackage[utf8]{inputenc}
\usepackage{subfig} 
\usepackage[justification = centering, labelsep =period]{caption} 
\usepackage[mathcal]{eucal}

\usepackage{calc}
\definecolor{bleu1}{RGB}{0,57,128}
\def\bleu1{\color{bleu1}}

\usepackage{etoolbox}
\patchcmd{\section}{\normalfont}{\normalfont \bleu1}{}{}
\patchcmd{\subsection}{\normalfont}{\normalfont \bleu1}{}{}
\patchcmd{\subsubsection}{\normalfont}{\normalfont \bleu1}{}{}

\usepackage{tikz}
\usetikzlibrary{arrows.meta}

\newtheorem{proposition}{Proposition}[section]
\newtheorem{theorem}{Theorem}[section]
\newtheorem{definition}{Definition}[section] 
\newtheorem{lemma}{Lemma}[section]
\newtheorem{remark}{Remark}[section]
\newtheorem{corollary}{Corollary}[section]

\newtheorem{problem}{Problem}

\newcommand{\Z}{{\mathbb Z}}
\newcommand{\C}{{\mathbb C}}
\newcommand{\R}{{\mathbb R}}
\newcommand{\Q}{{\mathbb Q}}
\newcommand{\T}{{\mathbb T}}

\def\GL{{\mathrm{GL}}}
\def\diag{{\mathrm{diag}}}

\def\0{{\mathbf 0}}

\def\ii{{\textrm i}}

\usepackage{tikz,tikz-3dplot}
\tdplotsetmaincoords{80}{45}
\tdplotsetrotatedcoords{-90}{180}{-90}
\tikzset{surface/.style={draw=blue!70!black, fill=blue!40!white, fill opacity=.6}}

\newcommand{\tangentlabel}[5]{
    \pgfmathsetmacro{\ang}{#3}
    \pgfmathsetmacro{\rot}{\ang+90}
    \pgfmathparse{ifthenelse(\rot>90,\rot-180,ifthenelse(\rot<-90,\rot+180,\rot))}
    \let\finalrot\pgfmathresult
    \node[
        #1,
        rotate=\finalrot,
        anchor=south
    ] at ({(#2+#5)*cos(\ang)},{(#2+#5)*sin(\ang)}) {#4};
}

\usepackage{pgfplots}
\usepgfplotslibrary{polar}
	\usepgfplotslibrary{polar}
\pgfplotsset{compat=1.17}
\makeatletter
\tikzset{reuse path/.code={\pgfsyssoftpath@setcurrentpath{#1}}}
\makeatother

\usepackage{soul}

\begin{document}

	\title[]{Transfer Operators, Canonical Center Dynamics, and Spectral Applications for Long-Range Operators}

\author{Xianzhe Li}
\address{Department of Mathematics, University of California, Berkeley, CA 94720, USA} 
 \email{xianzhe@berkeley.edu}    

\author{Zhenfu Wang}
\address{
	Chern Institute of Mathematics and LPMC, Nankai University, Tianjin 300071, China
}
\email{zhenfuwang@mail.nankai.edu.cn}

 \author{Jiangong You}
 \address{
 Chern Institute of Mathematics and LPMC, Nankai University, Tianjin 300071, China} 
 \email{jyou@nankai.edu.cn}

\author{Qi Zhou}
\address{
	Chern Institute of Mathematics and LPMC, Nankai University, Tianjin 300071, China
}
\email{qizhou@nankai.edu.cn}

\begin{abstract}

We introduce an operator-theoretic framework for long-range operators over general dynamical systems with analytic hopping and small potential. By establishing a partially hyperbolic splitting on the fibered solution bundle, we define the Canonical Center Bundle (CCB) as the center subbundle of this splitting, which is shown to be globally trivial. The center bundle admits a representation via Riesz spectral projections of the transfer operator.
Furthermore, we show that, in the local regime, the center bundle arising in this framework essentially coincides, in the sense of gap convergence, with the Intrinsic Center Bundles (ICB) obtained from finite-range approximations in \cite{GJ}.

The partially hyperbolic structure thereby reduces the spectral problem to the center bundle, leading to a Johnson-type characterization of the spectrum in terms of the associated center cocycle. We then apply this framework to  quasi-periodic Schr\"odinger operators with analytic hopping, large analytic potentials and Diophantine frequency. In this setting, the center cocycle is analytic and satisfies a Center Thouless formula. As consequences, we establish the absolute continuity of the integrated density of states (IDS), resolving a problem of Eliasson; prove quantitative H\"older continuity of the IDS, partially answering a question of You; and obtain Anderson localization for the original Schr\"odinger operators.
\end{abstract}

	\maketitle	

\section{Introduction} \label{main}

\subsection{Background and motivation}

The interplay between finite-dimensional cocycle dynamics and the spectral theory of ergodic operators forms one of the central methods of modern spectral theory \cite{CL,DF,DF2,PF}. For classical nearest-neighbor Schr\"odinger operators, the eigenvalue equation admits an exact reduction to a finite-dimensional cocycle, allowing spectral questions to be translated into dynamical ones.  Long-range operators, describing quantum systems with non-local hopping and arising in effective models of electrons in complex media \cite{AM,Rod,Sar},  occupy a distinguished position in this development. On the one hand, through Aubry duality \cite{AA}, they encode essential spectral information of quasi-periodic Schr\"odinger operators. On the other hand, they raise a new conceptual challenge: unlike nearest-neighbor models, genuinely long-range operators admit no obvious finite-dimensional cocycle representation \cite{GJ}. Identifying the  dynamical structures that govern spectral phenomena in this infinite-dimensional regime therefore gives the key to understanding long-range operators.

We consider the long-range operator with analytic hopping on $\ell^2(\mathbb Z)$
\begin{equation}\label{eq:g}
(\mathbf L_{V, W,T,\theta}u)_n
=\sum_{k\in\mathbb Z}\widehat v_k\,u_{n+k}+\varepsilon W(T^n\theta)\,u_n,\qquad n\in\mathbb Z,
\end{equation}
where $\widehat v_k$ are the Fourier coefficients of a real-analytic function
$V\in C^\omega(\mathbb T,\mathbb R)$, $W\in C^0(\Omega,\mathbb R)$ is a dynamically defined potential over a topological dynamical system $(\Omega,T)$, that is, $\Omega$ is a compact metric space and $T:\Omega\to\Omega$ is a homeomorphism, and  $\theta\in\Omega$ represents the phase. This class of operators provides a unified framework encompassing random, quasi-periodic, and more general ergodic models. If $V(x)=2\cos 2\pi x$, the operator reduces to the classical Schr\"odinger operator.

Originating from the tight-binding description of electrons in a magnetic field \cite{p,harper,R}, one-dimensional quasi-periodic Schr\"odinger operators on $\ell^2(\mathbb Z)$ are defined by
\begin{equation}\label{eq:std_schro}
(\mathbf H_{2\cos,V,\alpha,x}u)_n = u_{n+1}+u_{n-1}+V(x+n\alpha)u_n,\qquad n\in\mathbb Z.
\end{equation}
 Here, $V\in C^\omega(\mathbb T,\mathbb R)$ is the potential, $x\in\mathbb T$ is the phase, and $\alpha\in\mathbb R\backslash\mathbb Q$ is the frequency. A canonical model within this class is the almost Mathieu operator (AMO):
\begin{equation}\label{eq:amo}
(\mathrm H_{\lambda,\alpha,x}u)_n = u_{n+1}+u_{n-1}+2\lambda\cos 2\pi(x+n\alpha)u_n,\qquad n\in\mathbb Z.
\end{equation}
The deterministic disorder induced by incommensurate frequencies governs the underlying spectral and dynamical structures \cite{ALSZ,AJ3,AYZ}. For comprehensive reviews, we refer to \cite{DF,DF2,J,Y}.

A fundamental link between quasi-periodic Schr\"odinger operators and long-range operators is provided by Aubry duality. Applying the Fourier transform to \eqref{eq:std_schro} yields the dual long-range operator
\begin{equation}\label{eq:dual}
(\mathbf L_{V,2\cos ,\alpha,\theta}u)_n=
\sum_{k\in\mathbb Z}\widehat v_k\,u_{n+k}+2\cos 2\pi (\theta+n\alpha)\,u_n.
\end{equation}
Through Aubry duality, long-range operators retain the essential spectral information of the original Schr\"odinger operators. They have therefore emerged as a fundamental tool in the study of quasi-periodic Schr\"odinger operators and have been crucial to a number of major breakthroughs, for instance, the resolution of the Ten Martini Problem \cite{Puig11,AJ3}.

The eigenequation of Schr\"odinger operators admits a reduction to the Schr\"odinger cocycle  $(\alpha,A_E)$ with
$
A_E(x)=
\begin{pmatrix}
E-V(x)&-1\\
1&0
\end{pmatrix},$
this cocycle formulation has served as a foundation for a broad range of fundamental results, both in the zero Lyapunov exponent regime \cite{AJ2,AK,AFK,DS,Eli92} and in the supercritical regime \cite{AJ3,BG00,GS01,GS1,Jit99}, culminating in Avila's global theory of one-frequency analytic cocycles \cite{avila}.  This theory provides a complete dynamical classification of spectral energies into subcritical, critical, and supercritical regimes. Central to the theory is his Almost Reducibility Conjecture (ARC) \cite{avila2010almost,avila-kam}, asserting that every subcritical cocycle is almost reducible (analytically conjugated arbitrarily close to a constant one).

Combined with Aubry duality, ARC led to a fundamental shift in the study of quasi-periodic Schr\"odinger operators. 
Traditionally, reducibility techniques were confined to the zero Lyapunov exponent regime, while the supercritical regime was studied through localization methods. The observation of \cite{AYZ,AYZ2} was that the dual long-range operator provides a bridge between these two regimes: quantitative almost reducibility on the dynamical side can be converted into spectral information for the supercritical Schr\"odinger operator, see also \cite{GY,GYZ22,GYZ23,GYZ24,LYZZ}.

More recently, based on Aubry duality, Ge-Jitomirskaya-You-Zhou \cite{GJYZ} developed a quantitative global theory built upon a multiplicative Jensen formula for the complexified Lyapunov exponent, which also yields new information on the supercritical region \cite{GJY,GJ,LXZ}. For trigonometric polynomial potentials, an alternative approach was  obtained by Han and Schlag \cite{HS4}. In this setting, the dual operator \eqref{eq:dual} induces a finite-dimensional cocycle, so the classical cocycle formalism and reducibility techniques remain available \cite{GJY,GJ,GYZ22,LXZ,WXYZ,WXZ}. A fundamental difficulty appears when the potential is genuinely analytic. In this case, as emphasized in \cite{GJ}, ``the dual operator becomes infinite-range. In this regime there is no finite-dimensional cocycle governing the dynamics, and no such transfer-matrix formalism is available''.
This ``no cocycle'' phenomenon constitutes one of the main obstacles in extending the dynamical approach beyond trigonometric polynomial.

A major breakthrough of Ge-Jitomirskaya \cite{GJ} was the discovery of an intrinsic center dynamics for analytic one-frequency quasi-periodic long-range operators with
$W(\theta)=2\cos(2\pi\theta).$
This provides, for the first time, a finite-dimensional dynamical framework capturing the relevant spectral information of the infinite-range model. 
 More precisely, for a trigonometric polynomial potential $V_l$, the corresponding finite-range operator induces a Frobenius companion cocycle $(\alpha,A^{l}_{\varepsilon}(E,\cdot))$, whose matrix representation has the companion form \eqref{equ:cocycle}.
 Ge-Jitomirskaya-You-Zhou \cite{GJYZ} proved that these cocycles are partially hyperbolic and admit invariant splittings
$
\mathbb C^{2l}=\mathcal E_s^l(\theta)\oplus\mathcal E_c^l(\theta)\oplus\mathcal E_u^l(\theta).
$
After holomorphic trivialization of the center bundle and passage to the limit, Ge-Jitomirskaya \cite{GJ} constructed a finite-dimensional intrinsic center bundle (ICB) from $\mathcal E_c^l(\theta)$ together with the associated intrinsic symplectic center cocycle (ICC). This center structure  has been applied to the spectral theory of Type I operators, in particular to the regularity of the integrated density of states and sharp spectral transitions.

The finite-to-infinite approximation approach of \cite{GJ} is  tied to the one-frequency quasi-periodic setting: it relies essentially on Aubry duality with $W(\theta)=2\cos(2\pi\theta)$, Avila's global theory \cite{avila}, and its quantitative version \cite{GJYZ}. For general hopping functions $W$, or for more general dynamical systems, these ingredients are presently unavailable.

The present work approaches this ``no cocycle'' problem from a geometric and operator-theoretic perspective, guided by the multiplicative Jensen mechanism underlying quantitative global theory \cite{GJYZ,HS4}; see Section~\ref{jensen} for a detailed discussion. For long-range operators over a general dynamical system with small potential $W$, we study the fibered solution bundle directly. The resulting spectral reduction leads naturally to a canonical center structure, which, in the one-frequency local regime, recovers the intrinsic center bundle of \cite{GJ} in the sense of gap convergence (Proposition~\ref{prop:gap}). As we shall see, this viewpoint is particularly well suited for further spectral applications, especially Anderson localization beyond the Type~I regime, and may provide a first step toward a broader global theory for long-range operators.

\subsection{Partial hyperbolicity in the fibered solution bundle }\label{def:norm}

To formulate the main reduction in a base-independent way, we work with a class of function spaces on which the base dynamics acts isometrically. This leads to the following notion.
\begin{definition}\label{def:admissible}
Let $(\Omega,T)$ be a topological dynamical system. A complex unital Banach algebra $\mathfrak A$ of complex-valued functions on $\Omega$ is said to be \emph{admissible for $(\Omega,T)$} if the following hold:
\begin{enumerate}
    \item The translation operator $\mathbf T f:=f\circ T$ acts as an isometric automorphism of $\mathfrak A$, that is,
   $ \|\mathbf T^n f\|_{\mathfrak A}=\|f\|_{\mathfrak A} $, for all $ n\in\mathbb Z
$.
    \item $\mathfrak A$ is continuously embedded into $B(\Omega,\mathbb C)$. Equivalently, there exists a constant $C_{\mathfrak A}>0$ such that
$    \|f\|_\infty:=\sup_{\theta\in\Omega}|f(\theta)|\le C_{\mathfrak A}\|f\|_{\mathfrak A},\qquad f\in\mathfrak A.$
    \item $\mathfrak A$ is inverse-closed: if $f\in\mathfrak A$ and $\inf_{\theta\in\Omega}|f(\theta)|>0$, then $f^{-1}\in\mathfrak A$.
\end{enumerate}
\end{definition}

By replacing $\|\cdot\|_{\mathfrak A}$ with an equivalent norm if necessary, we shall
assume  $C_{\mathfrak A}=1.$  
Throughout the paper, we  will fix the potential  $V(z)=\sum_{n\in\Z} \widehat v_n z^n $ to be analytic in the annulus $\mathbb A_h := \{z\in\mathbb C: e^{-h} < |z| < e^h\}$, via the identification $V(z) \equiv V(x)$ with $z = e^{2\pi i x}$, and then fix a constant $0 < \xi < h$. We define the exponentially weighted sequence space
$$
\mathcal{X}_{\xi}:=\Bigl\{u=(u_n)_{n\in\Z}:\ 
\|u\|_\xi:=\sum_{n\in\Z}|u_n|e^{-\xi|n|}<\infty\Bigr\},
$$
and the corresponding space of (weighted) bounded sections
$$
\mathcal{BS}_\xi := B(\Omega, \mathcal{X}_{\xi}) = \Bigl\{ \Phi:\Omega \to \mathcal{X}_{\xi} : \|\Phi\|_\infty := \sup_{\theta\in\Omega} \|\Phi(\theta)\|_\xi < \infty \Bigr\}.
$$
Equipped with their respective norms, both $\mathcal{X}_{\xi}$ and $\mathcal{BS}_\xi$ are Banach spaces. We define the standard shift operator on sequences as
$
(Su)_n := u_{n+1},  n\in\Z.
$

The full fibered solution space of the infinite-range eigenvalue equation is intrinsically infinite-dimensional. The weighted Banach scale
$\{ \mathcal{X}_\xi \}_{0 < \xi < h}$
should therefore not be viewed as imposing an artificial restriction of the original solution space, but rather as inducing a natural increasing filtration according to exponential growth. As the weight parameter $\xi$ increases, solutions with progressively larger exponential growth are admitted, while the analyticity threshold $h$ of the symbol imposes a fundamental geometric limit on this filtration. Throughout this paper, we study the transfer dynamics on these filtered levels of the fibered solution space:

\begin{definition}[Fibered solution bundle]
\label{def:solution-space}
Given $0<\xi<h$. Consider the trivial Banach bundle $\mathcal E:=\Omega\times\mathcal X_\xi$ and the skew-product map
$\mathscr T(\theta,u):=(T\theta,Su).$
For each $E\in\R$ and $\theta\in\Omega$, define
$$
\mathcal M_{V,E,\varepsilon}(\theta):=
\ker\bigl(\mathbf L_{V,\varepsilon W,T,\theta}-E\bigr)
\cap \mathcal X_\xi.
$$
The corresponding level of the fibered solution bundle is
$$
\boldsymbol{\mathcal M}_{V,E,\varepsilon}:=\bigl\{
(\theta,u)\in\mathcal E:\ u\in\mathcal M_{V,E,\varepsilon}(\theta)
\bigr\}.
$$
The covariance relation
$\mathbf L_{V,\varepsilon W,T,T\theta}\circ S=S\circ \mathbf L_{V,\varepsilon W,T,\theta}$
implies that $\boldsymbol{\mathcal M}_{V,E,\varepsilon}$ is invariant under $\mathscr T$.
\end{definition}

For a fixed energy $E\in\mathbb R$, define the full symbol function:
\begin{equation}\label{equ:symbol}
    L_{E}^{}(z) := V(z) - E,
\end{equation} denote the number of zeros on the unit circle by
\begin{equation}\label{zero}
    2m(E) := \#\{z\in\mathbb A_h: L_{E}^{}(z)=0,\ |z|=1\},
\end{equation}
counting algebraic multiplicities. The number of such center zeros is  even, denoted by $2m(E)$, because the Hermitian symmetry $\hat{v}_n = \overline{\hat{v}_{-n}}$ ensures that $L_{E}^{}(z)$ is a strictly real-valued function on the unit circle.  The quantity $2m(E)$ will play an important role throughout this paper, as it intrinsically determines the dimension of the center bundle in the dynamical partial hyperbolic splitting.

With the above framework, the center reduction can now be stated as follows.

\begin{theorem}
\label{thm:center_reduction}
Let  $\mathfrak A$ be admissible for a topological dynamical system $(\Omega,T)$,  $W\in\mathfrak A$  and  $V \in C^{\omega}(\mathbb{T}, \mathbb{R})$. For every $E \in \mathbb{R}$, there exists $\varepsilon_0=\varepsilon_0(E,V,W) > 0$ such that for all $|\varepsilon| < \varepsilon_0$, the fibered solution bundle admits a  $\mathscr{T}$-invariant splitting:
$$
\mathcal{M}_{V,E,\varepsilon}(\theta)=
\mathcal{M}_{V,E,\varepsilon}^u(\theta)
\oplus\mathcal{M}_{V,E,\varepsilon}^c(\theta)
\oplus\mathcal{M}_{V,E,\varepsilon}^s(\theta),
\qquad \forall \theta \in \Omega.
$$
This splitting is characterized by the following properties:
\begin{enumerate}
\item \emph{Uniform Partial Hyperbolicity:}
  The splitting is $\mathscr T$-invariant and uniformly partially hyperbolic: there exist
$0<r_s<r_-<1<r_+<r_u, $ $C>0$,
such that for all $\theta\in\Omega$ and $n\ge0$,
\begin{equation*}
    \begin{aligned}
        \bigl\| \mathscr{T}^n \big|_{\mathcal{M}_{V,E,\varepsilon}^s(\theta)} \bigr\| &\le C r_s^n, \qquad
        \bigl\| \mathscr{T}^{-n} \big|_{\mathcal{M}_{V,E,\varepsilon}^u(\theta)} \bigr\| \le C r_u^{-n}, \\
        \bigl\| \mathscr{T}^n \big|_{\mathcal{M}_{V,E,\varepsilon}^c(\theta)} \bigr\| &\le C r_+^n, \qquad \bigl\| \mathscr{T}^{-n} \big|_{\mathcal{M}_{V,E,\varepsilon}^c(\theta)} \bigr\| \le C r_-^{-n}.
    \end{aligned}
    \end{equation*}
    \item The \emph{Canonical Center Bundle (CCB)}
    $$
\boldsymbol{\mathcal M}^c_{V,E,\varepsilon}:=
\bigl\{(\theta,u)\in\mathcal E:\ u\in\mathcal M^c_{V,E,\varepsilon}(\theta)\bigr\}
$$
 has constant fiber dimension $2m(E)$ and is globally trivial.
 \item  \emph{{Intrinsic Center Cocycle (ICC)}:} There exists a global frame
$$U^\infty_{\varepsilon}(E,\cdot)=\bigl\{
U^\infty_{\varepsilon,1}(E,\cdot),\dots,
U^\infty_{\varepsilon,2m}(E,\cdot)\bigr\}$$
of $\mathcal X_\xi$-valued sections with $\mathfrak A$-valued coordinates for
$\boldsymbol{\mathcal M}^c_{V,E,\varepsilon}$.
The shift dynamics on the CCB is represented by the induced cocycle
$$
S U^\infty_\varepsilon(E,\theta)=U^\infty_\varepsilon(E,T\theta)C^\infty_\varepsilon(E,\theta),
\qquad \theta\in\Omega.
$$
\end{enumerate}
\end{theorem}

The partially hyperbolic structure in Theorem~\ref{thm:center_reduction}
arises from the spectral decomposition of the associated
Mather transfer operator \eqref{mather-operator}. The analyticity of the hopping function
induces an isolated spectral component near the energy level
$V(x)=E$, whose Riesz spectral projection defines the center bundle.
The resulting center bundle is canonical in the sense that it is
determined uniquely by the transfer operator and the associated
isolated spectral component, independently of any finite-range
approximation procedure. Hence we refer to it as the
\emph{Canonical Center Bundle} (CCB).

\begin{remark}We give further remarks concerning Theorem \ref{thm:center_reduction}.
\begin{enumerate}
\item For an arbitrary topological dynamical system $(\Omega,T)$, the space
$\mathfrak A=C^0(\Omega,\mathbb R)$ is always admissible. Consequently,
Theorem~\ref{thm:center_reduction} yields a globally trivial CCB over an arbitrary compact base, and arbitrary dynamics. 
\item For trigonometric polynomial  potentials $V_l$, the Algebraic Bridge (Proposition~\ref{thm:bridge1}) and the partial hyperbolicity of the Frobenius companion cocycle $(T,A^l_\varepsilon(E,\cdot))$ (Theorem~\ref{thm:sec4-main}) identify the center subspace $\mathcal E^l_{\varepsilon,c}(\theta)$ with $\mathcal M^c_{V_l,E,\varepsilon}(\theta)$ through an isomorphism $\mathcal B(\theta)$. This intertwining relation reads
$$
 S \circ \mathcal B(\theta)=\mathcal B(T\theta)\circ A^l_\varepsilon(E,\theta)\big|_{\mathcal E^l_{\varepsilon,c}(\theta)}.
$$
As a consequence, the CCB constructed in the present framework canonically recovers the ICB of \cite{GJ} in the sense of gap convergence (Proposition~\ref{prop:gap}). 
\item The symplecticity of the center bundle and monotonicity \cite{AK2,LXZ,WXZ} of the center cocycle, along with their spectral applications, are deferred to a subsequent work. 
\item The main theorem remains valid for complex-valued $W$, in which case the associated long-range operator is no longer self-adjoint.  From this perspective, the present result provides a dynamical framework for studying non-self-adjoint long-range (and short-range) operators via  partially hyperbolic structures. It is expected that this viewpoint may be useful for further developments in the spectral and pseudospectral analysis of such operators.
\end{enumerate}
\end{remark}

A particularly important specialization occurs for quasi-periodic
translations
$(\Omega,T)=(\mathbb T^d,\alpha).$
In this setting, the spaces $C^r(\mathbb T^d,\mathbb R)$ and
$C_h^\omega(\mathbb T^d,\mathbb R)$ of analytic functions on a complex strip
are admissible, since translations act isometrically on their natural norms.
Consequently, the regularity of the center bundle matches that of the potential $W$.
In particular, for
$W\in C_h^\omega(\mathbb T^d,\mathbb R)$, the center bundle is
holomorphically trivial, and the associated center cocycle is analytic:

\begin{corollary}\label{cor:icc-torus}
Assume that
$(\Omega,T)=(\mathbb T^d,\alpha)$
and $W\in C^\omega_h(\mathbb T^d,\mathbb R).$  If $|\varepsilon|<\varepsilon_0(E,V,W)$, 
then the canonical center bundle
$\boldsymbol{\mathcal M}^c_{V,E,\varepsilon}$
admits a global holomorphic frame $
U^\infty_\varepsilon(E,\cdot)$
Consequently, the center cocycle $
C^\infty_\varepsilon(E,\cdot)\in C^\omega_h\!\left(
\mathbb T^d,\mathrm{GL}(2m,\mathbb C)\right)$
is well defined by
$$
S U^\infty_\varepsilon(E,\theta)=
U^\infty_\varepsilon(E,\theta+\alpha)C^\infty_\varepsilon(E,\theta).
$$
Moreover,
\begin{equation}\label{loc-kam}
\bigl\|
C^\infty_\varepsilon(E,\cdot)-C^\infty_0(E)\bigr\|_{h}=\mathcal O(|\varepsilon|).
\end{equation}
\end{corollary}

\begin{remark}\label{non-dege}
    An additional advantage of this canonical  center cocycle  is that the eigenvalues of the unperturbed matrix $C^\infty_0(E)$ are the real roots of $V(x)=E.$ Thus it yields the required
non-degeneracy of the unperturbed center dynamics, which is crucial for the
KAM scheme; see Lemma~\ref{lem:G-analytic}.
\end{remark}

The following remark places the above corollary in the broader context of partially hyperbolic dynamics.

\begin{remark}
In classical partially hyperbolic dynamics, the regularity of invariant center bundles is governed by the geometry of the base dynamics and the available bunching conditions \cite{H,HPS}. Under suitable domination assumptions, higher regularity may be obtained via the $C^r$ section theorem, whereas without bunching the center bundle is typically only H\"older continuous.
Theorem~\ref{thm:center_reduction} (resp. Corollary \ref{cor:icc-torus}) addresses a different problem. The starting point is a shift dynamics associated with a long-range operator, for which no finite-dimensional center cocycle is available a priori. The key structural input is the admissible Banach algebra framework: the base transformation acts isometrically on $\mathfrak A$\footnote{For hyperbolic base dynamics, high-regularity function spaces are
typically not admissible.}, so that the transfer operator and its Riesz spectral projections remain in the same function space. As a consequence, the canonical center bundle inherits exactly the regularity of $\mathfrak A$, without loss under the dynamics. In particular, in the multi-frequency case (Corollary \ref{cor:icc-torus}), it yields a globally trivial center bundle in a form that had previously been available only in the one-frequency case \cite{AJS,GJ,LXZ}.
\end{remark}

Although Theorem \ref{thm:center_reduction} establishes the partial hyperbolicity of $\mathscr{T}$, this does not imply that all Lyapunov exponents of the center cocycle $(T, C_{\varepsilon}^{\infty})$ strictly vanish. Nevertheless, the exact number of zero Lyapunov exponents encodes crucial dynamical information. Specifically, we define
\begin{equation}\label{acce}
    2\nu(E) := \#\bigl\{j: L_j(T, C_{\varepsilon}^{\infty}) = 0\bigr\},
\end{equation}
and refer to $2\nu(E)$ as the \emph{Lyapunov neutral dimension}. In the one-frequency case, and restricted to Schr\"odinger operator $\mathbf H_{2\cos,V,\alpha,x}$  the quantity $\nu(E)$ coincides with Avila's acceleration $\omega(E)$ (Lemma \ref{lem:acce}),
which is dynamically more refined than the algebraic multiplicity of the real roots of $V(x)=E.$

\subsection{Extensions of Johnson's Theorem and Thouless formula}

A fundamental consequence of the partially hyperbolic splitting is that the spectral problem can be reduced to the finite-dimensional center dynamics. More precisely, the stable and unstable directions are separated by the partial hyperbolicity, so that the spectral information  is carried by the center part of the dynamics. This provides the mechanism for extending Johnson's classical correspondence \cite{Johnson}---which asserts that an energy lies in the resolvent set if and only if the associated finite-dimensional cocycle is uniformly hyperbolic---to the infinite-range setting. While this equivalence is well understood for general finite-range operators \cite{HaroPuig}, the infinite-range operator precludes a direct transfer-matrix formalism. The following theorem shows that this obstruction is resolved by the underlying partial hyperbolicity.

\begin{theorem}
\label{thm:Johnson_generalization}
Let $(\Omega,T)$ be minimal and
$W\in C^0(\Omega,\mathbb R)$.
For every $E\in\mathbb R$, there exist
$\varepsilon_0=\varepsilon_0(E,V,W)>0$
and
$\delta_1=\delta_1(E,V)>0$
such that the following holds:

If $|\varepsilon|<\varepsilon_0$, then for every
$\zeta\in I_{\delta_1}:=\{\zeta:|\zeta-E|<\delta_1\}$,  $\zeta$ belongs to the spectrum of the infinite-range  operator $\mathbf{L}_{V,\varepsilon W,T,\theta}$
if and only if the center cocycle
$(T,C^\infty_\varepsilon(\zeta,\cdot))$
is not uniformly hyperbolic.
\end{theorem}

The  partially hyperbolic structure of the fibered solution bundle also allows one to extend the classical Thouless formula \cite{AS,HaroPuig} of (generalized) Schr\"odinger operator\begin{equation}\label{eq:long_range_phys}
(\mathbf H_{W,V,\alpha,x}u)_n=
\sum_{k\in\mathbb Z^d}\widehat w_k u_{n+k}+V(x+\langle n,\alpha\rangle)u_n,
\qquad n\in\mathbb Z^d,
\end{equation} which relates the Lyapunov exponents to the integrated density of states, to the long-range setting. In the present setting, this yields a Center Thouless formula:

\begin{proposition}
\label{prop:thouless}
Let $(\Omega,T)=(\mathbb T,\alpha)$ with $\alpha\in \mathbb{R}\backslash\mathbb{Q}$,
$W\in C^\omega(\mathbb T,\mathbb R)$.
For every $E\in\mathbb R$, there exist constants
$\varepsilon_0>0$ and $\delta_1>0$
such that the following holds.

For every $|\varepsilon|<\varepsilon_0$, there exists a harmonic function
$h_\varepsilon$ on $I_{\delta_1}$ satisfying
\begin{equation}\label{eq:thou}
\sum_{i=1}^{m}
L_i\!\left(C^\infty_\varepsilon(\zeta,\theta)\right)
=\int_{I_{\delta_1}}\log |\zeta-t|\, d\mathcal N(t)
+h_\varepsilon(\zeta),
\qquad\zeta\in I_{\delta_1}.
\end{equation}
Consequently, $
\Delta\Bigl(\sum_{i=1}^{m}L_i(C^\infty_\varepsilon(\zeta,\theta))
\Bigr)=2\pi\, d\mathcal N$ 
in the sense of distributions, where
$d\mathcal N$ denotes the density-of-states measure of
$\mathbf H_{\varepsilon W,V,\alpha,x}$.
\end{proposition}
\begin{remark}
    The result also works for $\alpha\in \mathrm{DC}_d$ ($d>1$), while the key is continuity of the Lyapunov exponents \cite{DK}.  
\end{remark}

\begin{remark}
Although the above statements are local in $E$, this is sufficient for our purposes. Indeed, a standard finite covering argument yields the corresponding global result on any prescribed compact energy interval, after possibly shrinking $\varepsilon_0$; see Section~\ref{sec:4.2} for details.
\end{remark}

For each finite-range truncation $V_l$, the classical Thouless formula \cite{HaroPuig}:
$$\sum_{k=1}^{l}L_k(A_{\varepsilon}^{l}(E,\theta))
=
\int_\Sigma \log|E-z|\,d\mathcal{N}_l(z) - \log|\hat{v}_l|$$involves the normalization term $\log |\hat{v}_l|$, which diverges as $l\to\infty$. At the level of the truncated finite-dimensional cocycle, this divergence is exactly compensated by the sum of the Lyapunov exponents corresponding to the  hyperbolic subbundles. Precisely, we have the following asymptotic balance:
\begin{equation}\label{eq:com}
    \lim_{l\to\infty} \bigg( - \log|\hat{v}_l| - \sum_{k=1}^{l-m}L_k(A_{\varepsilon}^{l}(E,\theta)) \bigg) = h_\varepsilon(E).
\end{equation}
In other words, the hyperbolic part of the cocycle carries exactly the exponential growth required to absorb the singular normalization emerging from the infinite-range limit.

\subsection{Main spectral applications}

As a consequence of the geometric reduction developed in this paper, we obtain several spectral applications for supercritical quasi-periodic Schr\"odinger operators. Despite substantial progress over the past two decades, the supercritical regime remains far from being completely understood. As Avila \cite{avila-kam} pointed out:
\begin{quote}
``The supercritical analysis remains incomplete in crucial aspects, especially since many results rely on parameter exclusion arguments in the frequency.''
\end{quote}
The canonical center dynamics developed here provides a source of dynamical rigidity in the supercritical regime. Combined with the near-constant structure of the center cocycle, it enables a direct analysis of the dual long-range dynamics and leads to the following three applications.

\subsubsection{Anderson localization}
Anderson localization (AL)---characterized mathematically by the existence of a complete basis of exponentially decaying eigenfunctions---is a fundamental phenomenon in the theory of condensed matter physics \cite{anderson, ag, h}. For quasi-periodic Schr\"odinger operators, 
despite extensive progress over the past decades, the available localization results
essentially fall into two distinct categories.

The first class of results concerns \emph{fixed phase} localization, where AL is established for a prescribed phase across a generic set of frequencies. Bourgain and Goldstein \cite{BG00} proved that the one-dimensional Schr\"odinger operator $\mathbf{H}_{2 \varepsilon\cos, V, \alpha,x}$  exhibits localization for any fixed phase and a.e.\ frequency, provided the Lyapunov exponent is  positive.  The proof is of non-perturbative nature, that is  the smallness  of the $\varepsilon$ does not depend on the frequency. These results were later extended to the strip setting \cite{BJ1,HS3,Klein} and higher-dimensional setting  \cite{Bou02, Bou07,Liu, JLS}.

The second class of results  addresses \emph{fixed frequency} localization, which requires a prescribed frequency and generic phases. 
For AMO $\mathrm{H}_{\lambda, \alpha, x}$ defined in \eqref{eq:amo}, Jitomirskaya \cite{Jit94, Jit99} developed a non-perturbative framework to prove AL for all Diophantine frequencies and a.e.\ phase. Recall that $\alpha \in \R^d$ is \emph{Diophantine}, denoted $\alpha \in \mathrm{DC}_d$, if there exist $\gamma > 0$ and $\tau > d-1$ such that $\inf_{j\in\Z}|\langle k,\alpha\rangle-j| \ge \gamma |k|^{-\tau}$ for all $k \in \Z^d \setminus \{0\}$. The sharp phase transition of AMO  was first obtained by Avila-You-Zhou \cite{AYZ} and independently by Jitomirskaya-Liu \cite{JL18}. This sharp phase transition was recently extended to  the broader class of Type I operators \cite{GJ}, see also partial results in \cite{hs, HS}. For other nonperturbative localization results, one can consult  \cite{AJ2,AJZ,BJ,JL24,JY,HS4,Liu2,Liu3,Liu4,LY3,LY2} and the references therein.   

Despite these advances, fixed frequency localization for general analytic potentials has remained a long-standing open question:

\begin{problem}\label{prob:1}
For any fixed $\alpha\in \mathrm{DC}_1$, $V\in C^{\omega}(\T,\R)$ is non-constant, and sufficiently small $\varepsilon$, does $\mathbf{H}_{2\varepsilon \cos,V,\alpha,x}$ exhibit Anderson localization for a.e.\ phase $x\in \T$?
\end{problem}

In this paper, we resolve this problem. Furthermore, we extend the localization regime beyond nearest-neighbor models to encompass general analytic hopping terms $W$.

\begin{theorem}\label{mainthe:loca}
Suppose that $\alpha\in \mathrm{DC}_d$, $V\in C^{\omega}(\T,\R)$ is non-constant, and $W\in C^{\omega}(\T^d,\R)$. Then there exists $\varepsilon_1 = \varepsilon_1(\alpha, V, W) > 0$ such that, whenever $|\varepsilon| < \varepsilon_1$, the operator $\mathbf{H}_{\varepsilon W,V,\alpha,x}$ exhibits Anderson localization for a.e.\ $x \in \mathbb{T}$.
\end{theorem}
Theorem \ref{mainthe:loca} establishes a perturbative localization result, where the smallness of $\varepsilon$  depends on the frequency $\alpha$. For this class of perturbative results, prior literature also mainly focused on cosine or cosine-like potentials. Existing approaches generally follow three distinct paths:

\textbf{Multi-Scale Analysis (MSA):} 
Fr\"ohlich, Spencer and Wittwer used MSA to prove Anderson localization for one-dimensional Schr\"odinger operators with even $C^2$ cosine-type potentials \cite{FSW90}. This result was later extended to higher dimensions by Surace \cite{Sur96,Sur90}, and for the arithmetic case by Cao, Shi and Zhang \cite{CSZ23}. The evenness assumption was subsequently removed in one dimension by Forman and VandenBoom \cite{FV21}, and in higher dimensions by Cao, Shi and Zhang \cite{CSZ}.

\textbf{Kolmogorov-Arnold-Moser (KAM) Theory:} Parallel to MSA, Sinai \cite{Sin87} established localization for $C^2$ cosine-type potentials using KAM type analysis. This was generalized to higher dimensions by Chulaevsky and Dinaburg \cite{CD93, D}. For general analytic potentials, Eliasson \cite{Eli97, E} employed KAM techniques to prove the existence of pure point spectrum; however, the exponential decay of the corresponding eigenfunctions was not established.

\textbf{Reducibility and Aubry Duality:} An alternative approach to Anderson localization relies on leveraging Aubry duality in conjunction with the reducibility of the associated dual cocycles. This framework was first  developed by Avila-You-Zhou \cite{AYZ} to establish localization for AMO, and subsequently extended to the high-dimensional AMO by Jitomirskaya and Kachkovskiy \cite{JK}. While this reducibility mechanism has since been adapted to prove arithmetic localization \cite{GY,GYZ24}  and exponential dynamical localization \cite{GYZ23}, these advances remain intrinsically confined to the $\mathrm{SL}(2,\R)$ cocycle setting. A detailed discussion of this  framework is deferred to Section \ref{adl}.

Recently, Cao-Shi-Zhang \cite{CSZ1} utilized a modified MSA framework to provide an affirmative answer to Problem \ref{prob:1} for general analytic potentials,   their method relies on finite-volume approximations, extracting spectral data via multiscale Green's function estimates on increasing lattice blocks.  Our approach is fundamentally different. We adopt a dynamical systems perspective and view Anderson localization as a manifestation of \emph{dynamical rigidity} \cite{Yoccoz}. The key ingredient is the partially hyperbolic structure of the dual long-range dynamics, which provides a canonical decomposition into stable, center, and unstable directions and constrains the behavior of bounded dual solutions. Combined with a rigidity-induced regularity bootstrap, this leads to an independent dynamical proof of Anderson localization. Further discussion of the underlying mechanism is deferred to Section~\ref{adl}.

\subsubsection{Absolute continuity of the IDS}

Our second result establishes the absolute continuity of the integrated density of states (IDS).   Let $\mathcal{N}(E)$ denote the IDS of $\mathbf{H}_{\varepsilon W,V,\alpha,x}$  (see Section~\ref{def:IDS} for the precise definition).

For one-dimensional Schr\"odinger operator $\mathbf{H}_{2\cos,V,\alpha,x}$, 
when the Lyapunov exponent vanishes, Kotani's gem \cite{kotani2} implies that the absolute continuity of the IDS is equivalent to the presence of purely absolutely continuous spectrum for almost every phase. In the positive Lyapunov exponent regime, however, the spectral measures are typically singular. Indeed, for the operator $\mathbf{H}_{\varepsilon W,V,\alpha,x}$ with small $\varepsilon>0$, Eliasson \cite{Eli97, E} proved the existence of pure point spectrum for almost every $x\in\T$. Since the IDS is a phase-averaged quantity, its regularity does not immediately follow from the nature of localized spectral measures. This motivated Eliasson \cite{E} to formulate the following open problem:

\begin{problem}[Eliasson \cite{E}]\label{prob:eliasson}
What is the regularity of the map $E\mapsto \mathcal N(E)$ when $\varepsilon$ is small? Is it singular continuous or absolutely continuous?
\end{problem}

Our second result provides a definitive answer to Problem \ref{prob:eliasson}:

\begin{theorem}\label{thm:IDSac}
Suppose that $\alpha \in \mathrm{DC}_d$, $V\in C^{\omega}(\T,\R)$ is non-constant, and $W\in C^{\omega}(\T^d,\R)$.  Then there exists $\varepsilon_2=\varepsilon_2(\alpha,V,W)>0$ such that for all $0<|\varepsilon|<\varepsilon_2$, the IDS of $\mathbf{H}_{\varepsilon W,V,\alpha,x}$ is absolutely continuous.
\end{theorem}

The absolute continuity of the IDS is a fundamental analytic property. It plays a crucial role for establishing arithmetic phase transitions in Type I operators \cite{GJ}, and for proving the existence of absolutely continuous spectrum for the dual operator of Schr\"odinger operator in $\ell^2(\Z^d)$ \cite{BJ2}. Furthermore, within the present paper, it forms a central ingredient in the proof of Theorem \ref{mainthe:loca}.

Prior to this work, the regularity of the IDS in the positive Lyapunov exponent regime has been extensively studied. For almost every frequency, Goldstein and Schlag \cite{GS1} established its absolute continuity. In the fixed-frequency setting, Avila and Damanik \cite{AD} proved this property for the non-critical almost Mathieu operator, a result subsequently generalized to analytically perturbed non-critical almost Mathieu operators by Ge-Jitomirskaya-Zhao \cite{GJZ}, and to general Type I operators by Ge-Jitomirskaya \cite{GJ}. 

Regarding Eliasson's initial formulation (Problem \ref{prob:eliasson}), Wang-Xu-You-Zhou \cite{WXYZ} recently proved the absolute continuity of the IDS for Diophantine frequencies under the restriction that the potential $V$ is a trigonometric polynomial. Theorem \ref{thm:IDSac} removes this algebraic restriction, resolving Eliasson's problem for general analytic potentials.

\subsubsection{H\"older continuity of the IDS}

Having established the absolute continuity of the IDS, we now turn to its finer regularity. A natural next question is whether the IDS is H\"older continuous and, more specifically, what determines the H\"older exponent. This question is closely related to other spectral properties of quasi-periodic operators, including homogeneity of the spectrum \cite{DGSV,LYZZ}.

For one-dimensional analytic quasi-periodic Schr\"odinger operators $\mathbf{H}_{2\varepsilon\cos,V,\alpha,x}$, the regularity theory is relatively well understood. In the zero Lyapunov exponent regime and $\alpha \in \mathrm{DC}_d$, Amor \cite{Amor} proved $\mathcal N(E)$ is $1/2$-H\"older continuous in the perturbative setting (the smallness of the potential depend on $\alpha$), and Avila and Jitomirskaya \cite{AJ2} extended this to the non-perturbative setting. The exponent $1/2$ is optimal because of the square-root singularity at gap edges \cite{Puig1}. In the positive Lyapunov exponent regime and $\alpha \in \mathrm{DC}_1$,  Bourgain \cite{Bou00} proved a $(1/2-\epsilon)$-H\"older estimate for the almost Mathieu operator \eqref{eq:amo} with large coupling, while Goldstein and Schlag \cite{GS01} established H\"older continuity for general analytic potentials and strong Diophantine frequency, and later \cite{GS1} proved that 
if $V(x)$ is a small perturbation of a trigonometric
polynomial of degree $k_0$, then $\mathcal{N}(E)$ is $1/(2k_0)-\epsilon$ H\"older continuous for any $\epsilon>0$.
 You \cite{You} conjectured that the H\"older exponent is $1/(2\omega(E))$, where $\omega(E)$ is the acceleration (see Section~\ref{acceleration}); this dynamical invariant is central to Avila's global theory \cite{avila}. The conjecture has been verified for Type I operators \cite{GJ}, and lower bounds of the form $1/(2\omega(E))-\epsilon$ have been obtained for general analytic potentials and strip models \cite{HS,HS1}.

For higher-dimensional or long-range models $\mathbf{H}_{\varepsilon W,V,\alpha,x}$, the picture is less complete. Ge-You-Zhao \cite{GYZ22} proved H\"older continuity of the IDS for large trigonometric polynomial potentials with Diophantine frequencies, with an exponent controlled by the cardinality of the level set $\#\{x:V(x)=E\}$. Cao-Shi-Zhang \cite{CSZ} later extended this result to 
 large analytic potentials. 
 Thus, unlike the one-frequency theory, the available higher-dimensional bounds are still formulated in terms of geometric counting quantities rather than a dynamical invariant.

Our next result gives such a dynamical description. It shows that the H\"older exponent is governed by the Lyapunov neutral dimension $\nu(E)$ defined in \eqref{acce}.

\begin{theorem}\label{thm:IDSholder}
Suppose that $\alpha\in \mathrm{DC}_d$, $V\in C^{\omega}(\T,\R)$ is non-constant, and $W\in C^{\omega}(\T^d,\R)$.  Then for every $\epsilon>0$, there exists $\varepsilon_3=\varepsilon_3(\alpha,V,W,\epsilon)>0$ such that, whenever $|\varepsilon|<\varepsilon_3$, the IDS of $\mathbf{H}_{\varepsilon W,V,\alpha,x}$ is H\"older continuous. More precisely, for every fixed $E\in\mathbb R$, there exists a constant $C_1=C_1(\alpha,W,V,\epsilon,E)>0$ such that
$$
|\mathcal N(E)-\mathcal N(E')|
\le C_1 |E-E'|^{\frac{1}{2\nu(E)}-\epsilon}.
$$
\end{theorem}

Theorem~\ref{thm:IDSholder} replaces prior bounds based on the algebraic roots of $V(x)=E$ with the dynamical invariant $\nu(E)$. In the one-frequency setting, as shown in  Lemma \ref{lem:acce}, $\nu(E)$ coincides with Avila's acceleration $\omega(E)$, which is obviously not larger than the number of algebraic roots of $V(x)=E$. Consequently, Theorem~\ref{thm:IDSholder} provides a partial resolution to You's conjecture \cite{You} for general analytic large potentials, and reveals that this  dynamical mechanism may govern IDS regularity across arbitrary dimensions.

Here we restrict ourselves to H\"older continuity results under analytic potentials. The (non‑)H\"older continuity of the IDS for Liouvillean frequencies has been studied in \cite{ALSZ,HS2,HZ,LY,YZ}; for H\"older continuity with smooth potentials, see \cite{CSZ,CSZ23,LWY,XGW}.

\section{Strategy of the Proofs}

\subsection{A new perspective} 
Our viewpoint is that the problem of extracting a center dynamics from an infinite-dimensional fibered solution space is, at its core, a problem of isolating invariant sections. In this sense, the operator-theoretic framework of Mather \cite{Mather1968} (see also \cite{HaroDeLaLlave,LS90} and the referene therein) is the natural setting: rather than approximating the long-range operator by finite-dimensional truncations, one studies directly the associated transfer operator acting on a weighted space of sections.

\subsubsection{Multiplicative Jensen mechanism}\label{jensen}
A key observation is that analyticity of  $V_l$ enforces a separation mechanism in the unperturbed setting. Indeed, the roots of $L_E^{l}(z)=V_l(z)-E$ are isolated in the analytic annulus; the corresponding Lyapunov exponents are given by the logarithms of their moduli, with the roots on $|z|=1$ giving rise to the center directions. This is precisely the Jensen mechanism underlying the unperturbed problem.

A useful way to understand this reduction is through the multiplicative Jensen mechanism from quantitative global theory~\cite{GJYZ,HS4}.
 For finite truncations $H_n(\cdot)$ of Schr\"odinger operator, the characteristic determinants 
$f_n(E,z)=\det\bigl(H_n(z)-E\bigr)$
encode the zero distribution under Dirichlet boundary conditions, these zeros are asymptotically organized on concentric circles determined by the turning points of the complexified Lyapunov exponent \cite{GS1,HS,WWYZ}.  
Our perspective is that this mechanism should be understood at the level of the transfer operator itself. In the unperturbed setting, the roots of $L_E(z)=V(z)-E$ determine the spectrum  of the transfer operator. After perturbation, although the algebraic description is lost, the spectral separation persists and becomes a genuine spectral-gap property of the transfer operator on the weighted space $\mathcal{BS}_\xi$. By Aubry duality, these turning points of the complexified
Lyapunov exponent correspond to changes in the dominant
spectral contributions of the dual long-range operator.

This spectral gap yields uniform resolvent bounds and makes it possible to define Riesz spectral projections. The range of the center projection gives the Canonical Center Bundle (CCB), and the induced splitting produces the partially hyperbolic structure of the fibered solution bundle. In this way, the transfer-operator approach provides an infinite-dimensional mechanism for center extraction, distinct from the finite-dimensional approximation scheme of \cite{GJ}.

\subsubsection{Local nature}

In the one-frequency quasi-periodic setting considered in \cite{GJ}, the ICC is expected to lead to genuine global spectral results once the ARC for higher-dimensional cocycles is settled \cite{You}. Pending the resolution of the ARC, a KAM-type argument serves as a natural approach to local behavior, requiring dynamics that remain close to a non-degenerate constant model. Moreover, in multi-frequency settings, usually one can at best expect local perturbative conclusions \cite{Bou02a,Bou02b}. However, the local structure of ICC is not clear from the finite-range approximation \cite{GJ}. 

Observe that for small $\varepsilon$, although $\mathbf L_{V_l,\varepsilon W,\alpha,\theta}$ is local in the operator-theoretic sense, the associated Frobenius companion cocycle is non-local from a dynamical viewpoint; rather, it provides a semi-global encoding \cite{YZh} of the entire $2l$-dimensional state. As a result, the classical finite-to-infinite approximation scheme is inherently incompatible with uniform perturbative arguments. Indeed,$$A_{\varepsilon}^{l}(E,\cdot)=A_{0}^{l}(E)+O\!\left(\frac{\varepsilon W}{\hat v_l}\right),$$so the perturbation scale is effectively weighted by $|\hat v_l|^{-1}$. In particular, the smallness condition required for fixed-$l$ perturbation theory degenerates as $l\to\infty$. A second, independent obstruction arises from the companion representation itself: the matrices $A_0^l(E)$ are highly non-normal, and both their norms and the associated resolvent bounds may deteriorate with the truncation order. Specifically, $\|(z-A_0^l(E))^{-1}\|_1$ may grow linearly or even exponentially as $l\to\infty$ (Remark~\ref{abnormal}). Moreover, since reducibility and KAM estimates depend quantitatively on the cocycle norm and its dimension, the effective perturbative threshold inevitably degenerates along the approximation sequence.

From the transfer-operator viewpoint, the center dynamics arises from the Riesz spectral projection associated with an isolated spectral cluster. Since the resolvent is represented through the full symbol $D_{E,\varepsilon}(z)$ in \eqref{full-sym}, perturbative control of the symbol yields perturbative control of the spectral projection. Consequently, the center bundle varies continuously from the unperturbed one, and the resulting center cocycle remains close to the constant model determined by the roots of $V(z)-E$. 

\subsubsection{Further perspective on global theory}

Although the results of the present paper are formulated in a local (albeit non-perturbative) regime, the geometric reduction developed in Section~\ref{sec:ambient_framework} reveals a more fundamental operator-theoretic principle. Partial hyperbolicity of the fibered solution bundle ultimately depends only on the uniform invertibility of the full symbol operator $D_{E,\varepsilon}(z)$ (see \eqref{full-sym}) on the relevant spectral annuli. The local assumption serves merely as a convenient sufficient condition, guaranteeing invertibility through a Neumann-series argument (Lemma~\ref{thm:D-abstract}).
This suggests that the transfer-operator perspective may provide a natural framework for extending quantitative global theory. We hope to pursue this direction in future work.

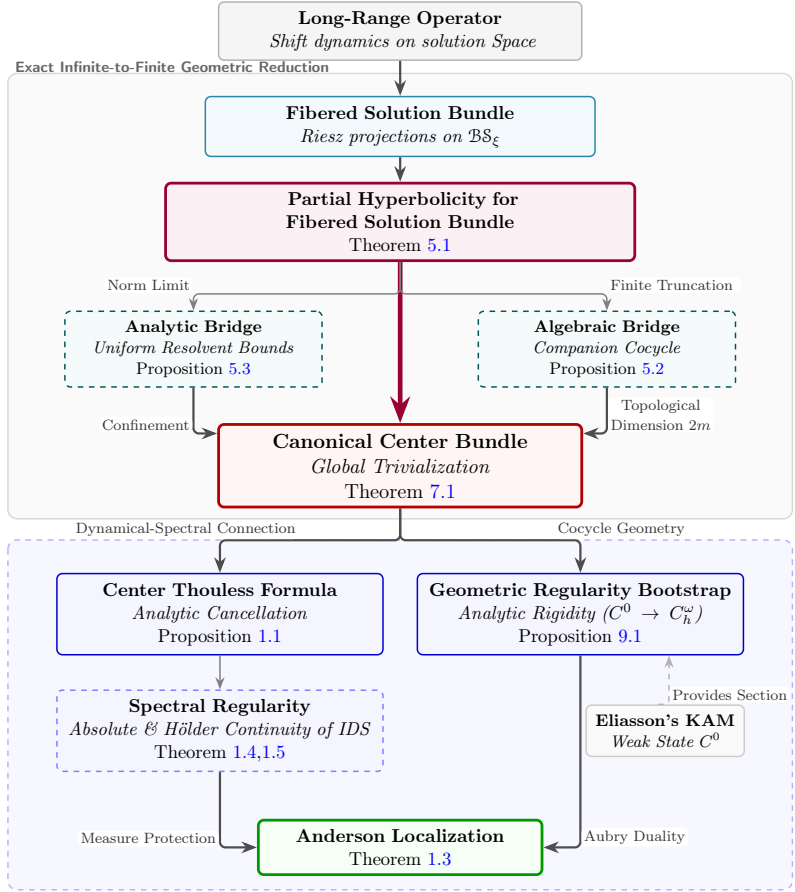
\begin{figure}[htbp]
\centering
\resizebox{0.72\textwidth}{!}{
\begin{tikzpicture}[
    >=Stealth,
    base/.style={rectangle, rounded corners=3pt, thick, align=center, font=\small, inner sep=5pt},
    source/.style={base, fill=gray!8, draw=gray!60, text width=6.5cm},
    fib/.style={base, fill=cyan!4, draw=cyan!60!black, text width=7cm},
    hyper/.style={base, fill=purple!6, draw=purple!80!black, line width=1.5pt, text width=8.5cm},
    bridge/.style={base, fill=teal!4, draw=teal!70!black, dashed, text width=4.5cm, font=\footnotesize},
    icc/.style={base, fill=red!4, draw=red!70!black, line width=1.5pt, text width=6.5cm, font=\normalsize},
    engine/.style={base, fill=blue!4, draw=blue!80!black, text width=5.8cm},
    extbox/.style={base, fill=gray!4, draw=gray!40, font=\footnotesize, align=center},
    result/.style={base, fill=green!4, draw=green!60!black, line width=1.5pt, text width=5cm},
    main_arrow/.style={->, line width=1.2pt, draw=black!70},
    core_arrow/.style={->, line width=2.5pt, draw=purple!80!black},
    sub_arrow/.style={->, thick, draw=black!50},
    dash_arrow/.style={->, dashed, thick, draw=gray!60},
    lbl/.style={fill=white, inner sep=1.5pt, font=\scriptsize, text=black!80, align=center}
]

\filldraw[fill=gray!4, draw=gray!30, thick, rounded corners=5pt] (-7.4, -0.8) rectangle (7.4, -9.2);
\node[anchor=south west, text=gray!80!black, font=\scriptsize\bfseries\sffamily] at (-7.4, -0.9) {Exact Infinite-to-Finite Geometric Reduction};

\node[source] (op) at (0, 0) {\textbf{Long-Range Operator}\\ \textit{ Shift dynamics on solution Space }};

\node[fib] (fib) at (0, -1.8) {\textbf{Fibered Solution Bundle}\\ \textit{Riesz projections on $\mathcal{BS}_\xi$}};

\node[hyper] (hyper) at (0, -3.6) {\textbf{ Partial Hyperbolicity for Fibered Solution Bundle}\\ Theorem \ref{thm:sec7final}};

\node[bridge] (ana_bridge) at (-3.9, -6.0) {\textbf{Analytic Bridge}\\ \textit{Uniform Resolvent Bounds}\\ Proposition \ref{thm:ambient-projection}};
\node[bridge] (alg_bridge) at (3.9, -6.0) {\textbf{Algebraic Bridge}\\ \textit{Companion Cocycle}\\ Proposition \ref{thm:bridge1}};

\node[icc] (icc) at (0, -8.2) {\textbf{Canonical Center Bundle}\\ \textit{Global Trivialization} \\Theorem \ref{thm:centert}};

\filldraw[fill=blue!3, draw=blue!30, dashed, thick, rounded corners=5pt] (-7.4, -9.6) rectangle (7.4, -16.2);

\node[extbox] (eliasson) at (5.0, -13.2) {\textbf{Eliasson's KAM}\\ \textit{Weak State $C^0$}};

\node[engine] (thouless) at (-3.4, -11.0) {\textbf{Center Thouless Formula}\\ \textit{Analytic Cancellation}\\ Proposition \ref{prop:thouless}};
\node[engine] (bootstrap) at (3.4, -11.0) {\textbf{Geometric Regularity Bootstrap}\\ \textit{Analytic Rigidity ($C^0 \to C^\omega_h$)}\\ Proposition \ref{prop:rigidity}};

\node[engine, draw=blue!50, dashed] (ids) at (-3.4, -13.2) {\textbf{Spectral Regularity}\\ \textit{Absolute \& H\"{o}lder Continuity of IDS}\\ Theorem \ref{thm:IDSac},\ref{thm:IDSholder}};

\node[result] (al) at (0, -15.4) {\textbf{Anderson Localization}\\ Theorem \ref{mainthe:loca}};

\draw[main_arrow] (op) -- (fib);
\draw[main_arrow] (fib) --  (hyper);

\draw[core_arrow] (hyper.south) -- (icc.north);

\draw[sub_arrow, rounded corners=3pt] (hyper.south) -- ++(0,-0.6) -| node[lbl, above left] {Norm Limit} (ana_bridge.north);
\draw[sub_arrow, rounded corners=3pt] (hyper.south) -- ++(0,-0.6) -| node[lbl, above right] {Finite Truncation } (alg_bridge.north);

\draw[main_arrow, rounded corners=3pt] (ana_bridge.south) |- node[lbl, above left] {Confinement} (icc.170);
\draw[main_arrow, rounded corners=3pt] (alg_bridge.south) |- node[lbl, above right] {Topological\\Dimension $2m$} (icc.10);

\draw[main_arrow, rounded corners=3pt] (icc.south) -- ++(0,-0.6) -| node[lbl, above left, xshift=+1.5cm, yshift=0cm] {Dynamical-Spectral Connection} (thouless.north);
\draw[main_arrow, rounded corners=3pt] (icc.south) -- ++(0,-0.6) -| node[lbl, above right, xshift=-0.5cm, yshift=0cm] {Cocycle Geometry} (bootstrap.north);

\draw[sub_arrow] (thouless) -- (ids);
\draw[dash_arrow, rounded corners=3pt] (eliasson.105) -| node[lbl, right, pos=0.6] {Provides Section} (bootstrap.-25);

\draw[main_arrow, rounded corners=3pt] (ids.south) |- node[lbl, above left] {Measure Protection} (al.west);
\draw[main_arrow, rounded corners=3pt] (bootstrap.south) |- node[lbl, above right] {Aubry Duality} (al.east);

\end{tikzpicture}
}
\vspace{0.2cm}
\caption{Logical structure of the proof}
\label{fig:logic_architecture}
\end{figure}

 The logical structure of the proof is summarized in Figure~\ref{fig:logic_architecture}. The subsequent development proceeds along two complementary paths, spectral and dynamical, whose interaction ultimately leads to Anderson localization. Now let us state the main difficulty and novelty in the key steps. 

\subsection{Canonical Center Bundle}
The first step toward the construction of the Canonical Center Bundle is to establish partial hyperbolicity. The transfer-operator framework developed above yields such a structure simultaneously for the fibered solution bundle and for the companion cocycles $A^l_{\varepsilon}(E,\theta)$. These two center structures will later be connected by the Algebraic Bridge.
\subsubsection{Partially hyperbolicity}
For the companion cocycle $A^l_{\varepsilon}(E,\theta)$, the unperturbed case $\varepsilon=0$ is constant, and the splitting is determined explicitly by the zero distribution of $L_E^l(z)$. Hence partial hyperbolicity is immediate when $\varepsilon=0$. For $\varepsilon\neq 0$, a standard cone argument gives persistence, but the admissible perturbation size depends on the matrix dimension $2l$, thus one cannot obtain uniform small $\varepsilon$, this makes the finite-dimensional argument non-uniform as $l\to\infty$. To obtain a uniform result, we work instead with the transfer operator.
The main difficulty is to prove uniform invertibility for the truncated transfer operator. Two obstructions must be handled simultaneously: the apparent norm growth caused by the non-normal companion matrix, and the diverging coefficient in the rank-one truncation error. We remove both by a weighted gauge transformation, motivated by the dimension-free Aubry duality developed in \cite[Section~7]{LXZ}. Here, the weights suppress the unperturbed off-diagonal growth into summable series, restoring the resolvent boundedness (Lemma \ref{prop:weighted}). Crucially, the weighted resolvent's directional response  cancels the diverging perturbation coefficient. This  algebraic cancellation, detailed in Proposition \ref{prop:improved} (see also Remark \ref{rem:mechanism}), reduces the singular truncation error to a uniformly bounded perturbation, securing the infinite-range limit.

For the shift dynamics on the fibered solution bundle, one works on the exponentially weighted space $\mathcal X_\xi$, which accommodates the generalized solutions of the infinity range operator (for instance, when $z_0$ is a multiple
root of $L_E(z)=V(z)-E$, the unperturbed operator admits solutions of the form
$n^jz_0^n$).
The center bundle is extracted through the symbol calculus: uniform invertibility of the operator-valued symbol on zero-free annuli yields resolvent bounds and a well-defined Riesz projection (Proposition \ref{thm:D-regular}), while the zeros of $L_E$ on $|z|=1$ determine the center dimension.
This mechanism is similar in spirit to the use of anisotropic Banach spaces in the study of Ruelle-Pollicott resonances \cite{Rulle,Pollicott} for Anosov maps \cite{Baladi,BKL,Faure}, where the choice of norm makes the spectral features of the transfer operator accessible. The two settings are nevertheless different: in the Anosov setting, the adapted spaces are designed to exploit
the hyperbolicity of maps on finite-dimensional manifolds,
whereas here the weighted space is used to study the shift dynamics on solution space. In the present setting, the weighted space $\mathcal X_\xi$ provides the resolvent bounds needed to define the relevant Riesz projections and recover the Lyapunov exponents and Oseledets splitting similar to the finite dimensional cocycle.

\subsubsection{Two bridge framework }

Although the center space $\mathcal M^c_{V,E,\varepsilon}$ is defined intrinsically by the Riesz projection, its fiber dimension is not directly accessible from that construction. The main difficulty is that the center fibers for different truncation scales $l$ do not live in a common space, and their geometry varies with $l$, $\varepsilon$, and $\theta$. 

To overcome this obstruction, we use a two-bridge argument. The Algebraic Bridge (Proposition \ref{thm:bridge1}) identifies, in the finite-range case, the center fiber with the center subspace of the companion cocycle, and hence reduces the dimension count to the zero set of the unperturbed symbol $L_E(z)$. The Analytic Bridge (Proposition \ref{thm:ambient-projection}) then lifts the Riesz projections to the fixed Banach space $\mathcal{BS}_\xi$, where the finite-range and infinite-range center objects can be compared by uniform resolvent estimates. This yields a stable way to propagate the finite-range dimension $2m(E)$ to the infinite-range limit.

The constant fiber dimension of $\boldsymbol{\mathcal M}^c_{V,E,\varepsilon}$, which  endows it with the geometric structure of a vector bundle, is established via finite-range trigonometric polynomial approximations. This derivation follows a standard reduction scheme in global analysis. The point is to recover an intrinsic geometric object from a sequence of non-intrinsic finite-dimensional models. In the limit $l\to\infty$, the coordinate dependence of the companion truncations disappears, and the resulting center bundle becomes a dynamically invariant spectral object.

\subsubsection{Global trivialization}
The mere existence of a finite-dimensional center bundle is insufficient for dynamical analysis; we require a global holomorphic frame to trivialize it.  In the one-frequency case, as $\mathbb T^1$ is Stein, the Oka--Grauert principle \cite{Forstneric2017} yields a holomorphic trivialization of the  bundles \cite{GJ,LXZ}. However not to mention the general base $\Omega,$ even over higher-dimensional tours $\mathbb{T}^d$ ($d \ge 2$), the existence of such global frames is generally obstructed by non-vanishing Chern classes \cite{DH}. In our setting, however, the unperturbed equation provides an explicit global frame for the center directions. This frame can be continued to
the perturbed problem through the resolvent operator, yielding a global holomorphic frame $U^\infty_\varepsilon$ for
$\mathcal M^c_{V,E,\varepsilon}(\theta)$.
Consequently, the skew product on the center bundle is coordinated into a finite-dimensional analytic cocycle.

We explicitly construct the center frame using algebraically weighted contour integrals (which is close related to Riesz projections, c.f. Lemma \ref{recons} and Proposition \ref{thm:ambient-projection}). This specific formulation intertwines the shift dynamics exactly with the basis index,  enforcing a companion matrix structure on the reduced cocycle (Theorem \ref{thm:centert}). To guarantee that this infinite-dimensional frame does not degenerate as the truncation range $l \to \infty$, we reduce linear dependence to a meromorphic vanishing problem, which is then ruled out by Hermite interpolation and polynomial degree considerations (Lemma \ref{lem:lindep-uniform}). This algebraic rigidity secures a uniformly positive Gram matrix (Lemma \ref{prop:extended}), i.e. uniformly bounded
left inverse of the holomorphic frame, ensuring the metric stability of the center cocycle in the full infinite-range limit (Theorem \ref{thm:centert}).

\subsection{The Center Thouless formula}

The second major step is the Center Thouless formula (Proposition~\ref{prop:thouless}), which relates the center cocycle to the integrated density of states (IDS). In the classical Schr\"odinger operators \eqref{eq:std_schro}, the Thouless formula rests on the relation between the transfer matrix and the characteristic polynomial \cite{AS}. For long-range operators, this relation is no longer available. Finite-range truncations lead to companion matrices with strongly non-normal behavior, and their norms diverge with the truncation scale, so a direct passage to the limit is not possible.

We overcome this difficulty by working with the center cocycle. Thanks to the partially hyperbolicity of companion cocycle (Theorem \ref{thm:sec4-main}) and the Algebraic Bridge (Proposition \ref{thm:bridge1}),   the sum of the center Lyapunov exponents defines a sequence of subharmonic functions on the complexified energy plane. 
The finite-range Thouless formula
decomposes it into the logarithmic potential of the IDS and a harmonic
correction. The logarithmic potential part passes to the limit by the weak
convergence of the IDS. The remaining harmonic corrections are locally bounded
from below; hence, by Harnack compactness, they admit a locally uniform
harmonic limit \eqref{eq:com}.  Thus \eqref{eq:thou} gives the center Thouless formula and identifies the center Lyapunov exponents as the spectral quantities governing the IDS.
As a consequence, one obtains the absolute continuity (Theorem \ref{thm:IDSac}) and H\"older continuity (Theorem \ref{thm:IDSholder}) of the IDS, which are the spectral inputs needed later for the localization argument via Aubry duality.

\subsection{Anderson localization}\label{adl}
Traditional reducibility approaches to Anderson localization rely heavily on the $\mathrm{SL}(2,\mathbb{R})$ structure, where a single fibered rotation number  dictates the localized phase via Aubry duality. For long-range operators with trigonometric polynomial potentials, the dual cocycle evolves in the higher-dimensional Hermitian-symplectic group $\mathrm{HSp}(2l)$\footnote{$\mathrm{HSp}(2l) = \{ A \in \mathbb{C}^{2l \times 2l} \mid A^* J A = J \}$, where $J$ is the standard symplectic matrix.}\cite{HaroPuig}. 
While the special case $\mathrm{HSp}(2)$ has recently been resolved by Ge and Jitomirskaya \cite{GJ} through a modified duality-based scheme (with absolutely continuity of IDS as the key input), the general higher-dimensional Hermitian-symplectic setting remains substantially more difficult.
In higher dimensions, the classical rotation number generalizes to the mean Maslov index \cite{LW}. As an averaged quantity, it is insufficient to uniquely determine the localized phase. Moreover, the existence of dynamically defined rotation vectors in this setting remains a wide-open problem.

As we mentioned, our method is dynamical. 
We begin with a phase $x$ admitting an
$\ell^1(\mathbb Z)$ eigenfunction\footnote{Although \cite{Eli97,E} state the
result in terms of $\ell^2(\mathbb Z^d)$ eigenfunctions, the proof in \cite{Eli97,E} already imply the $\ell^1(\mathbb Z^d)$ property used here; for a
detailed proof, see \cite[Theorem 10.2]{WXZ}.}, a property guaranteed for almost every phase
by Eliasson's framework \cite{Eli97,E}. Under Aubry duality, this eigenfunction
translates into a continuous section
$\widehat u\in C^0(\mathbb T^d,\C)$ in the dual long-range operator
$\mathbf L_{V,\varepsilon W,\alpha,\theta}$.
By the partially hyperbolic structure of solution space (Theorem \ref{thm:center_reduction}), the solution is confined to the center bundle. Passing to a global holomorphic frame reduces the dynamics to a $2m$-dimensional center cocycle. On this finite-dimensional model, analytic reducibility  of center cocycle forces the rotational phases of the diagonalized cocycle to
rigidly align with $x$, which yields a regularity bootstrap (Proposition \ref{prop:rigidity}) upgrading $C^0$ coefficients to holomorphic ones. 

To translate this geometric regularity back to the original operator, we couple the bootstrap mechanism with the measure-theoretic protection provided by the absolute continuity of IDS (Theorem \ref{thm:IDSac}). This absolute continuity rigorously isolates the pure point spectrum from the zero-measure reducible exceptional set of the center cocycle (Theorem \ref{main:reduce}), where the structure of canonical center cocycle plays an important role (consult Remark \ref{non-dege} for details).  As a result, the analytic rigidity of the center bundle is  preserved on the  eigenvalues. Via Aubry duality, the analytic regularity of the dual state inherently enforces  exponential decay of the  eigenfunctions. This proves Anderson localization
in Theorem \ref{mainthe:loca}, without requiring finite
truncations or infinite-dimensional localized perturbative estimates.

Note in the trigonometric polynomial case, the preceding dynamical rigidity argument was already carried out \cite{WYZ}, to do this, the center cocycles are already obtained by partially hyperbolicity of   companion  cocycle $(\alpha,A^{l}_{\varepsilon}(E,\cdot))$ \cite{GJYZ}.
At first sight, one might try to obtain Anderson localization for analytic potential by combining localization for  trigonometric polynomial with convergence of the associated center cocycles as in \cite{GJ}. Such a strategy would require a uniform mechanism for tracking the relevant eigenvalue branches and eigenfunctions along the truncation sequence.
The point is that the mere existence of a limiting ICC does not provide such a mechanism. Even if the finite-range center cocycles converge, there is in general no canonical way to identify which branch $E_k^{l}(x)$, or which eigenfunction $u^{l,k}$, should be followed as $l\to\infty$. Consequently, one cannot expect localization bounds of the form
$
|u_n^{l,k}(x)| \le C_x e^{-h|n+k|}
$
with constants independent of both $l$ and $k$, and the limiting procedure does not by itself produce a complete localized basis for the infinite-range operator. Arbitrary choices may drift among different branches; after
normalization, one may even have
$u^{l_j,k_j}(x)\rightarrow 0$. 
The two-dimensional center case of \cite{GJ} is special because the ICC carries a fibered rotation number (or rotation pairs defined in \cite{GJ}), which supplies a rigid phase-selection rule, as first observed in \cite{AYZ}. In that setting, the relation $\rho(E)=x+\langle k,\alpha\rangle$
canonically selects the relevant spectral branches and allows one to propagate localization from finite-range approximants to the limit. No comparable invariant is available in higher-dimensional center dynamics.

Thus the essential obstruction here is not the convergence of the center cocycles, but the lack of a canonical branch-selection mechanism. This exactly motivates our partially hyperbolic for solution space approach, where the spectral information is extracted directly from the full transfer operator rather than from a limiting family of finite-range approximants.

\section{Preliminaries}

\subsection{Notations}
For the reader's convenience, the principal notation used throughout this paper is summarized in the following table.
\begin{longtable}{
    >{\raggedright\arraybackslash}p{0.18\textwidth}
    >{\raggedright\arraybackslash}p{0.54\textwidth}
    >{\raggedright\arraybackslash}p{0.20\textwidth}
}

\caption{Summary of Notation} \label{tab:notation} \\
\toprule
\textbf{Symbol} & \textbf{Meaning} & \textbf{Reference} \\
\midrule
\endfirsthead

\caption[]{Summary of Notation (continued)} \\
\toprule
\textbf{Symbol} & \textbf{Meaning} & \textbf{Reference} \\
\midrule
\endhead

\midrule
\multicolumn{3}{r}{\footnotesize Continued on next page} \\
\endfoot

\bottomrule
\endlastfoot

$\mathbf{H}_{\varepsilon W,V,\alpha,x}$ & (Generalized) Schr\"odinger operator on $\Z^d$ & \eqref{eq:long_range_phys} \\
$\mathbf{L}_{V,\varepsilon W,T,\theta}$ & Long-range operator on $\Z$ &  \eqref{eq:g}\\
$(\Omega,T)$ & topological dynamical system  \\
$\mathfrak{B}(\mathcal{H})$ & Bounded linear operator from space $\mathcal{H}$ to $\mathcal{H}$ & \\
$\mathfrak A$ & Banach algebra of functions space & Definition \ref{def:admissible} \\
$\mathcal{X}_{\xi}$ & Weighted sequence Banach space & Section~\ref{def:norm} \\
$\mathcal{BS_\xi}$ & Banach space of weighted bounded sections & Section~\ref{def:norm} \\
$\|\cdot\|_\xi$ & Weighted norm on $\mathcal{X}_{\xi}$ & Section~\ref{def:norm} \\
$\|\cdot\|_\infty$ & Sup norm on $\mathcal{BS_\xi}$ & Section~\ref{def:norm} \\
$S$ & Shift operator on sequences & Section~\ref{def:norm} \\
$\boldsymbol{\mathcal M}_{V,E,\varepsilon}$ & Bounded solution space  & Section~\ref{def:norm} \\
$\mathcal M_{V,E,\varepsilon}(\theta)$ & Fiber of the bounded solution space & Definition \ref{def:solution-space} \\
$\mathcal M^{s,c,u}_{V,E,\varepsilon}(\theta)$ & Stable, Center, Unstable subspace & Theorem~\ref{thm:center_reduction} \\
$\mathbb{P}^{l}_{\varepsilon,c}$ & Extended center projection  & Equation \eqref{eq:ambient-P-def} \\
$L_{E}(z)$ & Infinite-range symbol $V(z)-E$ & Equation \eqref{equ:symbol} \\
$L_E^{l}(z)$ & Truncated symbol & Equation \eqref{equ:tsymbol} \\
$m(E)$ & Number of zeros of $L_{E}$ on $|z|=1$ & Equation \eqref{zero} \\
$r_-(E),\,r_+(E)$ & Radii defining the center annulus & Section~3 \\
$\eta_-(E),\,\eta_+(E)$ & Width parameters for the separating annuli & Section~3 \\
$\mathbb A_h$ & Analytic annulus $\{e^{-h}<|z|<e^h\}$ & \\
$\mathcal A_-(E),\mathcal A_+(E)$ & Annuli around $r_-(E)$ and $r_+(E)$ & Section~3 \\
$\mathcal A(E)$ & Union of the separating annuli & Section~3 \\
$\mathcal A'(E)$ & Smaller annuli & Section~3 \\
$\delta_0$ & Positive lower bound of $|L_{E}|$ on $\mathcal A(E)$ & Equation \eqref{delta0} \\
$r_s(E),\,r_u(E)$ & Inner and outer radii & Section~3 \\
$\Gamma_s,\,\Gamma_u$ & Inner and outer contours & Section~3 \\
$\Gamma_-,\,\Gamma_+$ & Boundary circles of the center annulus & Section~3 \\
$\Gamma_c$ & Center contour & Section~3 \\
$I_\delta$ & Complex neighborhood of interval $I$ & Equation \eqref{complexneig} \\
$\|f\|_{\mathfrak A,\delta}$ & Norm on space $\mathfrak A\times I_\delta$ & Equation \eqref{xdeltanorm} \\
$A^{l}_\varepsilon(\zeta,\theta)$ & Finite-dimensional companion cocycle  & Equation \eqref{equ:cocycle} \\
$U^{l}_\varepsilon(\zeta,\theta)$ & Holomorphic frame of the center bundle & Equation \eqref{equ:basis} \\
$u^{l,p}_\varepsilon(\zeta,\theta)$ & $p$-th vector in the holomorphic center frame & Equation \eqref{equ:basis1} \\
$C^{l}_\varepsilon(\zeta,\theta)$ &  Center Cocycle & Equation \eqref{eq:defAc} \\
$\mathcal N(E)$ & Integrated density of states & Equation \eqref{sids} \\

\end{longtable}

\subsection{Complex  cocycles}\label{com}
Let $T:\Omega\to \Omega$ be homeomorphism,  $A\in C^0(\Omega,{\rm M_m}(\C))$, $\nu$ be a $T$-invariant ergodic probability measure on $\Omega$, a cocycle $(T, A)$ is a linear skew product:
$$
(T,A)\colon \left\{
\begin{array}{rcl}
	\Omega\times \C^{m} &\to& \Omega \times \C^{m}\\
	(\theta,v) &\mapsto& (T\theta,A(\theta)v)
\end{array}
\right.  .
$$
For $n\in\mathbb{Z}$, $A_n$ is defined by $(T,A)^n=(T^n,A_n).$ Thus $A_{0}(\theta)=id$,
\begin{equation*}
	A_{n}(\theta)=\prod_{j=n-1}^{0}A(T^{j}\theta)=A(T^{n-1}\theta)\cdots A(T\theta)A(\theta),\ for\ n\ge1,
\end{equation*}
and $A_{-n}(\theta)=A_{n}(T^{-n}\theta)^{-1}$.

We denote by $L_1(A)\geq L_2(A)\geq...\geq L_m(A)$ the Lyapunov exponents of $(T,A)$ repeated according to their multiplicities, i.e.,
$$
L_k(A)=\lim\limits_{n\rightarrow\infty}\frac{1}{n}\int_{\Omega}\ln(\sigma_k(A_n(\theta)))d\nu,
$$
where for any matrix $B\in {\rm M_m}(\C)$, we denote by
$\sigma_1(B)\geq...\geq \sigma_m(B)$ its singular values (eigenvalues
of $\sqrt{B^*B}$).  Since the k-th exterior product $\Lambda^k A_n$ satisfies $\sigma_1(\Lambda^k A_n)=\|\Lambda^k A_n\|$, $L^k(A)=\sum_{j=1}^kL_j(A)$ satisfies
$$
L^k(A)=\lim\limits_{n\rightarrow \infty}\frac{1}{n}\int_{\Omega}\ln\|\Lambda^kA_n(\theta)\|d\nu.
$$

 Recall that
for complex cocycles $(T,A)\in B(\Omega,{\rm M_m}(\C))$, Oseledets theorem provides us with  strictly decreasing
sequence of Lyapunov exponents $L_j \in [-\infty,\infty)$ of multiplicity $m_j$, $1\leq j \leq \ell$ with $\sum_{j}m_j=m$, and for $a.e.$ $\theta$, there exists
a measurable invariant decomposition $$\C^m=E^{1}(\theta)\oplus E^2(\theta)\oplus\cdots\oplus E^{\ell}(\theta)$$ with $\dim E^j(\theta)=m_j$ for $1\leq j\leq \ell$ such that $$
\lim\limits_{n\rightarrow\infty}\frac{1}{n}\ln\|A_n(\theta)v\|=L_j,\ \  \forall v\in E^j(\theta)\backslash\{0\}.
$$
An invariant decomposition $\C^m=E^{1}(\theta)\oplus E^2(\theta)\oplus\cdots\oplus E^{\ell}{(\theta)}$  is   dominated if for any unit vector $v_j\in E^j(\theta)\backslash \{0\}$, we have $$\|A_n(\theta)v_j\|>\|A_n(\theta)v_{j+1}\|.$$
In this paper, we focus on the case that the associated cocycle is {\it partially hyperbolic}, which means that there is an invariant dominated splitting $E^u\oplus E^c\oplus E^s$ of $\C^{2m}$ everywhere, and  there exist  some constants $C>0,c>0$, and for every $n\geqslant 0$,
$$
\begin{aligned}
	\| A_n(\theta)v\| \leqslant Ce^{-cn}\| v\|, \quad & v\in E^s(\theta),\\
	\| A_n(\theta)^{-1}v\| \leqslant Ce^{-cn}\|v\|,  \quad & v\in E^u(T^n\theta).
\end{aligned}
$$

In this paper, we shall use the cocycle induced by the finite-range eigenvalue equation
$$ (\mathbf L_{V_l,\varepsilon W,T,\theta}u)_n = \sum_{|k|\le l}\widehat v_k u_{n+k} + \varepsilon W(T^n\theta)u_n = Eu_n, \qquad n\in\mathbb Z. $$ The associated $2l$-dimensional cocycle is given by
\begin{equation}\label{equ:cocycle}
    A_{\varepsilon}^{l}(E,\theta)=\begin{pmatrix} -\frac{\hat v_{l-1}}{\hat v_l} & \cdots & -\frac{\hat v_{1}}{\hat v_l} & \frac{E-\hat v_0-\varepsilon W(\theta)}{\hat v_l} & -\frac{\hat v_{-1}}{\hat v_l} & \cdots & -\frac{\hat v_{-\ell+1}}{\hat v_l} & -\frac{\hat v_{-l}}{\hat v_l} \\ 1 &&&&&&&
\\& \ddots &&&&&&
\\&& 1 &&&&&
\\&&& 1&&&&
\\&&&& 1&&&
\\&&&&& \ddots&&
\\&&&&&& 1& 0
\end{pmatrix}.
\end{equation}
We call $(T,A^{l}_{\varepsilon}(E,\cdot))$ a Frobenius companion cocycle.
\subsection{Global theory of quasi-periodic cocycles.}\label{acceleration}
We give a brief overview of Avila's global theory for one-frequency quasi-periodic $\mathrm{SL}(2,\mathbb{R})$ cocycles \cite{avila}. Let $\alpha \in \mathbb{R} \setminus \mathbb{Q}$ and $A \in C^\omega_h(\mathbb{T}, \mathrm{SL}(2,\mathbb{R}))$. For $|y| < h$, define $A_y \in C^\omega(\mathbb{T}, \mathrm{SL}(2,\mathbb{R}))$ by
$A_y(\cdot) = A(\cdot + \mathrm{i} y)$. Then, the \textit{acceleration} of the cocycle $(\alpha, A)$ is defined as
\begin{equation}\label{defacc}
    \omega(\alpha, A) = \lim_{y \to 0^{+}} \frac{L(\alpha, A_{ y}) - L(\alpha, A)}{2\pi y}.
\end{equation}
		The key ingredient to the global theory is that the acceleration is quantized:
\begin{theorem}[\cite{avila}\label{acce1}]
	Suppose $(\alpha,A)\in \mathbb{R}\setminus\mathbb{Q} \times  C^{\omega}(\mathbb{T},\mathrm{SL}(2,\R))$. Then $\omega(\alpha,A)\in\mathbb{Z}$. 
\end{theorem}

\subsection{Classical Aubry Duality}
Suppose $ \mathbf{L}_{V,\varepsilon W,\alpha,\theta} $ has a solution $ u \in \ell^1(\mathbb{Z}) $ for some $ E $, define $ \hat{u}(x) = \sum_{n\in\mathbb{Z}} u_n e^{2\pi i n x} $. For every $ x\in\mathbb{T} $,
$$
\tilde{u}(n) = \hat{u}(x + \langle n, \alpha\rangle) e^{2\pi i \langle n, \theta\rangle}, \quad n\in\mathbb{Z}^d,
$$
is a solution of $ \mathbf{H}_{\varepsilon W,V,\alpha,x} $.

Conversely, suppose $ \mathbf{H}_{\varepsilon W,V,\alpha,x} $ has a solution $ u \in \ell^1(\mathbb{Z}^d) $, define $ \hat{u}(\theta) = \sum_{n\in\mathbb{Z}^d} u_n e^{i\langle n, \theta\rangle} $. Then for every $ \theta\in\mathbb{T}^d $,
$\tilde{u}(n) = \hat{u}(\theta + n\alpha) e^{2\pi i n x},  n\in\mathbb{Z},
$
is a solution of $ \mathbf{L}_{V,\varepsilon W,\alpha,\theta} $. 

\subsection{Integrated Density of States}\label{def:IDS}
The density of state, which can be interpreted as the average spectral measures of an ergodic family of self-adjoint operators $ \{ \mathbf{H}_{\varepsilon W,V,\alpha,x}\}_{x\in\T}  $ over $ x$:
\begin{equation}\label{sids}
    \mathcal{N}(E) = \int_{\mathbb{T}} \mu_{x,\delta_p}(-\infty, E] dx,\qquad \text{ for all } p
\end{equation}
where $ \mu_{x,\delta_p} $ denotes the associated spectral measure of $ \mathbf{H}_{\varepsilon W,V,\alpha,x} $ corresponding to the vector $\delta_p$.
We can also define the density of state of operator $ \{ \mathbf{L}_{V,\varepsilon W,\alpha,\theta} \}_{\theta \in \mathbb{T}^d}$
$$\widehat{\mathcal{N}}(E) =\int_{ \mathbb{T}^d} \mu_{\theta,\delta_p}(-\infty, E] d\theta,\qquad \text{ for all } p$$
where $ \mu_{\theta,\delta_p} $ denotes the associated spectral measure of $  \mathbf{L}_{V,\varepsilon W,\alpha,\theta} $ corresponding to the vector $\delta_p$.
We have the following relation of $ \mathcal{N}(E)$ and $\widehat{\mathcal{N}}(E) $.

\begin{proposition}\label{IDSeq}\cite{GJLS,Puig1}
    $ \mathcal{N}(E)=\widehat{\mathcal{N}}(E) $.
\end{proposition}

\section{The Algebraic Seed and Spectral Gap}
\label{sec:algebraic-seed}

\subsection{Distribution of Zeros of analytic functions}

The fundamental assumption that $V(z)$ is analytic in $\mathbb A_h$ plays a crucial role here. Analyticity ensures that the zeros of $L_{E}^{}(z)$ are discrete and isolated. Thanks to this isolation, any compact sub-annulus contains only finitely many zeros. As illustrated in Figure \ref{fig:zero-distribution}, this allows us to rigorously define a spectral gap. We can always choose radii
$$
e^{-h} < r_-(E) < 1 < r_+(E) < e^h
$$
such that the two solid orange circles $|z|=r_-(E)$ and $|z|=r_+(E)$ are strictly zero-free. These contours perfectly isolate the $2m(E)$ ``center zeros'' (depicted as red points in Figure \ref{fig:zero-distribution}) in the unit circle from the rest of the domain. Consequently, the closed annulus $\{z\in\mathbb C: r_-(E) \le |z| \le r_+(E)\}$ contains no other zeros of $L_{E}^{}$ besides those on $|z|=1$. In particular, we have
$$
\#\{z\in\mathbb{A}_h: L_{E}^{}(z)=0,\ r_-(E) < |z| < r_+(E)\} = 2m(E),
$$
counting multiplicities.

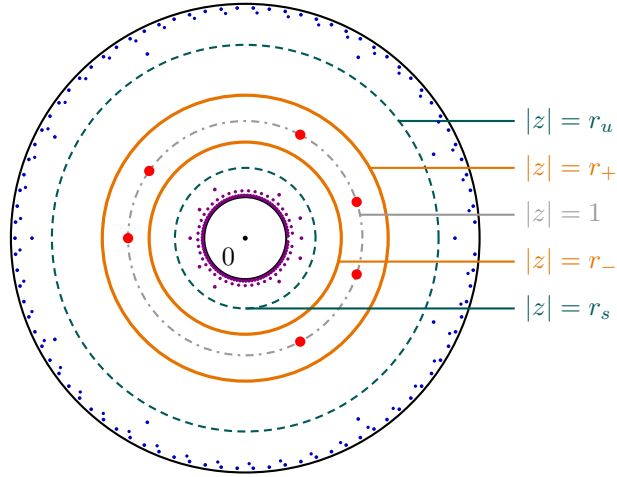
\begin{figure}[htbp]
\centering
\begin{tikzpicture}[scale=1.55,>=Latex,every node/.style={font=\small}]

\def\Rout{2.0}   
\def\Ru{1.65}    
\def\Rp{1.22}    
\def\Rone{1.0}   
\def\Rm{0.82}    
\def\Rs{0.6}    
\def\Rin{0.35}  

\draw[thick] (0,0) circle (\Rout);
\draw[thick] (0,0) circle (\Rin);

\draw[densely dashed, thick, teal!70!black] (0,0) circle (\Ru);
\draw[densely dashed, thick, teal!70!black] (0,0) circle (\Rs);

\draw[very thick, orange!90!black] (0,0) circle (\Rp);
\draw[very thick, orange!90!black] (0,0) circle (\Rm);

\draw[dash pattern=on 3pt off 2pt on 1pt off 2pt, thick, gray!80] (0,0) circle (\Rone);

\draw[thick,orange!90!black] (1.05,0.6)--(2.3,0.6);
\node[right,orange!90!black] at (2.3,0.6) {$|z|=r_+$};
\draw[thick,orange!90!black] (0.8,-0.2)--(2.3,-0.2);
\node[right,orange!90!black] at (2.3,-0.2) {$|z|=r_-$};
\draw[thick,teal!70!black] (1.3,1.0)--(2.3,1.0);
\node[right,teal!70!black] at (2.3,1.0) {$|z|=r_u$};
\draw[thick,teal!70!black] (0,-0.6)--(2.3,-0.6);
\node[right,teal!70!black] at (2.3,-0.6) {$|z|=r_s$};
\draw[thick,gray!80] (0.98,0.2)--(2.3,0.2);
\node[right,gray!80] at (2.3,0.2) {$|z|=1$};

\node at ({(\Rout+0.12)*cos(48)},{(\Rout+0.12)*sin(48)}) {};
\node at ({(\Rin-0.11)*cos(18)},{(\Rin-0.11)*sin(18)}) {};

\fill[red] ({1.00*cos(18)},{1.00*sin(18)}) circle (1.25pt);
\fill[red] ({1.00*cos(-18)},{1.00*sin(-18)}) circle (1.25pt);

\fill[red] ({1.00*cos(62)},{1.00*sin(62)}) circle (1.25pt);
\fill[red] ({1.00*cos(-62)},{1.00*sin(-62)}) circle (1.25pt);

\fill[red] ({1.00*cos(145)},{1.00*sin(145)}) circle (1.25pt);

\fill[red] (-1.00,0) circle (1.25pt);

\foreach \ang in {4,8,12,16,20,24,28,32,36,40,44,48,52,56,60,64,68,72,76,80,84,88,
                  92,96,100,104,108,112,116,120,124,128,132,136,140,144,148,152,156,
                  160,164,168,172,176,180,184,188,192,196,200,204,208,212,216,220,224,
                  228,232,236,240,244,248,252,256,260,264,268,272,276,280,284,288,292,
                  296,300,304,308,312,316,320,324,328,332,336,340,344,348,352,356}
{
  \fill[blue!75!black] ({1.965*cos(\ang)},{1.965*sin(\ang)}) circle (0.45pt);
}

\foreach \ang in {7,15,23,31,39,47,55,63,71,79,87,95,103,111,119,127,135,143,151,159,
                  167,175,183,191,199,207,215,223,231,239,247,255,263,271,279,287,295,
                  303,311,319,327,335,343,351}
{
  \fill[blue!75!black] ({1.925*cos(\ang)},{1.925*sin(\ang)}) circle (0.40pt);
}

\foreach \ang/\rad in {
  24/1.76,
  58/1.82,
  118/1.78,
  152/1.84
}{
  \fill[blue!75!black] ({\rad*cos(\ang)},{\rad*sin(\ang)}) circle (0.45pt);
  \fill[blue!75!black] ({\rad*cos(-\ang)},{\rad*sin(-\ang)}) circle (0.45pt);
}
\fill[blue!75!black] ({1.79*cos(0)},{1.79*sin(0)}) circle (0.45pt);
\fill[blue!75!black] ({1.82*cos(180)},{1.82*sin(180)}) circle (0.45pt);

\foreach \ang in {3,7,11,15,19,23,27,31,35,39,43,47,51,55,59,63,67,71,75,79,83,87,
                  91,95,99,103,107,111,115,119,123,127,131,135,139,143,147,151,155,
                  159,163,167,171,175,179,183,187,191,195,199,203,207,211,215,219,223,
                  227,231,235,239,243,247,251,255,259,263,267,271,275,279,283,287,291,
                  295,299,303,307,311,315,319,323,327,331,335,339,343,347,351,355}
{
  \fill[violet] ({0.365*cos(\ang)},{0.365*sin(\ang)}) circle (0.45pt);
}

\foreach \ang in {6,14,22,30,38,46,54,62,70,78,86,94,102,110,118,126,134,142,150,158,
                  166,174,182,190,198,206,214,222,230,238,246,254,262,270,278,286,294,
                  302,310,318,326,334,342,350}
{
  \fill[violet] ({0.405*cos(\ang)},{0.405*sin(\ang)}) circle (0.40pt);
}

\foreach \ang/\rad in {
  20/0.50,
  52/0.47,
  122/0.49,
  148/0.46
}{
  \fill[violet] ({\rad*cos(\ang)},{\rad*sin(\ang)}) circle (0.45pt);
  \fill[violet] ({\rad*cos(-\ang)},{\rad*sin(-\ang)}) circle (0.45pt);
}
\fill[violet] ({0.48*cos(0)},{0.48*sin(0)}) circle (0.45pt);
\fill[violet] ({0.47*cos(180)},{0.47*sin(180)}) circle (0.45pt);

\fill (0,0) circle (0.6pt);
\node[below left] at (0,0) {$0$};

\end{tikzpicture}
\caption{ Zero distribution and the isolating contour}
\label{fig:zero-distribution}
\end{figure}

Furthermore, we can choose $\eta_-(E)>0$ and $\eta_+(E)>0$ small enough such that the closed  annuli around the contours, defined by
$$
\mathcal A_-(E) := \{z\in\mathbb C: r_-(E)e^{-\eta_-(E)} \le |z| \le r_-(E)e^{\eta_-(E)}\},
$$
\begin{equation*}
    \mathcal A_+(E) := \{z\in\mathbb C: r_+(E)e^{-\eta_+(E)} \le |z| \le r_+(E)e^{\eta_+(E)}\},
\end{equation*}
do not intersect the zero set of $L_{E}^{}(z)$. Denoting $\mathcal{A}(E) = \mathcal A_-(E) \cup \mathcal A_+(E)$, the quantity
\begin{equation}\label{delta}
    \delta(E) := \inf_{z\in \mathcal A(E)} |L_{E}^{}(z)|>0
\end{equation}
is well-defined and strictly positive. This uniform positive lower bound quantifies the spectral gap needed in the subsequent construction. And we denote 
$$
\mathcal A'_-(E) := \{z\in\mathbb C: r_-(E)e^{-\eta_-(E)/2} \le |z| \le r_-(E)e^{\eta_-(E)/2}\},
$$
\begin{equation}\label{eq:a+}
    \mathcal A'_+(E) := \{z\in\mathbb C: r_+(E)e^{-\eta_+(E)/2} \le |z| \le r_+(E)e^{\eta_+(E)/2}\},
\end{equation}
and $$\mathcal{A}'(E)=\mathcal A'_-(E)\cup\mathcal A'_+(E).$$

Later, when we need auxiliary contours separating the center region from the inner and outer parts, we may choose any radii $r_s(E)$ and $r_u(E)$ within zero-free regions such that:
$$
r_s(E)\in \bigl(r_-(E)e^{-\eta_-(E)/2},\,r_-(E)\bigr),
\qquad
r_u(E)\in \bigl(r_+(E),\,r_+(E)e^{\eta_+(E)/2}\bigr).
$$
Then, as shown by the dashed teal circles in Figure \ref{fig:zero-distribution}, we have the strict ordering
\begin{equation*}
    r_s(E) < r_-(E) < 1 < r_+(E) < r_u(E),
\end{equation*}
and $L_{E}^{}(z)$ naturally has no zeros on $|z|=r_s(E)$ and $|z|=r_u(E)$.  Moreover, we consider the zeros of the truncated analytic Laurent series   \begin{equation}\label{equ:tsymbol}
    L_{E}^{l}(z) := \sum_{|n|\le l} \hat{v}_n z^n - E.
\end{equation}
Roughly speaking, by classical Jentzsch--Szeg\H{o} type theorems  \cite{Jen,PPS} for the partial sums of analytic functions, if $V(z)$ converges exactly in the maximal annulus $\mathbb A_{h}$, then every point on the boundary circles $|z|=e^{-h}$ and $|z|=e^{h}$ is an accumulation point of the zeros of $L_{E}^{l}(z)$ as $l \to \infty$.    Therefore these contours separate the remaining infinite sea of zeros into an inner family (violet points) and an outer family (blue points).
We define the corresponding positively oriented circles as:
\begin{equation}\label{contour}
\begin{aligned}
    &\Gamma_s := \{z\in\mathbb C: |z|=r_s\}, \quad
    &&\Gamma_u := \{z\in\mathbb C: |z|=r_u\}, \\
    &\Gamma_- := \{z\in\mathbb C: |z|=r_-\}, \quad
    &&\Gamma_+ := \{z\in\mathbb C: |z|=r_+\}, \quad\Gamma_c := \Gamma_+ \cup \{-\Gamma_-\}. \quad
\end{aligned}
\end{equation}

\begin{remark}
Were $V(z)$ merely smooth rather than analytic, its zero set could be pathologically distributed, which would destroy any possibility of establishing a spectral gap. 
\end{remark}

\subsection{Distribution of Zeros for the Truncated Symbol}\label{sec:4.2}
The following elementary lemma demonstrates that, once the center annulus is isolated by the zero-free boundary circles, the number of zeros enclosed within it is completely stable under local perturbations of energy (Lemma \ref{lem:locals}) and   sufficiently large truncations (Lemma~\ref{prop:assume}) . This stability is a direct and powerful consequence of the argument principle.

\begin{lemma}\label{lem:locals}
Fix $E_0\in\mathbb R$. Then there exist radii $e^{-h} < r_-(E_0) < 1 < r_+(E_0) < e^h$, numbers $\eta_-(E_0)>0, \eta_+(E_0)>0$, a constant $\delta_{E_0}>0$, and an open interval $J_{E_0}\subset\mathbb R$ containing $E_0$, such that for every $E\in \overline{J_{E_0}}$:
\begin{enumerate}
    \item We have the uniform lower bound 
    $
    |L_{E}^{}(z)|\ge \delta_{E_0}
    $
    for all $z \in \mathcal{A}(E_0).$ 
    \item  $p(E):=\#\{z\in\mathbb A_h: L_{E}^{}(z)=0,\ r_-(E_0) < |z| < r_+(E_0)\} = 2m(E_0).
$\end{enumerate}
\end{lemma}

\begin{proof}
(1) As $V(\cdot)$ is analytic, by the geometric construction above  applied at $E = E_0$, there exist zero-free annuli $\mathcal{A}(E_0)$ and a constant $\delta(E_0) > 0$ such that $|L_{E_0}^{}(z)| \ge \delta(E_0)$ for all $z \in \mathcal{A}(E_0)$. 
Set $\delta_{E_0} := \delta(E_0)/2$, and define $J_{E_0} := (E_0 - \delta_{E_0}, E_0 + \delta_{E_0})$. For any perturbed energy $E \in \overline{J_{E_0}}$ and any $z \in \mathcal{A}(E_0)$, we have the algebraic identity $L_{E}^{}(z) = L_{E_0}^{}(z) - (E - E_0)$. By the  triangle inequality,
$$
|L_{E}^{}(z)| \ge |L_{E_0}^{}(z)| - |E-E_0| \ge 2\delta_{E_0} - \delta_{E_0} = \delta_{E_0}.
$$
This establishes the uniform lower bound on the annuli for all $E \in \overline{J_{E_0}}$.

(2) Since  (1) guarantees that $L_{E}^{}(z)$ is strictly non-vanishing on the boundary contours $\Gamma_\pm := \{z : |z|=r_\pm(E_0)\}$ for all $E \in \overline{J_{E_0}}$, we can apply the argument principle. The number of enclosed zeros is given by the contour integral:
\begin{equation*}
   p(E) = \frac{1}{2\pi i}\oint_{\Gamma_+} \frac{V'(z)}{V(z)-E} dz - \frac{1}{2\pi i}\oint_{\Gamma_-} \frac{V'(z)}{V(z)-E} dz.
\end{equation*}
Because the denominator $V(z)-E$ never vanishes on the contours, the integrand is jointly continuous in $(E, z) \in \overline{J_{E_0}} \times (\Gamma_+ \cup \Gamma_-)$. Consequently, the integral $p(E)$ depends continuously on the parameter $E$. Since a zero-counting function must be integer-valued, it is  constant on the connected interval $\overline{J_{E_0}}$. Thus, $p(E) = p(E_0) = 2m(E_0)$.
\end{proof}

\begin{lemma}\label{prop:assume}
Let $I\subset\mathbb R$ be a compact interval. Then there exist finitely many open intervals $J_1,\dots,J_k$ covering $I$, radii $e^{-h}<r_-^{(j)}<1<r_+^{(j)}<e^h$, numbers $\eta_-^{(j)}>0,\eta_+^{(j)}>0$, integers $P_j\ge 0$ for $1\leq j\leq k$, and constant $\delta_0>0$,   integer $l_0\ge 1$ such that the following hold for every $j\in\{1,\dots,k\}$:

\begin{enumerate}
    \item For every $E\in \overline{J_j}$, $|L_{E}^{}(z)|\ge \delta_0$ for all $z$ satisfying 
    $$
    r_-^{(j)}e^{-\eta_-^{(j)}}\le |z|\le r_-^{(j)}e^{\eta_-^{(j)}},\quad r_+^{(j)}e^{-\eta_+^{(j)}}\le |z|\le r_+^{(j)}e^{\eta_+^{(j)}}.
    $$
    \item For every $E\in \overline{J_j}$ and every $l\in [l_0,\infty]$, $|L^{l}_{E}(z)|\ge \frac{\delta_0}{2}$ for all $z$ satisfying 
    $$
    r_-^{(j)}e^{-\eta_-^{(j)}}\le |z|\le r_-^{(j)}e^{\eta_-^{(j)}}, \quad r_+^{(j)}e^{-\eta_+^{(j)}}\le |z|\le r_+^{(j)}e^{\eta_+^{(j)}}.
    $$
    \item For every $E\in \overline{J_j}$ and every $l\in [l_0,\infty]$,
    $$
    \#\{z:\ r_-^{(j)}<|z|<r_+^{(j)},\ L^{l}_{E}(z)=0\}=P_j,
    $$
    counting multiplicities.
\end{enumerate}
\end{lemma}

\begin{proof}
(1) By Lemma~\ref{lem:locals}, every $E_0 \in I$ admits an open neighborhood $J_{E_0}$ and associated geometric parameters $(r_\pm(E_0), \eta_\pm(E_0), \delta_{E_0})$ such that $|L_{E}^{}(z)| \ge \delta_{E_0}$ for $z\in \mathcal A(E_0)$ and $E \in \overline{J_{E_0}}$. 
Since $I$ is compact, we can extract a finite subcover $J_1, \dots, J_k$ corresponding to energy centers $E_1, \dots, E_k$. For each $1 \le j \le k$, we set $r_\pm^{(j)} := r_\pm(E_j)$, $\eta_\pm^{(j)} := \eta_\pm(E_j)$, and define the uniform lower bound $\delta_0 := \min_{1 \le j \le k}\delta_{E_j}$. 

(2) Fix $j \in \{1,\dots,k\}$. The truncation error $L_{E}^{}(z) - L^{l}_E(z) = \sum_{|n|>l} \hat{v}_n z^n$ is  independent of $E$ and converges to zero uniformly on compact subsets of $\mathbb{A}_h$. Thus, there exists an integer $l_j \ge 1$ such that for all $l \in[l_j,\infty]$, 
$
\sup_{z \in \mathcal{A}_j} |L_{E}^{}(z) - L^{l}_{E}(z)| < \frac{\delta_0}{4},
$
where $\mathcal{A}_j$ denotes the union of the two annuli defined in (1). 
Then (2) follows  from $(1)$.

(3) Let $D_j := \{z \in \mathbb{C} : r_-^{(j)} < |z| < r_+^{(j)}\}$. For every $E \in \overline{J_j}$ and $l \in [l_j,\infty]$, we have established the strict inequality
$$
|L_{E}^{}(z) - L^{l}_{E}(z)| < \frac{\delta_0}{4} < \delta_0 \le |L_{E}^{}(z)|
$$
everywhere on the boundary contours $\partial D_j = \{|z|=r_+^{(j)}\} \cup \{|z|=r_-^{(j)}\}$. By Rouché's Theorem, $L^{l}_{E}(z)$ and $L_{E}^{}(z)$  have  the same number of zeros inside $D_j$. According to Lemma~\ref{lem:locals}(2), $L_{E}^{}(z)$ has a constant number of enclosed zeros for all $E \in \overline{J_j}$, which we denote by $P_j := 2m(E_j)$. Hence, $L^{l}_{E}(z)$ also possesses exactly $P_j$ zeros inside $D_j$, proving (3).
Finally, taking $l_0 := \max_{1\le j \le k} l_j$ ensures that all three properties hold simultaneously for every $j \in \{1, \dots, k\}$ and all $l\in[l_0,\infty]$.
\end{proof}

Note by Lemma~\ref{prop:assume}, the original compact interval can be covered by finitely many open intervals, on each of which all the required spectral separation properties hold uniformly. Since the dynamical and geometric arguments in the sequel are local in energy, it is sufficient to prove them on one such interval at a time. 
Accordingly, throughout the rest of this paper we fix one interval from this finite cover and still denote it by $I$. For notational convenience, although we continue to write quantities such as $m(E)$, $\mathcal A'(E)$, $r_\pm(E)$, and $\eta_\pm(E)$, on the present interval $I$ these objects are understood to be chosen uniformly and hence may be regarded as fixed with respect to $E\in I$. In particular, we will set 
\begin{equation}\label{delta0}
 \inf_{E\in I}\inf_{z\in \mathcal A'(E)}|L_{E}^{}(z)|\ge \delta_0, \quad\inf_{E\in I}\inf_{z\in \mathcal A'(E)}|L_{E}^{l}(z)|\ge \delta_0/2,
\end{equation}
for all $l\in[l_0,\infty)$ where $\delta_0$, $l_0$ is defined in Lemma~\ref{prop:assume}.
 With this convention, all subsequent arguments are carried out uniformly on this fixed interval $I$; the global statements on the original compact interval will then follow via a standard finite covering argument.

\subsection{Symbol Operator}

Having established the non-vanishing of the scalar symbols $L_E^{}(z), L_E^{l}(z)$ on the zero-free annuli, we now lift this property to the operator level.
Fix any admissible $\mathfrak A$ for topological dynamical system 
$(\Omega,T)$, we show that the scalar symbol guarantees the uniform invertibility of the associated infinite-dimensional  \textit{symbol operator} on
$\mathfrak A$.

For a real interval $I$, let
\begin{equation}\label{complexneig}
I_\delta:=
\{\zeta\in\mathbb C:\operatorname{dist}(\zeta,I)<\delta\}.
\end{equation}
Denote by $\mathcal H_\delta(\mathfrak A)$ the space of functions
$g:I_\delta\times\Omega\to\mathbb C$ that are holomorphic in $\zeta$ and satisfy
\begin{equation}\label{xdeltanorm}
\|g\|_{\mathfrak A,\delta}
:=\sup_{\zeta\in I_\delta}\|g(\zeta,\cdot)\|_{\mathfrak A}<\infty.
\end{equation}

Recall that the unperturbed spectral gaps are  determined by the non-vanishing of the scalar  polynomial $L_E^{l}(z) = \sum_{j=-l}^{l} \hat{v}_j z^j - E$. 
To utilize this algebraic property within our functional framework, we must formally lift the scalar symbol to an operator on $\mathfrak A$.
In the phase space, the underlying dynamics are driven by the translation operator $\mathbf T $. We define the unperturbed symbol operator $D_{E,0}^{l}(z)$ by evaluating the polynomial $L_E^l$ directly at the weighted translation operator $z\mathbf T$. This formal substitution yields:
$$
D_{E,0}^{l}(z) := 
\sum_{j=-l}^{l} \hat{v}_j (z\mathbf T)^j - E \,\mathrm{Id}=\sum_{j=-l}^{l} \hat{v}_j z^j \mathbf T^j - E \,\mathrm{Id}.
$$
Pointwise, this functional action explicitly reads:
$$
(D_{E,0}^{l}(z)f)(\theta) = \sum_{j=-l}^{l} \hat{v}_j z^j f(\mathbf T^j\theta) - E f(\theta).
$$
This exact algebraic homomorphism ensures that the operator $D_{E,0}^{l}(z)$ completely inherits the spectral structure of the scalar symbol.

The \textit{full symbol operator} includes the external potential $W$. For $\varepsilon \in \R$, we define $D_{E,\varepsilon}^{l}(z) : \mathfrak A \to \mathfrak A$ by adding the multiplication operator  $(M_W f)(\theta):=W(\theta)f(\theta)$: 
\begin{equation}\label{full-sym}
    D_{E,\varepsilon}^{l}(z):=D_{E,0}^{l}(z) + \varepsilon M_W
=\sum_{j=-l}^{l} \hat{v}_j z^j \mathbf T^j - E \,\mathrm{Id} + \varepsilon M_W.
\end{equation}
Because $V$ is analytic and $\mathfrak A$ is a Banach algebra on which $\mathbf T$ acts as an isometric automorphism, both $D_{E,0}^{l}(z)$ and $D_{E,\varepsilon}^{l}(z)$ are inherently bounded linear operators on $\mathfrak A$.

We are now ready to prove that the uniform boundedness of the scalar symbol from zero directly guarantees the uniform invertibility of the symbol operator $D_{\zeta,\varepsilon}^{l}(z)$, which plays the key role in establish partially hyperbolicity of the fibered solution bundle, as we will explain in Section \ref{sec:ambient_framework}.  The proof effectively relies on interpreting the inverse of the scalar symbol as a Fourier multiplier, and bounding the small perturbations via a Neumann series.

\begin{lemma}\label{thm:D-abstract}
Let $\mathfrak A$ be an admissible space,  and let $\delta_0, l_0, \eta_{\pm}(E)$ be as previously fixed. Define $\delta_1 := \delta_0/4$, $\kappa := \frac{1}{4}\min\{\eta_{\pm}(E)\} > 0$, and the universal constants:
\begin{equation}\label{e0md}
M_D := \frac{4}{\delta_0} \left( \frac{1+e^{-\kappa}}{1-e^{-\kappa}} \right), \qquad \varepsilon_0 := \bigl(M_D\|W\|_\mathfrak A\bigr)^{-1}.
\end{equation}
For all $l \in [l_0, \infty]$, $|\varepsilon| \le \varepsilon_0$, $\zeta \in I_{\delta_1}$, and $z \in \mathcal{A}'(E)$, the following properties hold:

\begin{enumerate}
    \item The symbol operator $D_{\zeta,\varepsilon}^{l}(z)$ is uniformly  invertible on $\mathfrak A$, with estimates
    \begin{equation}\label{eq:uniform}
         \big\|\bigl(D_{\zeta,\varepsilon}^{l}(z)\bigr)^{-1}\bigr\|_{\mathfrak{B}(\mathfrak A)} \le M_D. 
    \end{equation}
    \item  As $l \to \infty$, the truncated inverse converges to the infinite-range limit in the uniform operator topology,
    \begin{equation}\label{equ:symcon}
    \bigl\|
    \bigl(D_{\zeta,\varepsilon}^{l}(z)\bigr)^{-1} -
    \bigl(D_{\zeta,\varepsilon}^{\infty}(z)\bigr)^{-1}
    \bigr\|_{\mathfrak{B}(\mathfrak A)}
    \to 0,
    \end{equation}
    uniformly in $\zeta \in I_{\delta_1}$, $|\varepsilon| \le \varepsilon_0$, and $z \in \mathcal{A}'(E)$.

    \item  The section defined by $g^{l}_{\zeta,\varepsilon}(z,\cdot) := \bigl(D^{l}_{\zeta,\varepsilon}(z)\bigr)^{-1}\mathbf{1}$ satisfies the norm bound
    \begin{equation}\label{eq:unig}
    \|g^{l}_{\cdot,\varepsilon}(z,\cdot)\|_{\mathfrak A,\delta_1} \le M_D,
    \end{equation}
    and the perturbation estimate
    $$
    \|g^{l}_{\cdot,0}(z,\cdot) - g^{l}_{\cdot,\varepsilon}(z,\cdot)\|_{\mathfrak A,\delta_1} \le M_D^2\|W\|_\mathfrak A |\varepsilon|.
    $$
\end{enumerate}
\end{lemma}

\begin{proof}
We detail the proof for the outer annulus component  $z\in\mathcal A'_+(E)$ defined in \eqref{eq:a+}, the inner annulus case follows by identical reasoning. Since  $L_{\zeta}^{}(z)=V(z)-\zeta$ depends continuously on $\zeta$, by Lemma \ref{prop:assume} and \eqref{delta0},  for all $l \in [l_0, \infty]$, we have
$$
\inf_{\zeta\in I_{\delta_1}}\inf_{z\in\mathcal A'_+(E)} |L_{\zeta}(z)| \ge \frac{3}{4}\delta_0,
\qquad
\sup_{\zeta\in I_{\delta_1}}\sup_{z\in\mathcal A'_+(E)} |L^{l}_{\zeta}(z)-L_{\zeta}(z)| \le \frac{1}{4}\delta_0.
$$
This immediately yields the strict lower bound:
\begin{equation}\label{eq:lb}
\inf_{\zeta\in I_{\delta_1}}\inf_{z\in\mathcal A'_+(E)} |L^{l}_{\zeta}(z)| \ge \frac{1}{2}\delta_0.
\end{equation}

Fix $\zeta\in I_{\delta_1}$ and $z\in\mathcal{A}'_+(E)$. The function $a^{l}(z;\omega):=L^{l}_{\zeta}(z\omega)$ is bounded away from zero on $\omega\in \T$, so its inverse $b^{l}(z;\omega):=1/a^{l}(z;\omega)$ is analytic on the  annulus $\mathbb A:=\{\omega\in\mathbb C:  e^{-\eta_+(E)/2}< |\omega|<  e^{\eta_+(E)/2}\}$. It admits a uniformly convergent Laurent expansion $b^{l}(z;\omega)=\sum_{n\in\Z} b^{l}_n(z) \omega^n$. Using \eqref{eq:lb},
applying Cauchy's integral estimates on the contours $|\omega| = e^{\pm \kappa}$, with $\kappa = \frac{1}{4}\min\{\eta_\pm(E)\}$, yields the exponential decay bound $|b^{l}_n(z)|\le \frac{2}{\delta_0} e^{-\kappa |n|}$. Thus, the coefficients are absolutely summable:
$$
\sum_{n\in\Z}|b^{l}_n(z)| \le \frac{2}{\delta_0} \left( \frac{1+e^{-\kappa}}{1-e^{-\kappa}} \right) = \frac{M_D}{2}.
$$

We utilize these Laurent coefficients to explicitly construct the inverse operator. Define $B_{\zeta,0}^{l}(z):=\sum_{n\in\Z} b^{l}_n(z) \mathbf T^n$, since $\|\mathbf T^n\|_{\mathfrak{B}(\mathfrak A)}=1$, it implies 
$$
\|B_{\zeta,0}^{l}(z)\|_{\mathfrak{B}(\mathfrak A)} \le \sum_{n\in\Z}|b^{l}_n(z)| \le \frac{M_D}{2}.
$$
Writing the unperturbed operator as $D_{\zeta,0}^{l}(z)=\sum_{k\in\Z} d_k^{l}(z) \mathbf T^k$ (where $d_k^{l}(z)=\hat{v}_k z^k$ for $k\neq 0$, and $d_0^{l}(z)=\hat{v}_0-\zeta$), the composition gives:
$$
D_{\zeta,0}^{l}(z)B_{\zeta,0}^{l}(z) = \sum_{m\in\Z} \Bigl( \sum_{k+n=m} d_k^{l}(z)b^{l}_n(z) \Bigr) \mathbf T^m.
$$
The inner summation evaluates exactly the $m$-th Laurent coefficient of the constant function $a^{l}(z;\omega)b^{l}(z;\omega) \equiv 1$, which is $\delta_{m,0}$. This forces $D_{\zeta,0}^{l}(z)B_{\zeta,0}^{l}(z)=\mathrm{Id}$. Symmetrically, $B_{\zeta,0}^{l}(z) D_{\zeta,0}^{l}(z)=\mathrm{Id}$, proving that $\bigl(D_{\zeta,0}^{l}(z)\bigr)^{-1} = B_{\zeta,0}^{l}(z)$.

For the full perturbed operator, we write:
$$
D^{l}_{\zeta,\varepsilon}(z) = D^{l}_{\zeta,0}(z)+\varepsilon M_W = D^{l}_{\zeta,0}(z)\left(\mathrm{Id}+\varepsilon \bigl(D^{l}_{\zeta,0}(z)\bigr)^{-1}M_W\right).
$$
Thus  whenever $|\varepsilon| \le \varepsilon_0$, the perturbation term is strictly bounded:
$$
|\varepsilon|\, \bigl\|\bigl(D_{\zeta,0}^{l}(z)\bigr)^{-1}\bigr\|_{\mathfrak{B}(\mathfrak A)} \|M_W\|_{\mathfrak{B}(\mathfrak A)} \le \varepsilon_0 \frac{M_D}{2} \|W\|_{\mathfrak A} = \frac{1}{2}.
$$
Consequently, the operator $\mathrm{Id}-\varepsilon \bigl(D^{l}_{\zeta,0}(z)\bigr)^{-1}M_W$ is invertible via its Neumann series. This implies that $D^{l}_{\zeta,\varepsilon}(z)$ is uniformly invertible on $\mathfrak A$ with estimate \eqref{eq:uniform}.

To establish the convergence \eqref{equ:symcon}, observe that $D_{\zeta,\varepsilon}^{l}(z) - D_{\zeta,\varepsilon}^{\infty}(z) = -\sum_{|j|>l} \hat v_j z^j \mathbf T^j$. For $z\in\mathcal A'(E)$, the exponential decay of $\hat v_j$ guarantees
$$
\bigl\| D_{\zeta,\varepsilon}^{l}(z) - D_{\zeta,\varepsilon}^{\infty}(z) \bigr\|_{\mathfrak{B}(\mathfrak A)}
\le
\sum_{|j|>l}|\hat v_j|\,|z|^{|j|} \to 0 \qquad \text{as } l \to \infty,
$$ 
uniformly in $\zeta\in I_{\delta_1}$ and $z\in\mathcal A'(E)$. Applying the resolvent identity,
$$
\bigl(D_{\zeta,\varepsilon}^{l}(z)\bigr)^{-1} - \bigl(D_{\zeta,\varepsilon}^{\infty}(z)\bigr)^{-1}=
\bigl(D_{\zeta,\varepsilon}^{l}(z)\bigr)^{-1}
\Bigl(D_{\zeta,\varepsilon}^{\infty}(z) - D_{\zeta,\varepsilon}^{l}(z)\Bigr)
\bigl(D_{\zeta,\varepsilon}^{\infty}(z)\bigr)^{-1},
$$
and leveraging the uniform bound $M_D$, we deduce
$$
\bigl\| \bigl(D_{\zeta,\varepsilon}^{l}(z)\bigr)^{-1} - \bigl(D_{\zeta,\varepsilon}^{\infty}(z)\bigr)^{-1} \bigr\|_{\mathfrak{B}(\mathfrak A)} 
\le M_D^2 \bigl\| D_{\zeta,\varepsilon}^{\infty}(z) - D_{\zeta,\varepsilon}^{l}(z) \bigr\|_{\mathfrak{B}(\mathfrak A)},
$$
which tends to zero, completing the proof of \eqref{equ:symcon}.

Finally, since $\mathfrak A$ is a unital Banach algebra ($\|\mathbf{1}\|_{\mathfrak A}=1$),  \eqref{eq:unig} follows directly  from  \eqref{eq:uniform}.
Another application of the resolvent identity bounds the perturbation:
\begin{align*}
\|g^{l}_{\cdot,0}(z,\cdot)-g^{l}_{\cdot,\varepsilon}(z,\cdot)\|_{\mathfrak A,\delta_1} 
&\le \sup_{\zeta\in I_{\delta_1}} \bigl\|\bigl(D^{l}_{\zeta,\varepsilon}(z)\bigr)^{-1} (\varepsilon M_W) \bigl(D^{l}_{\zeta,0}(z)\bigr)^{-1}\bigr\|_{\mathfrak{B}(\mathfrak A)} \|\mathbf{1}\|_{\mathfrak A} \\
&\le M_D^2 \cdot \|W\|_{\mathfrak A} |\varepsilon| .
\end{align*}
\end{proof}

\section{ The Fibered Solution Bundle and Dual-Bridge Framework}
\label{sec:ambient_framework}

In this section, we investigate the dynamical properties of the fibered solution bundle introduced in Definition \ref{def:solution-space}. Our primary objective is to establish that the fibered solution bundle admits a partially hyperbolic splitting. By leveraging an abstract Riesz spectral projection, we isolate the \emph{intrinsic center space}, which we will subsequently prove to form a globally trivial bundle in Section \ref{trivial-center}. Crucially, this geometric construction is entirely coordinate-free and independent of any finite-dimensional truncation.
\begin{theorem}
\label{thm:sec7final}
Fix $E\in\R$. For every $\varepsilon$ with $|\varepsilon|<\varepsilon_0$ (where $\varepsilon_0$ is defined in Lemma \ref{thm:D-abstract}), the fibered solution bundle admits a $\mathscr{T}$-invariant splitting:
$$
\mathcal{M}_{V,E,\varepsilon}(\theta)=
\mathcal{M}_{V,E,\varepsilon}^s(\theta)
\oplus\mathcal{M}_{V,E,\varepsilon}^c(\theta)
\oplus\mathcal{M}_{V,E,\varepsilon}^u(\theta),
\qquad \forall \theta\in\Omega.
$$
where the invariance satisfies $S \mathcal{M}_{V,E,\varepsilon}^\ast(\theta) = \mathcal{M}_{V,E,\varepsilon}^\ast(T\theta)$ for $\ast \in \{s, c, u\}$. Furthermore, this splitting is  partially hyperbolic: there exists a constant $C= C(\xi, \|V\|_h, r_\pm, \eta_\pm) > 0$ such that for every $\theta\in\Omega$, every integer $n \ge 1$, and any initial vector $u_\ast \in \mathcal{M}_{V,E,\varepsilon}^\ast(\theta)$:
\begin{enumerate}
    \item (Stable Contraction) \quad $\|S^n u_s\|_\xi \le C r_s^n \|u_s\|_\xi$,
    \item (Center Domination) \quad $C^{-1} r_{-}^n \|u_c\|_\xi \le \|S^n u_c\|_\xi \le C r_+^n \|u_c\|_\xi$,
    \item (Unstable Expansion) \quad $\|S^n u_u\|_\xi \ge C^{-1} r_u^n \|u_u\|_\xi$,
\end{enumerate}
where the spectral radii are separated by the  weight $\xi$ satisfying:
\begin{equation}\label{radius}
   e^{-\xi} < r_s < r_- \le 1 \le r_+ < r_u < e^\xi.
\end{equation}
\end{theorem}

\begin{remark}\label{uniform-constant}
    In the later sections, we will apply Theorem \ref{thm:sec7final} not only to $V$, but also to its truncations $V_l$. The constant $C$ can be chosen uniformly in $l$. This follows from the fact that  $r_\pm,\eta_\pm$ have already been chosen uniformly for the family $\{V_l\}_{l\in[l_0,\infty]}$ in terms of Lemma \ref{prop:assume}, while
$
\sup_{l\in[l_0,\infty]}\|V_l\|_h\le \|V\|_h.
$
Hence the constant appearing in Theorem \ref{thm:sec7final} remain uniform for the truncated family as well.
\end{remark}

\subsection{Transfer Operator and Resolvent operator}

We define Mather's transfer operator on the Banach space of global weighted sections $\mathcal{BS}_\xi$ by
\begin{equation}\label{mather-operator}(\mathcal{ L }\Phi)(\theta) := S\Phi(T^{-1}\theta).\end{equation}
It is clear $S\in \mathfrak{B}(\mathcal{X}_{\xi})$ with estimate $\|S\|_{\mathfrak{B}(\mathcal{X}_{\xi})} \le e^\xi$ and $\|S^{-1}\|_{\mathfrak{B}(\mathcal{X}_{\xi})} \le e^\xi$, which implies that $\mathcal{ L }$ is a bounded linear operator on $\mathcal{BS}_\xi$.
For a section $\Phi\in\mathcal{BS}_\xi$, we write
$
\Phi(\theta)=\bigl(\Phi_n(\theta)\bigr)_{n\in\mathbb Z},
$
where $\Phi_n(\theta)$ denotes the $n$-th coordinate of the sequence
$\Phi(\theta)\in X_\xi$.

Distinct from the underlying geometric fibered space, we define $\mathcal{M}_{V,E,\varepsilon}^b$ as its associated Banach space of global sections:
$$\mathcal{M}_{V,E,\varepsilon}^b := \bigl\{ \Phi\in \mathcal{BS}_\xi : \Phi(\theta) \in \mathcal{M}_{V,E,\varepsilon}(\theta), \ \forall \theta\in\Omega \bigr\}.$$
The covariance relation $\mathbf{L}_{V,\varepsilon W,T,\theta} \circ S = S \circ \mathbf{L}_{V,\varepsilon W,T,T^{-1}\theta}$  guarantees that the transfer operator $\mathcal{L}$ preserves the section space $\mathcal{M}_{V,E,\varepsilon}^b$. Our objective  is to establish the uniform invertibility of $z-\mathcal{L}$ on $\mathcal{M}_{V,E,\varepsilon}^b$, setting the analytical stage to construct the invariant splitting of the geometric fibers via Riesz projections.

To construct the inverse of $z-\mathcal{L}$, we  first analyze the resolvent equation $(z-\mathcal{L})\Phi = \Psi$ for a given section $\Psi \in \mathcal{BS}_\xi$. Recalling that the transfer operator acts via the shift \eqref{mather-operator}, evaluating the resolvent equation at the $n$-th sequence coordinate yields a first-order non-homogeneous difference equation:
\begin{equation}\label{eq:recurrence}z\Phi_n(\theta) - \Phi_{n+1}(T^{-1}\theta) = \Psi_n(\theta), \qquad n \in \Z.\end{equation}
Assigning the initial data $f(\theta) := \Phi_0(\theta)$ and integrating \eqref{eq:recurrence} forward for $n \ge 1$ and backward for $n \le -1$ gives the explicit algebraic decomposition:
$$\Phi_n(\theta) = (\mathcal{P}_z f)_n(\theta) + (\mathcal{Q}_z\Psi)_n(\theta),$$
where  $\mathcal{P}_z: B(\Omega, \C) \to \mathcal{BS}_\xi$ is defined by
\begin{equation}\label{eq:Kz-def}(\mathcal{P}_z f)_n(\theta) := z^n f(T^n\theta),
\end{equation}
and  $\mathcal{Q}_z: \mathcal{BS}_\xi \to \mathcal{BS}_\xi$ is defined as :
\begin{equation}\label{eq:Qz-def}
(\mathcal{Q}_z\Psi)_n(\theta) :=
\begin{cases}
-\displaystyle\sum_{j=0}^{n-1} z^{n-1-j}\Psi_j\bigl(T^{n-j}\theta\bigr), & n\ge 1,\\
0, & n=0,\\
\displaystyle\sum_{j=1}^{|n|} z^{-|n|-1+j}\Psi_{-j}\bigl(T^{-(|n|-j)}\theta\bigr), & n \le -1.
\end{cases}
\end{equation}

Thus the only remaining freedom in the representation is the initial date $f$. 
Write  the potential as  $V(\theta) = \sum_{|k|\le l} \hat{v}_k e^{2\pi i k\theta}$ for $l \in [1, \infty]$, we define:
$$(\mathbb{H}_\varepsilon\Phi)(\theta) := \mathbf{L}_{V,\varepsilon W,T,\theta}\Phi(\theta) - E\Phi(\theta).$$ The following lemma shows how the constraint
$\mathbb H_\varepsilon\Phi\equiv0$ fixes this initial date.

\begin{lemma}\label{lem:zeroth} 
  Let $z \in \mathcal{A}'(E)$, $f\in B(\Omega,\mathbb C)$, and $\Psi\in\mathcal{BS}_\xi$. Assume that the formal sequence$$\Phi := \mathcal P_z f + \mathcal Q_z\Psi$$  satisfies the fiberwise constraint $\mathbb H_\varepsilon\Phi \equiv 0$. Then the initial data is uniquely determined by $f = D^l_{E,\varepsilon}(z)^{-1}\mathcal G_z^l(\Psi),$
  where $D^l_{E,\varepsilon}(z) $ is the full symbol operator  defined  in   \eqref{full-sym}, and 
  $$\begin{aligned}
\mathcal{G}_z^l(\Psi)(\theta) := \sum_{k=1}^l \hat{v}_k \sum_{j=0}^{k-1} z^{k-1-j}\Psi_j\bigl(T^{(k-j)}\theta\bigr) 
- \sum_{m=1}^l \hat v_{-m}\sum_{j=1}^{m} z^{-m-1+j}\Psi_{-j}\bigl(T^{-(m-j)}\theta\bigr).
\end{aligned}$$
\end{lemma}

\begin{proof}
Evaluating the constraint $(\mathbb H_\varepsilon\Phi)(\theta) = 0$ exactly at the coordinate $n=0$ yields
$$
\sum_{|k|\le l} \hat{v}_k \Phi_k(\theta) + \bigl(\varepsilon W(\theta) - E\bigr)\Phi_0(\theta) = 0.
$$
Substituting the decomposition $\Phi = \mathcal{P}_z f + \mathcal{Q}_z \Psi$, we group the respective terms. Since $(\mathcal{P}_z f)_k(\theta) = z^k f(T^k\theta)$, the contribution from the homogeneous part $\mathcal{P}_z f$ is exactly $D^l_{E,\varepsilon}(z)f(\theta)$. 

For the non-homogeneous component $\mathcal{Q}_z \Psi$, noting that $(\mathcal{Q}_z \Psi)_0 = 0$, direct substitution into the off-diagonal sum yields
$$
\sum_{k=1}^l \hat{v}_k (\mathcal{Q}_z \Psi)_k(\theta) + \sum_{m=1}^l \hat{v}_{-m} (\mathcal{Q}_z \Psi)_{-m}(\theta) = -\mathcal{G}_z^l(\Psi)(\theta).
$$
Consequently, the $n=0$ equation reduces  to $D^l_{E,\varepsilon}(z)f - \mathcal{G}_z^l(\Psi) = 0$. Since $z\in\mathcal{A}'(E)$, Lemma \ref{thm:D-abstract} guarantees the invertibility of $D^l_{E,\varepsilon}(z)$, which determines $f = D^l_{E,\varepsilon}(z)^{-1}\mathcal G_z^l(\Psi)$.
\end{proof}

This motivates the following definition
of the \emph{extended resolvent operator} $\mathcal{R}_\varepsilon^l(z): \mathcal{BS}_\xi \to \mathcal{BS}_\xi$ as\begin{equation}\label{def:rec}\mathcal{R}_\varepsilon^l(z)\Psi := \mathcal{P}_z(f) + \mathcal{Q}_z\Psi.\end{equation}

\begin{proposition}\label{thm:D-regular}For any $z \in \mathcal{A}'(E)$, the extended resolvent operator $\mathcal{R}_\varepsilon^l(z)$ satisfies:
\begin{enumerate}
 \item There exists a constant $C = C(\xi, \|V\|_h, r_\pm, \eta_\pm) > 0$ such that 
$
\|\mathcal R_\varepsilon^{l}(z)\|_{\mathfrak{B}(\mathcal{BS_\xi})}   < C.
$
\item  The operator $z-\mathcal{L}$ is invertible on the section space $\mathcal{M}_{V,E,\varepsilon}^b$, and its inverse coincides with the restriction of the extended resolvent:$$(z-\mathcal{L})^{-1} = \mathcal{R}_\varepsilon^l(z)\big|_{\mathcal{M}_{V,E,\varepsilon}^b}.$$
\end{enumerate}
\end{proposition}
\begin{proof}
(1)
Write $r := |z|$. By the definition of the weighted norm on $\mathcal{BS_\xi}$, we immediately have for $\mathcal{P}_z$:
$$
\sum_{n\in\Z} e^{-\xi |n|}\bigl|(\mathcal{P}_z f)_n(\theta)\bigr| = \sum_{n\in\Z}\bigl(|z|e^{-\xi}\bigr)^{|n|} |f(T^n\theta)| \le C\|f\|_{B(\Omega,\C)},
$$
since $|z|e^{-\xi} \le r_+ e^{\eta_+(E)/2} e^{-\xi} < 1$. Thus, $\|\mathcal{P}_z \|_\infty \leq C$.

For $\mathcal{Q}_z$, let $\Psi\in \mathcal{BS_\xi}$. We bound the positive indices $n \ge 1$; the negative case is strictly analogous.
\begin{align*}
\sum_{n= 1}e^{-\xi n}\bigl|(\mathcal{Q}_z\Psi)_n(\theta)\bigr|
\le \sum_{n=1}\sum_{j=0}^{n-1} e^{-\xi n} r^{n-1-j} |\psi_j(T^{n-j}\theta)|
 \le \frac{1}{e^\xi - r} \|\Psi\|_\infty.
\end{align*}
Considering both positive and negative branches, we obtain $$\|\mathcal{Q}_z\Psi\|_\infty \le \max\bigl\{ \frac{1}{e^\xi - r}, \frac{1}{r - e^{-\xi}} \bigr\} \|\Psi\|_\infty.$$

Finally, to estimate the term $\mathcal{G}_z^l(\Psi)$, we sum the positive branch interactions using $r < e^h$ and $\xi < h$:
$$
\left| \sum_{k=1}^l \hat{v}_k \sum_{j=0}^{k-1} z^{k-1-j}\Psi_j\bigl(T^{k-j}\theta\bigr) \right| 
\le \|\Psi\|_\infty \sum_{k\ge 1} |\hat{v}_k| \sum_{j=0}^{k-1} r^{k-1-j} e^{\xi j}.
$$
Since $r^{k-1-j} e^{\xi j} \le \max\{r, e^\xi\}^{k-1}$, the sum is bounded by $\|\Psi\|_\infty \sum_{k\ge 1} k |\hat{v}_k| \max\{r, e^\xi\}^{k-1}$. Because $\hat{v}_k$ decays as $e^{-h|k|}$ and $\max\{r, e^\xi\} < e^h$, this series converges absolutely to some constant $C_1$. Applying the same logic to the negative branch yields $\|\mathcal{G}_z^l(\Psi)\|_{B(\Omega,\C)} \le C \|\Psi\|_\infty$. Combining these estimates proves $
\|\mathcal R_\varepsilon^{l}(z)\|_{\mathfrak{B}(\mathcal{BS_\xi})} < C.$

(2) To prove injectivity, assume $\Phi \in \mathcal{M}_{V,E,\varepsilon}^b$ satisfies $(z-\mathcal{L})\Phi = 0$. Applying Lemma \ref{lem:zeroth} with $\Psi=0$ yields $D^l_{E,\varepsilon}(z)\Phi_0 = 0$. The uniform invertibility of the symbol (Lemma \ref{thm:D-abstract}) forces $\Phi_0 = 0$, which consequently implies $\Phi \equiv 0$.

For surjectivity, take $\Psi \in \mathcal{M}_{V,E,\varepsilon}^b$ and define $\Phi := \mathcal{R}_\varepsilon^l(z)\Psi$. By construction, $(z-\mathcal{L})\Phi = \Psi$ holds on $\mathcal{BS}_\xi$. It remains to verify $\Phi \in \mathcal{M}_{V,E,\varepsilon}^b$. Consider the residual $R := \mathbb{H}_\varepsilon\Phi$. Since $[\mathbb{H}_\varepsilon, \mathcal{L}] = 0$, we deduce
$$
(z-\mathcal{L})R = \mathbb{H}_\varepsilon(z-\mathcal{L})\Phi = \mathbb{H}_\varepsilon\Psi = 0.
$$
Consequently, $R$ satisfies the homogeneous recurrence $z R_n(\theta) - R_{n+1}(T^{-1}\theta) = 0$. However, the initial data $f$ was strictly chosen such that $R_0 = D_{E,\varepsilon}^l(z)f - \mathcal{G}_z^l(\Psi) = 0$. A homogeneous recurrence with zero initial data dictates $R \equiv 0$, confirming $\Phi \in \mathcal{M}_{V,E,\varepsilon}^b$.

The operator $z-\mathcal{L}$ is therefore a continuous bijection on $\mathcal{M}_{V,E,\varepsilon}^b$. The Bounded Inverse Theorem concludes the proof.
\end{proof}

\subsection{Bounded solution space decomposition}

Recall the rigorously isolated radii $r_s < r_- < 1 < r_+ < r_u$ and their corresponding integration contours defined in \eqref{contour}. 
 Proposition \ref{thm:D-regular} ensures that the resolvent $(z-\mathcal{ L })^{-1}$ is analytic on the annuli containing these contours, we can define the Riesz spectral projections:
$$
\mathbb{P}^\sharp_{V,E,\varepsilon} := \frac{1}{2\pi i}\oint_{\Gamma_\sharp}(z-\mathcal{ L })^{-1} dz, \qquad \sharp\in\{s,c\},
$$
and set $\mathbb{P}^u_{V,E,\varepsilon} := \mathrm{Id} - \mathbb{P}^c_{V,E,\varepsilon} - \mathbb{P}^s_{V,E,\varepsilon}$.

 Cauchy's theorem gives
$$
\frac{1}{2\pi i}\int_{\Gamma_s}(z -\mathcal{ L })^{-1} dz=
\frac{1}{2\pi i}\int_{\Gamma_-}(z -\mathcal{ L })^{-1} dz
$$
and
$$
\frac{1}{2\pi i}\int_{\Gamma_+}(z -\mathcal{ L })^{-1} dz=
\frac{1}{2\pi i}\int_{\Gamma_u}(z -\mathcal{ L })^{-1} dz.
$$
Hence $\mathbb{P}^s_{V,E,\varepsilon}$, $\mathbb{P}^c_{V,E,\varepsilon}$, and
$\mathbb{P}^u_{V,E,\varepsilon}$ are precisely the Riesz projections associated with the three disjoint
spectral parts
$$
\{|z|<r_s\},
\qquad
\{r_-<|z|<r_+\},
\qquad
\{|z|>r_u\}.
$$
By the  standard consequences of the Riesz functional calculus, we have 
\begin{equation}\label{equ:riesz}
    (\mathbb{P}^\sharp_{V,E,\varepsilon})^2=\mathbb{P}^\sharp_{V,E,\varepsilon},
\quad
\mathbb{P}^\sharp_{V,E,\varepsilon}\mathbb{P}^\flat_{V,E,\varepsilon}=0\ (\sharp\neq \flat),
\quad
\mathbb{P}^s_{V,E,\varepsilon}+\mathbb{P}^c_{V,E,\varepsilon}+\mathbb{P}^u_{V,E,\varepsilon}=\mathrm{Id},\sharp,\flat\in\{s,c,u\}.
\end{equation}

Defining the global invariant subspaces via their corresponding spectral projections:
$$
\mathcal{S}_{V,E,\varepsilon} := \operatorname{Ran}\mathbb{P}^s_{V,E,\varepsilon}, \qquad
\mathcal{C}_{V,E,\varepsilon} := \operatorname{Ran}\mathbb{P}^c_{V,E,\varepsilon}, \qquad
\mathcal{U}_{V,E,\varepsilon} := \operatorname{Ran}\mathbb{P}^u_{V,E,\varepsilon},
$$
the space of global bounded sections admits the invariant direct sum decomposition:
$$
\mathcal{M}_{V,E,\varepsilon}^b
=\mathcal{S}_{V,E,\varepsilon}\oplus
\mathcal{C}_{V,E,\varepsilon}\oplus
\mathcal{U}_{V,E,\varepsilon}.
$$
These spectral subspaces are dynamically  characterized by their precise exponential growth rates under the transfer operator $\mathcal{L}$.

\begin{lemma}
\label{prop:growth}
The invariant subspaces admit the following  dynamical characterizations:
$$
\begin{aligned}
\mathcal{S}_{V,E,\varepsilon} &= \Bigl\{ \Phi\in \mathcal{M}_{V,E,\varepsilon}^b : \sup_{n\ge 0} \, r_s^{-n} \|\mathcal{L}^n\Phi\|_\infty < \infty \Bigr\}, \\
\mathcal{C}_{V,E,\varepsilon} &= \Bigl\{ \Phi\in \mathcal{M}_{V,E,\varepsilon}^b : \sup_{n\ge 0} \, \max\bigl\{ r_+^{-n} \|\mathcal{L}^n\Phi\|_\infty, \, r_-^n \|\mathcal{L}^{-n}\Phi\|_\infty \bigr\} < \infty \Bigr\}, \\
\mathcal{U}_{V,E,\varepsilon} &= \Bigl\{ \Phi\in \mathcal{M}_{V,E,\varepsilon}^b : \sup_{n\ge 0} \, r_u^n \|\mathcal{L}^{-n}\Phi\|_\infty < \infty \Bigr\}.
\end{aligned}
$$
\end{lemma}

\begin{proof}The forward inclusions are immediate consequences of standard Riesz spectral calculus, which guarantees a uniform constant $C_* > 0$ such that:
$$\|\mathcal{L}^n|_{\mathcal{S}_{V,E,\varepsilon}}\| \le C_* r_s^n, \qquad
\|\mathcal{L}^{-n}|_{\mathcal{U}_{V,E,\varepsilon}}\| \le C_* r_u^{-n},$$
with symmetric bounds on $\mathcal{C}_{V,E,\varepsilon}$ governed by $r_+$ and $r_-^{-1}$.

Conversely, we verify the characterization for the stable subspace $\mathcal{S}_{V,E,\varepsilon}$; the remaining cases are  analogous. 
For the center component $\Phi^c =\mathbb{P}^c_{V,E,\varepsilon}\Phi$, 
$$
\|\Phi^c\|_\infty
=\|\mathcal{ L }^{-n}\mathcal{ L }^{n}\Phi^c\|_\infty\leq
C_* r_-^{-n}\|\mathcal{ L }^{n}\Phi^c\|_\infty,
$$
hence
$
\|\mathcal{ L }^{n}\Phi^c\|_\infty\ge
C_*^{-1} r_-^{n}\|\Phi^c\|_\infty.
$
Since  there exists $C>0$ such that
$
\|\mathcal{ L }^{n}\Phi\|_\infty\leq C r_s^n.
$
As $r_s < r_-$, taking the limit $n \to \infty$ forces $\Phi^c = 0$. Parallel application of the unstable projection $\mathbb{P}^u$, exploiting  $r_s < r_u$ forces $\Phi^u = 0$. Consequently, $\Phi = \Phi^s \in \mathcal{S}_{V,E,\varepsilon}$.
\end{proof}

As a direct consequence, we have the following:

\begin{corollary}\label{commu}
    For every $\varphi\in B(\Omega,\C)$ and every $\sharp\in\{s,c,u\}$,
$
\mathbb{P}^\sharp_{V,E,\varepsilon}M_\varphi=M_\varphi \mathbb{P}^\sharp_{V,E,\varepsilon}
$, where 
$
(M_\varphi \Phi)(\theta):=\varphi(\theta)\Phi(\theta).
$
\end{corollary}
\begin{proof}
In fact, we first note that
$
\|M_\varphi \Phi\|_\infty\leq \|\varphi\|_\infty \|\Phi\|_\infty.
$
Moreover, for every $n\ge0$,
$$
\mathcal{ L }^{n}(M_\varphi\Phi)=M_{\varphi\circ T^{-n}} \mathcal{ L }^{n}\Phi,
$$
hence
$
\|\mathcal{ L }^{n}(M_\varphi\Phi)\|_\infty\leq
\|\varphi\|_\infty \|\mathcal{ L }^{n}\Phi\|_\infty.
$
If $\Phi\in\mathcal{S}_{V,E,\varepsilon}$, then by Lemma
\ref{prop:growth}, $M_\varphi\Phi\in \mathcal{S}_{V,E,\varepsilon}$. The same argument applies to the corresponding center and unstable spaces.

Now let $\Phi\in\mathcal{M}_{V,E,\varepsilon}^b$, and write its unique decomposition
$
\Phi=\Phi^s+\Phi^c+\Phi^u,
\Phi^s\in\mathcal{S}_{V,E,\varepsilon},\Phi^u\in  \mathcal{U}_{V,E,\varepsilon},\Phi^c\in \mathcal{C}_{V,E,\varepsilon}.$
Then
$
M_\varphi\Phi=M_\varphi\Phi^s+M_\varphi\Phi^c+M_\varphi\Phi^u
$
is the decomposition of $M_\varphi\Phi$ with respect to the same direct sum, because
each summand remains in the corresponding subspace. By uniqueness of the
direct-sum decomposition, the $\sharp$-component of $M_\varphi\Phi$ is exactly
$M_\varphi\Phi^\sharp$. Therefore
$
\mathbb{P}^\sharp_{V,E,\varepsilon}(M_\varphi\Phi)=M_\varphi(\mathbb{P}^\sharp_{V,E,\varepsilon}\Phi).
$
\end{proof}

\subsection{Fiber decomposition}

Fix $\theta_0\in\Omega$.
Recall that  $\mathcal M_{V,E,\varepsilon}(\theta)=
\ker\bigl(\mathbf L_{V,\varepsilon W,T,\theta}-E\bigr){|_{\mathcal{X}_\xi}}$. It is easy to see that
$$
\mathcal{M}_{V,E,\varepsilon}(\theta_0)=
\{\Phi(\theta_0):\Phi\in\mathcal{M}_{V,E,\varepsilon}^b\}
\subset \mathcal{X}_{\xi}.
$$
 For each $\sharp\in\{s,c,u\}$, define
$$
v=\Phi(\theta_0)\in \mathcal{M}_{V,E,\varepsilon}(\theta_0),\quad
\mathbb{P}^\sharp_{V,E,\varepsilon}(\theta_0)v:=
\bigl(\mathbb{P}^\sharp_{V,E,\varepsilon}\Phi\bigr)(\theta_0).
$$
\begin{lemma}
\label{prop:fiber}
For any $\theta_0\in\Omega$, $\mathbb{P}^\sharp_{V,E,\varepsilon}(\theta_0)$ defines a well-defined projection on $\mathcal{M}_{V,E,\varepsilon}(\theta_0)$. If we set
$
\mathcal{M}_{V,E,\varepsilon}^\sharp(\theta_0)=
\operatorname{Ran}\mathbb{P}^\sharp_{V,E,\varepsilon}(\theta_0),
$
then
$$
\mathcal{M}_{V,E,\varepsilon}(\theta_0)=
\mathcal{M}_{V,E,\varepsilon}^u(\theta_0)
\oplus\mathcal{M}_{V,E,\varepsilon}^c(\theta_0)\oplus\mathcal{M}_{V,E,\varepsilon}^s(\theta_0).
$$
Furthermore, these projected sub-fibers coincide precisely with the point-evaluations of the corresponding global section spaces:
$$
\begin{aligned}
\mathcal{M}_{V,E,\varepsilon}^s(\theta_0) &= \bigl\{\Phi(\theta_0) : \Phi\in\mathcal{S}_{V,E,\varepsilon}\bigr\}, \\
\mathcal{M}_{V,E,\varepsilon}^c(\theta_0) &= \bigl\{\Phi(\theta_0) : \Phi\in\mathcal{C}_{V,E,\varepsilon}\bigr\}, \\
\mathcal{M}_{V,E,\varepsilon}^u(\theta_0) &= \bigl\{\Phi(\theta_0) : \Phi\in\mathcal{U}_{V,E,\varepsilon}\bigr\}.
\end{aligned}
$$
\end{lemma}

\begin{proof}
We first show $\mathbb{P}^\sharp_{V,E,\varepsilon}(\theta_0)$ is well-defined.
Let $\Phi_1,\Phi_2\in\mathcal{M}_{V,E,\varepsilon}^b$
satisfy
$
\Phi_1(\theta_0)=\Phi_2(\theta_0).
$
Set $\Psi:=\Phi_1-\Phi_2$. Then $\Psi(\theta_0)=0$.
Define the bounded scalar function
$$
\chi_{\theta_0}(\theta):=
\begin{cases}
0,&\theta=\theta_0,\\
1,&\theta\neq \theta_0.
\end{cases}
$$
Since $\Psi(\theta_0)=0$, we have pointwise
$
\Psi=M_{\chi_{\theta_0}}\Psi.
$
By Corollary \ref{commu}, we have$$
\mathbb{P}^\sharp_{V,E,\varepsilon}\Psi=
\mathbb{P}^\sharp_{V,E,\varepsilon}M_{\chi_{\theta_0}}\Psi=
M_{\chi_{\theta_0}}\mathbb{P}^\sharp_{V,E,\varepsilon}\Psi.
$$
Evaluating at $\theta_0$, we obtain
$$
\bigl(\mathbb{P}^\sharp_{V,E,\varepsilon}\Psi\bigr)(\theta_0)=
\chi_{\theta_0}(\theta_0) 
\bigl(\mathbb{P}^\sharp_{V,E,\varepsilon}\Psi\bigr)(\theta_0)=0.
$$
Thus
$
\bigl(\mathbb{P}^\sharp_{V,E,\varepsilon}\Phi_1\bigr)(\theta_0)=
\bigl(\mathbb{P}^\sharp_{V,E,\varepsilon}\Phi_2\bigr)(\theta_0),
$
so  is independent of the chosen section.

The projection identities \eqref{equ:riesz}
descend pointwise to
$$
(\mathbb{P}^\sharp_{V,E,\varepsilon}(\theta_0))^2=\mathbb{P}^\sharp_{V,E,\varepsilon}(\theta_0),
\quad
\mathbb{P}^\sharp_{V,E,\varepsilon}(\theta_0)P^\flat_{E,\varepsilon}(\theta_0)=0
\quad (\sharp\neq\flat),
$$
and
$$
\mathbb{P}^s_{V,E,\varepsilon}(\theta_0)+\mathbb{P}^c_{V,E,\varepsilon}(\theta_0)+\mathbb{P}^u_{V,E,\varepsilon}(\theta_0)=
\mathrm{Id}
$$
on $\mathcal{M}_{V,E,\varepsilon}(\theta_0)$, where $\sharp,\flat\in\{s,c,u\}.$
Therefore the ranges form the direct-sum
decomposition.

For any $v\in \{\Phi(\theta_0):\Phi\in\mathcal{S}_{V,E,\varepsilon}\}$, there exists $\Phi\in\mathcal{S}_{V,E,\varepsilon}$, such that $\Phi(\theta_0)=v$, it follows that 
$\mathbb{P}^s_{V,E,\varepsilon}(\theta_0)v=(\mathbb{P}^s_{V,E,\varepsilon}\Phi)(\theta_0)=\Phi(\theta_0)=v$, therefore, $v\in \mathcal{M}_{V,E,\varepsilon}^s(\theta_0)$. On the other hand, if $v\in\mathcal{M}_{V,E,\varepsilon}^s(\theta_0) $, let $\Phi(\theta)=v$,  if $\theta=\theta_0$; $\Phi(\theta)=0$, if $\theta\neq\theta_0$. Then  $\mathbb{P}^s_{V,E,\varepsilon}\Phi\in \mathcal{S}_{V,E,\varepsilon}$ and $(\mathbb{P}^s_{V,E,\varepsilon}\Phi)(\theta_0)=\mathbb{P}^s_{V,E,\varepsilon}(\theta_0)\Phi(\theta_0)=v.$   It follows that $v\in \{\Phi(\theta_0):\Phi\in\mathcal{S}_{V,E,\varepsilon}\}.$ The proofs of others
are completely analogous. 
\end{proof}

\subsection{Proof of Theorem \ref{thm:sec7final}}

The invariant decomposition is given by Lemma  \ref{prop:fiber}. Fix $\theta_0 \in \Omega$ and $u\in \mathcal{M}_{V,E,\varepsilon}^s(\theta_0)$. Define $\Phi \in \mathcal{BS}_\xi$ by $\Phi(\theta_0) = u$ and $\Phi(\theta) = 0$ for $\theta \neq \theta_0$, which satisfies $\|\Phi\|_\infty = \|u\|_\xi$. The projected section $\Phi^s := \mathbb{P}^s_{V,E,\varepsilon}\Phi$ belongs to $\mathcal{S}_{V,E,\varepsilon}$. Because the projection acts fiberwise, evaluating at the basepoint yields $\Phi^s(\theta_0) = \mathbb{P}^s_{V,E,\varepsilon}(\theta_0)u = u$.

By Lemma \ref{prop:growth}, the fiberwise translation translates to the section norm as follows:
$$
\|S^n u\|_\xi = \bigl\|(\mathcal{L}^n\Phi^s)(T^n\theta_0)\bigr\|_\xi \leq \|\mathcal{L}^n\Phi^s\|_\infty \leq C r_s^n\|\Phi^s\|_\infty \leq C\|\mathbb{P}^s_{V,E,\varepsilon}\|_{\mathfrak{B}(\mathcal{M}_{V,E,\varepsilon}^b)} r_s^n\|u\|_\xi.
$$
Setting a uniform constant $C_0 := C \max_{\sharp \in \{s,c,u\}} \|\mathbb{P}^\sharp_{V,E,\varepsilon}\|_{\mathfrak{B}(\mathcal{M}_{V,E,\varepsilon}^b)}$, we obtain the contraction bound $\|S^n u\|_\xi \le C_0 r_s^n \|u\|_\xi$. The rest estimates follows the same line.  \qed

\subsection{The Dual-Bridge Framework}
With the fibered solution bundle and the associated Riesz projections rigorously established, we now perform the core infinite-to-finite reduction outlined in Figure \ref{fig:logic_architecture}. The subsequent analysis constructs the Dual-Bridge framework. First, the \emph{Algebraic Bridge}  reduces the center dynamics of the $l$-truncated system to a finite-dimensional companion cocycle, yielding an exact dimension count. Subsequently, the \emph{Analytic Bridge} lifts the associated spectral projections to a fixed section space $\mathcal{BS}_\xi$, providing the uniform resolvent convergence necessary to propagate this finite dimension to the infinite-range limit.

\subsubsection{The Algebraic Bridge}

\begin{proposition}\label{thm:bridge1}
Assume $l\in [l_0,\infty)$. Define the window map evaluating the central $2l$ coordinates:
$$
\Pi_{[-l,l-1]}(\theta) : \mathcal M_{V_l,E,\varepsilon}^c(\theta) \to \C^{2l},
\qquad
\Pi_{[-l,l-1]}(\theta)u := 
\begin{pmatrix}
u_{l-1}, \dots, u_{-l}
\end{pmatrix}^\top.
$$
For $v=(v_{l-1},\dots,v_{-l})^\top\in\C^{2l}$, define the induced weighted norm 
\begin{equation}\label{weight-norm}
\|v\|_{\xi}:=\sum_{j=-l}^{l-1}|v_j|e^{-\xi|j|},
\end{equation}
and denote  $\mathcal C_{2l}(\theta) := \Pi_{[-l,l-1]}(\theta)\mathcal M_{V_l,E,\varepsilon}^c(\theta) \subset \C^{2l}$. Then the following properties hold:

\begin{enumerate}
    \item The map 
    $
    \Pi_{[-l,l-1]}(\theta) : \mathcal M_{V_l,E,\varepsilon}^c(\theta) \xrightarrow{\cong} \mathcal C_{2l}(\theta)
    $ is a  linear isomorphism.

    \item There exists a universal constant $C_*>0$, independent of $l$, such that the subspace $\mathcal C_{2l}(\theta)$ is exactly characterized by the absolute center growth bounds:
    $$
    \mathcal C_{2l}(\theta)   =
    \left\{v\in\C^{2l}:\ 
    \begin{aligned}
        &\|(A_\varepsilon^l)_n(E,\theta)v\|_{\xi}\le C_* r_+^n\|v\|_{\xi},\\ 
    &\|(A_\varepsilon^l)_{-n}(E,\theta)v\|_{\xi}\le C_* r_-^{-n}\|v\|_{\xi},
    \end{aligned}\ 
    \forall n\ge0
    \right\}.
    $$
    
    \item  The window map exactly conjugates the spatial shift to the transfer matrix:
    $$
    \Pi_{[-l,l-1]}(T\theta)\circ \mathscr T\big|_{\mathcal M_{V_l,E,\varepsilon}^c(\theta)} =
    A_{\varepsilon}^{l}(E,\theta) \circ \Pi_{[-l,l-1]}(\theta).
    $$ 
    Consequently, $\mathcal C_{2l}(\theta)$ forms an invariant subspace for the cocycle:
    \begin{equation}\label{in-cocy}
         A_{\varepsilon}^{l}(E,\theta)\mathcal C_{2l}(\theta) = \mathcal C_{2l}(T\theta).
    \end{equation}
\end{enumerate}
\end{proposition}

\begin{proof}
(1) We first prove that $\Pi_{[-l,l-1]}(\theta)$ is injective. If $u\in \mathcal M_{V_l,E,\varepsilon}^c(\theta)$ and $\Pi_{[-l,l-1]}(\theta) u=0$, then $u$ is a solution to a $2l$-order linear difference equation containing $2l$ consecutive zeros. This forces $u\equiv 0$. Surjectivity is immediate from the definition of $\mathcal C_{2l}(\theta)$. Thus $\Pi_{[-l,l-1]}(\theta)$ is a linear isomorphism.

(2) Take $u=(u_n)_{n\in\Z}\in \mathcal{M}_{V_l,E,\varepsilon}^c(\theta)$, and set $v:=\Pi_{[-l,l-1]}(\theta) u$. Since $u$ lies  in the center space, Theorem \ref{thm:sec7final} (see also Remark \ref{uniform-constant}) guarantees the existence of a constant $C>0$ (independent of $l$) such that $|u_m| \le \|S^m u\|_\xi \le C r_+^m\|u\|_\xi$ for $m \ge 0$, and $|u_{-m}| \le C r_-^{-m}\|u\|_\xi$ for $m \ge 0$. Splitting the norm yields:
$$
\begin{aligned}
\|u\|_{\xi} &= \sum_{m=-l}^{l-1}|u_m|e^{-\xi|m|} + \sum_{m\ge l}|u_m|e^{-\xi m} + \sum_{m\ge l+1}|u_{-m}|e^{-\xi m}\\
&\le \|v\|_{\xi} + C\|u\|_{\xi}\sum_{m\ge l}(r_+e^{-\xi})^m + C\|u\|_{\xi}\sum_{m\ge l+1}(r_-^{-1}e^{-\xi})^m\\
&= \|v\|_{\xi}+\delta_l\|u\|_{\xi},
\end{aligned}
$$
where $\delta_l := C\sum_{m\ge l}(r_+e^{-\xi})^m + C\sum_{m\ge l+1}(r_-^{-1}e^{-\xi})^m$. Since $e^{-\xi} < \min(r_+^{-1}, r_-)$, the series converge and $\delta_l \to 0$ as $l \to \infty$. Therefore, there exists $l_*>0$ such that for all $l\geq l_*$, we have $\delta_l \le 1/2$, which implies $\|u\|_{\xi}\leq 2 \|v\|_{\xi}$.

For the finitely many truncations $l_0\leq l<l_*$, the map $\Pi_{[-l,l-1]}(\theta)$ is a linear isomorphism between finite-dimensional spaces, yielding a constant $C_l>0$ such that $\|u\|_{\xi}\le C_l\|v\|_{\xi}$. Defining $C_0:=\max\{2, C_{l_0},\dots,C_{l_*-1}\}$ explicitly ensures that 
$
\|v\|_{\xi} \leq  \|u\|_{\xi} \le C_0\,\|v\|_{\xi}
$
holds uniformly for all $l\ge l_0$.

For each $n\in\Z$, define $U_n:= \Pi_{[-l,l-1]}(T^n\theta)(S^n u)$. As $u \in \mathcal{M}_{V_l,E,\varepsilon}^c(\theta)$,  it follows that  $U_n = (A^{l}_\varepsilon)_{n}(E,\theta)v$. Thus,
$$
\|(A_\varepsilon^l)_n(E,\theta)v\|_{\xi} = \|U_n\|_{\xi} = \sum_{j=-l}^{l-1}|u_{n+j}|e^{-\xi|j|} \le \|S^n u\|_{\xi} \le C r_+^n\|u\|_{\xi} \leq C C_0 r_+^n\|v\|_{\xi},  n\ge0.
$$
A symmetric argument yields $\|(A_\varepsilon^l)_{-n}(E,\theta)v\|_{\xi} \le C C_0 r_-^{-n}\|v\|_{\xi}$. Setting $C_* := C C_0$ confirms that $\mathcal C_{2l}(\theta)$ is contained in the characterizing set.

Conversely, assume $v \in \C^{2l}$ satisfies the exponential bounds with constant $C_*$. Defining $V_n := (A^{l}_\varepsilon)_{n}(E,\theta)v$ generates a unique bi-infinite sequence $u$ satisfying $\Pi_{[-l,l-1]}(\theta) u=v$. Extracting the $0$-th coordinate of $V_n$ gives $|u_n| \le \|V_n\|_\xi \le C_* r_+^n\|v\|_\xi$ for $n \ge 0$, and similarly $|u_{-n}| \le C_* r_-^{-n}\|v\|_\xi$. For $n \ge 0$, we have
$$
\begin{aligned}
\|S^nu\|_{\xi} &= \sum_{k\in\Z} |u_{n+k}|e^{-\xi |k|} = \sum_{m\ge 0}|u_m|e^{-\xi|m-n|} + \sum_{m < 0}|u_m|e^{-\xi|m-n|}\\
&\le C_*\|v\|_\xi \sum_{m\ge 0} r_+^m e^{-\xi|m-n|} + C_*\|v\|_\xi \sum_{j\ge 1} r_-^{-j} e^{-\xi(n+j)},
\end{aligned}
$$
where we substituted $j = -m \ge 1$ in the second sum. Using the above bounds and $e^{-\xi} < \min(r_+^{-1}, r_-)$, we obtain
$$
\|S^nu\|_{\xi} \le C_* r_+^n \|v\|_\xi \left( \sum_{m\ge0}(r_+e^{-\xi \text{sgn}(m-n)})^{|m-n|} + \sum_{j\ge1}(r_-^{-1}e^{-\xi})^j \right) \le \tilde{C} r_+^n \|v\|_\xi.
$$
Symmetrically, $\|S^{-n}u\|_\xi \le \tilde{C} r_-^{-n}\|v\|_\xi$. Since $\tilde{C}$ is independent of $n$, 
$u$ possesses no stable or unstable components  as described in  Theorem \ref{thm:sec7final}, forcing $u\in \mathcal M_{V_l,E,\varepsilon}^c(\theta)$. Hence 
$v \in \mathcal C_{2l}(\theta)$, concluding (2).

(3) Let $u\in \mathcal M_{V_l,E,\varepsilon}^c(\theta)$. Evaluating the window map at the shifted base point yields:
$$
\Pi_{[-l,l-1]}(T\theta)(\mathscr T u) = A_{\varepsilon}^{l}(E,\theta)\Pi_{[-l,l-1]}(\theta) u.
$$
Since $\mathscr T\bigl(\mathcal M_{V_l,E,\varepsilon}^c(\theta)\bigr) = \mathcal M_{V_l,E,\varepsilon}^c(T\theta)$, this immediately implies \eqref{in-cocy}.
\end{proof}

Proposition \ref{thm:bridge1}  reduces the center dynamics to the finite-dimensional invariant subspace $\mathcal C_{2l}(\theta)$, actually governed  by the restricted companion cocycle $A_{\varepsilon}^{l}(E,\theta)$. This  reduction motivates the analysis in Section \ref{PHFC}, where we establish the partial hyperbolicity of the cocycle and explicitly compute the dimension of the center bundle.

\subsubsection{The Analytic Bridge}
 By Proposition \ref{thm:D-regular}, we define the extended center projection on the Banach space $\mathcal{BS}_\xi$ by the contour integral
\begin{equation}\label{eq:ambient-P-def}
\mathbb{ P }_{\varepsilon, c}^{l}
:=
\frac{1}{2\pi i}\int_{\Gamma_c}\mathcal R^{l}_\varepsilon(z) dz
\in \mathfrak{B}(\mathcal{BS}_\xi).
\end{equation}
 We first clarify this definition. Let $\Psi \in \mathcal{BS}_\xi$ and set $u := \mathbb{ P }_{\varepsilon,c}^{l}\Psi$. Recall the decomposition of the extended resolvent $\mathcal{R}_\varepsilon^l(z)\Psi = \mathcal{P}_z f + \mathcal{Q}_z \Psi$, where  $f(\cdot) := D_{\varepsilon,E}^{ l}(z)^{-1}\mathcal{G}_z^l(\Psi)$. By construction,  $\mathcal Q_z\Psi$ can extend holomorphically to the annular region
bounded by $\Gamma_c$.  Consequently, its contour integral over $\Gamma_c$ vanishes identically. The extended projection is thus entirely determined by the homogeneous component:
$$
u_n(\theta) = \frac{1}{2\pi i}\int_{\Gamma_c} z^n f(z,T^n\theta) dz, \qquad n\in\Z.
$$
The next lemma gives a sufficient condition under which this sequence
is in fact a center solution. This criterion will be used below for the extended projection, and later for the
explicit construction of center sections.

\begin{lemma}\label{recons}
 Let $F(z,\theta)$ be defined for
$z\in\Gamma_c$ and $\theta\in\Omega$. Assume that
$
f(z,\cdot):=D^{(l)}_{E,\varepsilon}(z)^{-1}F(z,\cdot) 
$ is well-defined and 
$ \sup_{z\in\Gamma_c}\|f(z,\cdot)\|_{B(\Omega,\C)}<\infty.$
Assume moreover that, for every $n\in\mathbb Z$ and every
$\theta\in\Omega$,
$$
\frac{1}{2\pi i}\int_{\Gamma_c}
z^nF(z,T^n\theta)dz=0.
$$ Define
$$
u_n(\theta):=\frac{1}{2\pi i}\int_{\Gamma_c}
z^n f(z,T^n\theta)dz,
\qquad n\in\mathbb Z.
$$
Then $
u(\theta)\in \mathcal M^c_{V_l,E,\varepsilon}(\theta).
$
\end{lemma}
\begin{proof}
    A direct calculous shows that 
$$
\begin{aligned}
&(\mathbf L_{ V_l,\varepsilon W,T,\theta}u)_n(\theta)-Eu_n(\theta)\\
&=\frac{1}{2\pi i}\int_{\Gamma_c}
z^n\left[
\sum_{k=-l}^{l}\hat v_k z^k
f(z,T^{n+k}\theta)
+\varepsilon W(T^n\theta)f(z,T^n\theta)-
E f(z,T^n\theta)
\right]dz  \\
&=\frac{1}{2\pi i}\int_{\Gamma_c}z^n
\left(D_{\varepsilon,E}^{(l)}(z)f(z,\cdot)
\right)(T^n\theta)
dz=\frac{1}{2\pi i}\int_{\Gamma_c}
z^n F(z,T^n\theta)
dz=0 .
\end{aligned}
$$
Therefore
$
\mathbf L_{V_l,\varepsilon W,T,\theta}u(\theta)=
E u(\theta).
$

It remains to prove the center growth estimates. 
 For $n\ge 1$,
\begin{align*}
|u_n(\theta)|
&\le
\frac{1}{2\pi}\int_{\Gamma_c} |z|^n\,|f(z,T^n\theta)|\,|dz|\le
\left( r_+ r_+^n+ r_- r_-^n\right)
\sup_{z\in\Gamma_c}\|f(z,\cdot)\|_{B(\Omega,\C)}\\ &\leq C r_+^n.
\end{align*}
 Similarly, $|u_{-n}(\theta)|\leq  C r_-^{-n}$ .
 Therefore, $u(\theta)\in \mathcal M^c_{V_l,E,\varepsilon}(\theta).$
\end{proof}

 The following result establishes the structural properties and the uniform norm convergence of these extended operators, which serve as the analytical foundation for the subsequent dimension reduction.

\begin{proposition}\label{thm:ambient-projection}
For every $l\in[l_0,\infty]$,  $\mathbb{ P }_{\varepsilon, c}^{l}$ satisfies the following properties:
\begin{enumerate}
\item The restriction of $\mathbb{ P }_{\varepsilon, c}^{l}$ to the section space $\mathcal M_{V_l,E,\varepsilon}^b$ coincides with the intrinsic Riesz projections:    $
    \mathbb P_{\varepsilon,c}^{l}\big|_{\mathcal M_{V_l,E,\varepsilon}^b} = \mathbb P_{V_l,E,\varepsilon}^c.
    $
    \item  
    $
    \operatorname{Ran}(\mathbb{ P }_{\varepsilon,c}^{l}) \subset \mathcal{C}_{V_l,E,\varepsilon}.
    $
    \item  $
\|\mathbb{ P }_{\varepsilon,c}^{\infty}-\mathbb{ P }_{\varepsilon,c}^{l}\|_{\mathfrak{B}( \mathcal{BS}_\xi)}\to 0,
 l\to\infty.
$
\end{enumerate}
\end{proposition}
\begin{proof}
(1) follows directly from Proposition \ref{thm:D-regular}. 

(2)  Let $\Psi \in \mathcal{BS}_\xi$. Since $z^n\mathcal G_z^{(l)}(\Psi)(T^n\theta)$ can
extend holomorphically to the annular region bounded by $\Gamma_c$. Hence, by
Cauchy's theorem,
$$
\frac{1}{2\pi i}\int_{\Gamma_c}
z^n\mathcal G_z^{(l)}(\Psi)(T^n\theta)dz=0.
$$ Therefore Lemma \ref{recons} implies
$
u:= \mathbb{ P }_{\varepsilon,c}^{l}\Psi\in  \mathcal{C}_{V_l,E,\varepsilon}.
$

(3) For every $z\in\Gamma_c$, we recall the definition of  $\mathcal R^{l}_\varepsilon(z)$ given in \eqref{def:rec},
$$
\mathcal R^{l}_\varepsilon(z)-\mathcal R^{\infty}_{\varepsilon}(z)=
\mathcal{P}_z\Bigl(
(D_{\varepsilon,E}^{ l}(z))^{-1}\mathcal{G}_z^l-
D^{\infty}_{\varepsilon,E}(z)^{-1}\mathcal{G}_z^\infty
\Bigr),
$$
because the operator $\mathcal{Q}_z$ is the same in both constructions.

Hence we have
\begin{align*}
\|\mathcal R_{\varepsilon}^{\infty}(z)-\mathcal R^{l}_\varepsilon(z)\|
&\leq
\|\mathcal{P}_z\| \Bigl\|
(D_{\varepsilon,E}^{ l}(z))^{-1}\mathcal{G}_z^l-
D^{\infty}_{\varepsilon,E}(z)^{-1}\mathcal{G}_z^\infty\Bigr\|
\\
&\leq
\|\mathcal{P}_z\|\Bigl(
\|(D_{\varepsilon,E}^{ l}(z))^{-1}-D^{\infty}_{\varepsilon,E}(z)^{-1}\| \|\mathcal{G}_z^\infty\|
+\|(D_{\varepsilon,E}^{ l}(z))^{-1}\| \|\mathcal{G}_z^l-\mathcal{G}_z^\infty\|\Bigr).
\end{align*}
By \eqref{equ:symcon} and the definition of $\mathcal{G}_z^l$, 
we conclude that
$
\sup_{z\in\Gamma_c}\|\mathcal R^{\infty}_{\varepsilon}(z)-\mathcal R^{l}_\varepsilon(z)\|\to 0.
$
By the definitions of $\mathbb{ P }_{\varepsilon,c}^{\infty}$ and $\mathbb{ P }_{\varepsilon,c}^{l}$,
$$
\mathbb{ P }_{\varepsilon,c}^{\infty}-\mathbb{ P }_{\varepsilon,c}^{l}=
\frac{1}{2\pi i}\int_{\Gamma_c}
\bigl(\mathcal R^{\infty}_{\varepsilon}(z)-\mathcal R^{l}_\varepsilon(z)\bigr) dz.
$$
Hence by \eqref{contour},
$
\|\mathbb{ P }_{\varepsilon,c}^{\infty}-\mathbb{ P }_{\varepsilon,c}^{l}\|
\leq(r_++r_-)
\sup_{z\in\Gamma_c}
\|\mathcal R^{\infty}(z)-\mathcal R^{l}_\varepsilon(z)\|\to0.
$
\end{proof}

\section{Partial hyperbolicity of companion cocycle}\label{PHFC}

In this section, we study the basic properties of the companion cocycle associated with the truncated long-range operator. While this finite-dimensional matrix is non-normal and suffers from coordinate-dependent norm explosions, its algebraic roots securely establish the essential spectral gap, which allows us to establish the  partially hyperbolic splitting  of the companion cocycle.
Our main result of this section is as follows.

\begin{theorem}\label{thm:sec4-main}
There exists $\varepsilon_*=\varepsilon_*(\delta_0, r_\pm,\eta_\pm)>0$ such that for all $l\in[l_0,+\infty)$ and $|\varepsilon|\le\varepsilon_*$, the cocycle $A_{\varepsilon}^{l}(E,\theta)$ admits an invariant splitting
$$
\mathbb{C}^{2l}
=
\mathcal{E}^{l}_{\varepsilon,s}(\theta)
\oplus
\mathcal{E}^{l}_{\varepsilon,c}(\theta)
\oplus
\mathcal{E}^{l}_{\varepsilon,u}(\theta),
\qquad \theta\in\Omega.
$$
The splitting is partially hyperbolic: there exists a constant $C\ge 1$, independent of $l$, such that for all $\theta\in\Omega$, $n\ge 1$, and any vectors $u_\sharp\in\mathcal{E}^{l}_{\varepsilon,\sharp}(\theta)$, $\sharp\in\{s,c,u\}$,
\begin{enumerate}
    \item $\|(A^{l}_\varepsilon)_n(E,\theta) u_s\|_\xi \le C r_s^n \|u_s\|_\xi$,
    \item $\|(A^{l}_\varepsilon)_n(E,\theta) u_u\|_\xi \ge C^{-1} r_u^n \|u_u\|_\xi$,
    \item $C^{-1} r_-^n \|u_c\|_\xi \le \|(A^{l}_\varepsilon)_n(E,\theta) u_c\|_\xi \le C r_+^n \|u_c\|_\xi$,
\end{enumerate}
 where the spectral radii satisfy \eqref{radius}. Moreover, $\dim\mathcal{E}^{l}_{\varepsilon,c}(\theta)=2m(E)$.
\end{theorem}

\begin{remark}We give two remarks concerning this theorem:
\begin{enumerate}
    \item By the Algebraic Bridge (Proposition \ref{thm:bridge1}), one  proves that
$
\mathcal{E}^{l}_{\varepsilon,c}(\theta)=\mathcal{C}_{2l}(\theta),
$
and hence it establishes a  linear isomorphism $\mathcal M_{V_l,E,\varepsilon}^c(\theta) \cong \mathcal{E}^{l}_{\varepsilon,c}(\theta)$. Consequently, $\dim\mathcal M_{V_l,E,\varepsilon}^c(\theta) = 2m(E)$ for $l\in[l_0,\infty)$.
\item By the stability of partial hyperbolicity, $(T, A^{l}_\varepsilon(E,\cdot))$ remains partially hyperbolic for $\varepsilon$ sufficiently small. Crucially, our analysis ensures that this smallness condition on $\varepsilon$ is  uniform in the truncation parameter $l$.
\end{enumerate}

\end{remark}

\subsection{Resolvent estimate of companion matrix}

The following lemma provides the explicit resolvent formula of the unperturbed cocycle.

\begin{lemma}
\label{lem:unperturbed-resolvent}
Assume $L^{l}_{E}(z)\neq 0$. Then $z-A_{0}^{l}(E)$ is invertible. Set
$$
(z-A_{0}^{l}(E))^{-1} = \bigl[r_{i,j}^0(z)\bigr]_{i,j=1}^{2l}.
$$
Then for $i,j\in\{1,\dots,2l\}$, we have
$$
r_{i,j}^0(z)=
\begin{cases}
z^{-(i-j)}r_{j,j}^0(z), & i\ge j,\\
z^{j-i}r_{j,j}^0(z)-z^{j-i-1}, & i<j,
\end{cases}
\qquad 
r_{j,j}^0(z) = \frac{T_j(z)}{L^{l}_{E}(z)},
$$
where  
$$
T_j(z)=\sum_{k=l-j+1}^{l}\widetilde v_k z^{k-1}, \qquad \widetilde v_k=
\begin{cases}
\hat{v}_k, & k\neq 0,\\
\hat{v}_0-E, & k=0.
\end{cases}
$$
\end{lemma}
The proof is mainly linear algebra, which we will leave it in Appendix \ref{appendixa}.

\begin{remark}\label{abnormal}
By Lemma \ref{lem:unperturbed-resolvent}, the upper-triangular terms  are given by $r_{i,j}^0(z) = z^{j-i}r_{j,j}^0(z) - z^{j-i-1}$. Summing these absolute values over $1 \le i < j$ yields a column sum that grows at least linearly with the truncation size $l$ for $|z| \approx 1$ (and exponentially if $|z| > 1$). Therefore, the  unweighted resolvent norm $\|(z - A_{0}^{l}(E))^{-1}\|_{1}$ diverges as $l \to \infty$. 
\end{remark}

The above discussions also motivate us to use the weighted norm. 
Recall that for $v\in\mathbb{C}^{2l}$, $\|v\|_\xi = \sum_{j=-l}^{l-1} |v_j| e^{-\xi|j|}$. 
Define
$$
R_l = \operatorname{diag}\bigl(e^{-\xi|l-1|}, e^{-\xi|l-2|}, \dots, e^{-\xi|-l|}\bigr),
$$
so that $\|v\|_\xi = \|R_l v\|_{1}$. 
For any matrix $M$ on $\mathbb{C}^{2l}$, the induced operator norm is therefore
\begin{equation}\label{equ:induced}
\|M\|_\xi = \|R_l M R_l^{-1}\|_{1}.
\end{equation}
Under the induced norm $\|\cdot\|_\xi$, the divergent powers $z^{p-q}$ are  rescaled by $e^{-\xi|p|}e^{\xi|q|}z^{p-q}$, these exponential weights  suppress the off-diagonal growth into summable series. 
As a consequence,  the  weighted resolvent norm $\|(z-R_l A_{0}^{l}(E)R_l^{-1})^{-1}\|_{1}$  converges: 

\begin{lemma}
\label{prop:weighted}
There exists $C_r = C_r(\delta_0, r_\pm,\eta_\pm)>0$ such that for all $l\in[l_0,+\infty)$ and $z\in\mathcal{A}(E)$,
\begin{equation}\label{prop:inverse1}
\|(z-R_l A_{0}^{l}(E)R_l^{-1})^{-1}\|_{1}\le C_r.
\end{equation}
Moreover, with $\mathbf e_1$ the first standard basis vector of $\mathbb{C}^{2l}$, 
\begin{equation}\label{prop:inverse2}
\sup_{z\in\mathcal{A}(E)}\|(z-R_l A_{0}^{l}(E)R_l^{-1})^{-1}\mathbf e_1\|_{1}\le C_r|\hat{v}_l|e^{\xi(l-1)}.
\end{equation}
\end{lemma}
\begin{proof}
We first prove \eqref{prop:inverse1}. We assume $z\in\mathcal{A}_+(E)$, the other case can be estimated similarly. Set $r_1:=r_+(E) e^{-\eta_+(E)}$ and $r_2:=r_+(E) e^{\eta_+(E)}$. Then $e^{-\xi}<r_1\le |z|\le r_2<e^\xi$ for all $z\in\mathcal A_+(E)$.

Let $(z-A_{0}^{l}(E))^{-1} = \bigl[r_{i,j}^0(z)\bigr]_{i,j=1}^{2l}$. Define the weighted resolvent as
$
\widehat R^{l}_0(z):= R_l (z-A_{0}^{l}(E))^{-1}R_l^{-1}.
$
Its entries are given by
$$
\widehat r_{i,j}^0(z) := e^{-\xi|l-i|}r_{i,j}^0(z)e^{\xi|l-j|}, \qquad i,j\in\{1,\dots,2l\}.
$$

Fix $j\in\{1,\dots,2l\}$, and set $q:=l-j$. By Lemma \ref{lem:unperturbed-resolvent},
$$
r_{j,j}^0(z)=\frac{T_q^+(z)}{L^{l}_{E}(z)}, \qquad T_q^+(z):=\sum_{k=q+1}^{l}\widetilde v_k z^{k-1},
$$
and
$$
zr_{j,j}^0(z)-1 = -\frac{T_q^-(z)}{L^{l}_{E}(z)}, \qquad T_q^-(z):=\sum_{k=-l}^{q}\widetilde v_k z^k.
$$

For each $i\in\{1,\dots,2l\}$, set $p:=l-i$. Then $\widehat r_{i,j}^0(z) = e^{-\xi|p|}r_{i,j}^0(z)e^{\xi|q|}$. Let $s:=|z|$. We have $r_1\le s\le r_2$.

\textbf{1. Bounds on $T_q^+(z)$ and $T_q^-(z)$.}

Since $V$ is analytic, we have $|\widetilde v_k|\le C_0 e^{-h|k|}$ for $|k|\le l$, for some constant $C_0=C_0(V,E)$.

If $q\ge 0$, then
\begin{align}
|T_q^+(z)| \le \sum_{k=q+1}^{l} |\widetilde v_k||z|^{k-1} \le C_0\sum_{k=q+1}^{\infty} e^{-h k}s^{k-1} 
= C_0 s^{-1}\sum_{k=q+1}^{\infty}(s e^{-h})^k \le C_1 (s e^{-h})^q .
\label{eq:Tplus}
\end{align}
Also,
\begin{align}
|T_q^-(z)| &\le \sum_{k=-l}^{-1} |\widetilde v_k||z|^k+|\widetilde v_0| +\sum_{k=1}^{q}|\widetilde v_k||z|^k 
\le C_0\sum_{k=-\infty}^{-1}(s e^h)^k +|\widetilde v_0| +C_0\sum_{k=1}^{\infty}(s e^{-h})^k \le C_2 .
\label{eq:Tminus}
\end{align}

If $q<0$, then
\begin{align}
|T_q^+(z)| &\le \sum_{k=q+1}^{-1}|\widetilde v_k||z|^{k-1} +|\widetilde v_0||z|^{-1} +\sum_{k=1}^{l}|\widetilde v_k||z|^{k-1} \nonumber\\
&\le C_0 s^{-1}\sum_{k=-\infty}^{-1}(s e^h)^k +|\widetilde v_0|s^{-1} +C_0 s^{-1}\sum_{k=1}^{\infty}(s e^{-h})^k \le C_3 ,
\label{eq:Tplusn}
\end{align}
and
\begin{align}
|T_q^-(z)| &\le C_0\sum_{k=-\infty}^{q} e^{h k}s^k = C_0\sum_{k=-\infty}^{q}(s e^h)^k \le C_4(s e^h)^q .
\label{eq:Tminusn}
\end{align}
Since $r_1\le s\le r_2$, the constants $C_1,C_2,C_3,C_4$ may be chosen depending only on $v,h,r_\pm,\eta_\pm,E$, independent of $l,q,z$.

\textbf{2. Estimate of the part $i\ge j$ (equivalently $p\le q$).}

If $i\ge j$, then by Lemma \ref{lem:unperturbed-resolvent},
$$
r_{i,j}^0(z) = z^{-(i-j)}\frac{T_q^+(z)}{L^{l}_{E}(z)} = z^{p-q}\frac{T_q^+(z)}{L^{l}_{E}(z)}.
$$
Using $|L^{l}_{E}(z)|\ge \delta_0$, we obtain
$$
|\widehat r_{i,j}^0(z)| \le \delta_0^{-1}e^{-\xi|p|}s^{p-q}e^{\xi|q|}|T_q^+(z)|.
$$

If $q\ge 0$, then by \eqref{eq:Tplus},
$$
|\widehat r_{i,j}^0(z)| \le C_r e^{-\xi|p|}s^{p-q}e^{\xi q}(s e^{-\xi})^q = C_r e^{-\xi|p|}s^p.
$$
Therefore,
\begin{align*}
\sum_{p\le q} |\widehat r_{i,j}^0(z)| \le C_r\sum_{p=-l}^{l-1} e^{-\xi|p|}s^p \le C_r\left( \sum_{p=0}^{\infty}(r_2e^{-\xi})^p + \sum_{m=1}^{\infty}(r_1^{-1}e^{-\xi})^m \right) \le C_r.
\end{align*}

If $q<0$, then by \eqref{eq:Tplusn},
$$
|\widehat r_{i,j}^0(z)| \le C_r e^{-\xi|p|}s^{p-q}e^{-\xi q} = C_r(se^\xi)^{p-q}.
$$
Therefore,
$$
\sum_{p\le q} |\widehat r_{i,j}^0(z)| \le C_r\sum_{m=0}^{\infty}(r_1e^\xi)^{-m} \le C_r.
$$
Thus, $\sum_{i\ge j} |\widehat r_{i,j}^0(z)|\le C_r$ uniformly in $j,l,z$.

\textbf{3. Estimate of the part $i<j$ (equivalently $p>q$).}

If $i<j$, then
$$
r_{i,j}^0(z) = z^{j-i}r_{j,j}^0(z)-z^{j-i-1} = z^{p-q-1}\bigl(zr_{j,j}^0(z)-1\bigr) = -z^{p-q-1}\frac{T_q^-(z)}{L^{l}_{E}(z)}.
$$
Therefore,
$$
|\widehat r_{i,j}^0(z)| \le \delta_0^{-1}e^{-\xi|p|}s^{p-q-1}e^{\xi|q|}|T_q^-(z)|.
$$

If $q\ge 0$, then by \eqref{eq:Tminus},
$$
|\widehat r_{i,j}^0(z)| \le C_r e^{-\xi|p|}s^{p-q-1}e^{\xi q} = C_r s^{-1}(s e^{-\xi})^{p-q}.
$$
Therefore,
$$
\sum_{p>q} |\widehat r_{i,j}^0(z)| \le C_r r_1^{-1}\sum_{m=1}^{\infty}(r_2e^{-\xi})^m \le C_r.
$$

If $q<0$, then by \eqref{eq:Tminusn},
$$
|\widehat r_{i,j}^0(z)| \le C_r e^{-\xi|p|}s^{p-q-1}e^{-\xi q}(s e^\xi)^q = C_r e^{-\xi|p|}s^{p-1}.
$$
Hence,
$$
\sum_{p>q}|\widehat r_{i,j}^0(z)| \le \sum_{\substack{p>q, p\ge 0}}|\widehat r_{i,j}^0(z)| + \sum_{\substack{p>q, p<0}}|\widehat r_{i,j}^0(z)|.
$$
For $p\ge 0$, 
$$
\sum_{\substack{p>q, p\ge 0}}|\widehat r_{i,j}^0(z)| \le C_r r_1^{-1}\sum_{p=0}^{\infty}(r_2e^{-\xi})^p \le C_r.
$$
For $p<0$, 
$$
\sum_{\substack{p>q, p<0}}|\widehat r_{i,j}^0(z)| \le C_r r_1^{-1}\sum_{p=-\infty}^{-1}(r_1e^\xi)^p \le C_r.
$$
Combining these parts, we obtain $\sum_{i<j} |\widehat r_{i,j}^0(z)|\le C_r$ uniformly in $j,l,z$. 

Therefore, $\|\widehat R^{l}_0(z)\|_{1}\le C_r$, which proves \eqref{prop:inverse1}.
We next prove \eqref{prop:inverse2}. By Lemma \ref{lem:unperturbed-resolvent}, for $j=1$ one has $r_{i,1}^{0}(z)=z^{-(i-1)}r_{1,1}^{0}(z)$ for $1\le i\le 2l$, and $r_{1,1}^{0}(z)=\frac{\hat{v}_l z^{l-1}}{L^{l}_{E}(z)}$. Therefore,
$
r_{i,1}^{0}(z)=\frac{\hat{v}_lz^{l-i}}{L^{l}_{E}(z)}.
$
Defining $p:=l-i\in\{-l,\dots,l-1\}$, we have $r_{i,1}^{0}(z)=\frac{\hat{v}_lz^{p}}{L^{l}_{E}(z)}$. The weighted vector entries are
$
(\widehat R^{l}_0(z)\mathbf e_1)_i = e^{-\xi|p|}r_{i,1}^{0}(z)e^{\xi(l-1)}.
$
Hence,
$$
|(\widehat R^{l}_0(z)\mathbf e_1)_i| \le e^{-\xi|p|} \frac{|\hat{v}_l||z|^p}{|L^{l}_{E}(z)|} e^{\xi(l-1)}.
$$
This yields
\begin{align*}
\|\widehat R^{l}_0(z)\mathbf e_1\|_{1} &= \sum_{i=1}^{2l} |(\widehat R^{l}_0(z)\mathbf e_1)_i| \le \frac{|\hat{v}_l|e^{\xi(l-1)}}{|L^{l}_{E}(z)|} \sum_{p=-l}^{l-1} e^{-\xi|p|}|z|^p.
\end{align*}
Since $z\in \mathcal A_+(E)$, we have
\begin{align*}
\sum_{p=-l}^{l-1} e^{-\xi|p|}|z|^p &\le \sum_{p=0}^{\infty}(r_2e^{-\xi})^p + \sum_{m=1}^{\infty}((r_1)^{-1}e^{-\xi})^m =: C_r'.
\end{align*}
Using $|L^{l}_{E}(z)|\ge \delta_0$, we obtain
$
\|\widehat R^{l}_0(z)\mathbf e_1\|_{1} \le \frac{C_r'}{\delta_0}|\hat{v}_l|e^{\xi(l-1)}.
$
By absorbing constants, this proves \eqref{prop:inverse2}.
\end{proof}

\subsection{Weighted transfer operators}
Following the spectral decomposition of the infinite-dimensional transfer operator on $\mathcal{M}_{V,E,\varepsilon}(\theta)$ in Section \ref{sec:ambient_framework}, this subsection analyzes the finite-range truncated model. Its induced finite-dimensional cocycle serves as the exact algebraic counterpart to the global transfer operator.

For $l\in[l_0,\infty)$, define the bounded section space as 
$$
\mathcal{BS}_l = \bigl\{ \Phi:\Omega\to \C^{2l} \text{ bounded} \bigr\},
\qquad
\|\Phi\|_{\mathcal{BS}_l} = \sup_{\theta\in\Omega}\|\Phi(\theta)\|_{1}.
$$
The weighted cocycle matrices are
$$
\widehat A_{\varepsilon}^{l}(E,\theta) = R_l A_{\varepsilon}^{l}(E,\theta)R_l^{-1},
$$
and the associated transfer operators $\mathcal{L}^{l}_{\varepsilon,E}$ act on $\mathcal{BS}_l$ by
$$
(\mathcal{L}^{l}_{\varepsilon,E}\Phi)(\theta) = \widehat A^{l}_{\varepsilon}(E,T^{-1}\theta)\Phi(T^{-1}\theta).
$$
In particular, $\mathcal{L}^{l}_{0,E}\Phi(\theta) = \widehat A_{0}^{l}(E)\Phi(T^{-1}\theta)$. 
The goal of this subsection is to show the invertibility of the transfer operator.
Formally, 
\begin{equation}\label{eq:reso}
    (z-\mathcal{L}^{l}_{\varepsilon,E})^{-1} = \bigl(\mathrm{Id} - (z-\mathcal{L}^{l}_{0,E})^{-1}(\mathcal{L}^{l}_{\varepsilon,E} - \mathcal{L}^{l}_{0,E})\bigr)^{-1}(z-\mathcal{L}^{l}_{0,E})^{-1},
\end{equation}
while boundedness  $(z-\mathcal{L}^{l}_{0,E})^{-1}$ is essentially given by Lemma \ref{prop:weighted}, the key is to prove $ (z-\mathcal{L}^{l}_{0,E})^{-1}(\mathcal{L}^{l}_{\varepsilon,E} - \mathcal{L}^{l}_{0,E})$ is actually uniformly bounded.

\begin{proposition}
\label{prop:improved}
There exist $\varepsilon_*>0$ and $M=M(\delta_0, r_\pm, \eta_\pm)>0$,  such that for all $l\in[l_0,\infty)$, $|\varepsilon|\le\varepsilon_*$, and $z\in\mathcal{A}'(E)$,
$\bigl\|(z-\mathcal{L}^{l}_{\varepsilon,E})^{-1}\bigr\|_{\mathfrak{B}(\mathcal{BS}_l )} \le M.$
\end{proposition}

\begin{proof}
We detail the proof for the outer annulus component $z\in\mathcal A'_+(E)$ defined in \eqref{eq:a+}, the inner annulus case follows by identical reasoning. Fix such a $z$, we first show that $z-\mathcal{L}^{l}_{0,E}$ is invertible for all $l\in[l_0,\infty)$. Let
$$
\mathbb{A} := \{\lambda\in\mathbb{C}: e^{-\eta_+(E)/2} < |\lambda| < e^{\eta_+(E)/2}\}.
$$
By Lemma~\ref{prop:weighted}, there exists $C=C(\delta_0, r_\pm,\eta_\pm)>0$, independent of $l$, such that for all $l\in[l_0,+\infty)$,
\begin{equation}\label{eq:Est-lambda}
    \sup_{\lambda\in\mathbb{A}} \bigl\| (z-\lambda\widehat A_{0}^{l}(E))^{-1} \bigr\|_{1} \le C, \qquad\sup_{\lambda\in\mathbb{A}} \bigl\| (z-\lambda\widehat A_{0}^{l}(E))^{-1} \mathbf{e}_1\bigr\|_{1} \le C|v_l| e^{\xi(l-1)}. 
\end{equation}
Hence, the function $\lambda\mapsto (z-\lambda\widehat A_{0}^{l}(E))^{-1}$ is analytic on $\mathbb{A}$ and admits a Laurent expansion
\begin{equation}\label{eq:Laurent}
    (z-\lambda\widehat A_{0}^{l}(E))^{-1} = \sum_{m\in\mathbb{Z}} K_m^{l}(z)\,\lambda^m, \qquad \lambda\in\mathbb{A},
\end{equation}
where each $K_m^{l}(z)$ is a $2l\times 2l$ matrix.
A standard Cauchy estimate gives $\|K_m^{l}(z)\|_{1} \le C e^{-\eta_+(E)|m|/4}$ for all $m\in\mathbb Z$. Summing over $m$ yields
\begin{equation}\label{eq:est-Kml}
\sum_{m\in\mathbb Z}\|K_m^{l}(z)\|_{1}
\le C\sum_{m\in\mathbb Z} e^{-\eta_+(E)|m|/4}
\le \frac{2C}{1-e^{-\eta_+(E)/4}} =: M,
\end{equation}
with $M$ independent of $l$.

Define
$\widehat{\mathcal{R}^{l}_{0}} = \sum_{m\in\mathbb Z}K_m^{l} \mathbf T^m ,$
where $(\mathbf T\Phi)(\theta)=\Phi(T^{-1}\theta)$ for $\Phi\in\mathcal{BS}_l$. Note $\mathbf T$ is an isometry, \eqref{eq:est-Kml} gives $$
\|\widehat{\mathcal{R}^{l}_{0}}\|_{\mathfrak{B}(\mathcal{BS}_l )} \le \sum_{m\in\mathbb Z}\|K_m^l\|_{1} \le M.
$$
We claim $\widehat{\mathcal{R}^{l}_{0}}=(z-\mathcal{L}^{l}_{0,E})^{-1}$. Indeed, from \eqref{eq:Laurent}, we obtain
\begin{equation*}\label{equ:laurent}
zK_m^{l} - \widehat A_{0}^{l}(E) K_{m-1}^{l} = \delta_{m0}\,\mathrm{Id}, \qquad m\in\mathbb Z.
\end{equation*}
A direct computation then yields
$$
\begin{aligned}
(z-\mathcal{L}^{l}_{0,E})\widehat{\mathcal{R}^{l}_{0}}\Phi
&= (z-\widehat A_{0}^{l}(E)\mathbf T)\sum_{m\in\mathbb Z}K_m^{l} \mathbf T^m\Phi \\
&= \sum_{m\in\mathbb Z}zK_m^{l}\mathbf T^m\Phi - \sum_{m\in\mathbb Z}\widehat A_{0}^{l}(E)K_m^{l}\mathbf T^{m+1}\Phi \\
&= \sum_{m\in\mathbb Z}\bigl(zK_m^{l}-\widehat A_{0}^{l}(E)K_{m-1}^{l}\bigr)\mathbf T^{m}\Phi = \Phi.
\end{aligned}
$$
Thus $(z-\mathcal{L}^{l}_{0,E})\widehat{\mathcal{R}^{l}_{0}}=\mathrm{Id}$. Similarly, we have $\widehat{\mathcal{R}^{l}_{0}}(z-\mathcal{L}^{l}_{0,E})=\mathrm{Id}$.

Furthermore, define $k_m^{l}(z)=K_m^{l}(z)\mathbf e_{1}$. By \eqref{eq:Est-lambda},
a standard Cauchy estimate gives
$$
\|k_m^{l}(z)\|_{1} \le C |\hat{v}_l| e^{\xi(l-1)} e^{-\eta_+(E)|m|/4}, \qquad m\in\mathbb Z.
$$
Summing over $m$ gives
\begin{equation}
\label{eq:km}
\sum_{m\in\mathbb Z}\|k_m^{l}(z)\|_{1}
\le D |\hat{v}_l| e^{\xi(l-1)},
\end{equation}
where $D = C\sum_{m\in\mathbb Z}e^{-\eta_+(E)|m|/4}$ depends only on $\delta_0, r_\pm,\eta_\pm$ and is independent of $l$.

Note $A_{\varepsilon}^{l}(E,\theta)-A_{0}^{l}(E) = -\frac{\varepsilon W(\theta)}{\hat{v}_l} \mathbf e_{1}\mathbf e_l^\ast$, conjugating by $R_l$ gives
$$
\widehat A_{\varepsilon}^{l}(E,\theta)-\widehat A_{0}^{l}(E) = b^{l}_{\varepsilon}(\theta)\mathbf e_{1}\mathbf e_l^\ast,
$$
with
\begin{equation}
\label{eq:bepsilon}
b^{l}_{\varepsilon}(\theta) = -\varepsilon W(\theta)\frac{e^{-\xi(l-1)}}{\hat{v}_l}.
\end{equation}
Hence, for $\Phi\in\mathcal{BS}_l$,
\begin{equation}
\label{eq:global}
\bigl((\mathcal{L}^{l}_{\varepsilon,E}-\mathcal{L}^{l}_{0,E})\Phi\bigr)(\theta)
= b^{l}_{\varepsilon}(T^{-1}\theta) \mathbf e_{1} \mathbf e_l^\ast\Phi(T^{-1}\theta).
\end{equation}

Let
$
\Psi=(\mathcal{L}^{l}_{\varepsilon,E}-\mathcal{L}^{l}_{0,E})\Phi.
$
Using \eqref{eq:global}, we obtain
\begin{align*}
(\widehat{\mathcal{R}^{l}_{0}}(z)\Psi)(\theta)
&=\sum_{m\in\mathbb Z}
K_m^{l}(z)
\Bigl(b^{l}_{\varepsilon}(T^{-(m+1)}\theta)\mathbf e_{1}
\mathbf e_l^\ast \Phi(T^{-(m+1)}\theta)\Bigr)\\
&=\sum_{m\in\mathbb Z}
k_m^{l}(z)
b^{l}_{\varepsilon}(T^{-(m+1)}\theta)
\mathbf e_l^\ast \Phi(T^{-(m+1)}\theta).
\end{align*}
Hence,
\begin{equation*}
\|\widehat{\mathcal{R}^{l}_{0}}(z)(\mathcal{L}^{l}_{\varepsilon,E}-\mathcal{L}^{l}_{0,E})\Phi\|_{\mathcal{BS}_l}
\le\left(\sum_{m\in\mathbb Z}\|k_m^{l}(z)\|_{1}\right)
\sup_{\theta\in\Omega}|b^{l}_{\varepsilon}(\theta)|
\|\Phi\|_{\mathcal{BS}_l}.
\end{equation*}
By \eqref{eq:bepsilon},
$
\sup_{\theta\in\Omega}|b^{l}_{\varepsilon}(\theta)|
\le|\varepsilon|\|W\|
\frac{e^{-\xi(l-1)}}{|\hat{v}_l|}.
$
Combining this with \eqref{eq:km}, we obtain
\begin{equation*}
\|\widehat{\mathcal{R}^{l}_{0}}(z)(\mathcal{L}^{l}_{\varepsilon,E}-\mathcal{L}^{l}_{0,E})\|_{\mathfrak{B}(\mathcal{BS}_l )}
\le D\|W\|_\infty|\varepsilon|.
\end{equation*}

Take
$
\varepsilon_\ast:=\frac{1}{2D\|W\|_\infty}
$, which depends only on $\delta_0, r_\pm,\eta_\pm$. 
Then for every $|\varepsilon|\le \varepsilon_\ast$,
$
\|\widehat{\mathcal{R}^{l}_{0}}(z)(\mathcal{L}^{l}_{\varepsilon,E}-\mathcal{L}^{l}_{0,E})\|_{\mathfrak{B}(\mathcal{BS}_l )}
\le \frac12
$
uniformly in $l\in[l_0,\infty)$.
Hence
$
\mathrm{Id}-\widehat{\mathcal{R}^{l}_{0}}(z)(\mathcal{L}^{l}_{\varepsilon,E}-\mathcal{L}^{l}_{0,E})
$
is invertible and
$$
\left\|\Bigl(
\mathrm{Id}-\widehat{\mathcal{R}^{l}_{0}}(z)(\mathcal{L}^{l}_{\varepsilon,E}-\mathcal{L}^{l}_{0,E})\Bigr)^{-1}
\right\|_{\mathfrak{B}(\mathcal{BS}_l )}
\le 2.
$$

Therefore by \eqref{eq:reso}, 
$$
\bigl\|(z-\mathcal{L}^{l}_{\varepsilon,E})^{-1}\bigr\|_{\mathfrak{B}(\mathcal{BS}_l )}
\le 2\bigl\|\widehat{\mathcal{R}^{l}_{0}}(z)\bigr\|_{\mathfrak{B}(\mathcal{BS}_l )} \le 2M,
$$
uniformly in $l\in[l_0,\infty)$, provided $|\varepsilon|\le\varepsilon_\ast$.
\end{proof}

\begin{remark}\label{rem:mechanism}
The finite truncation introduces a singular rank-one perturbation diverges as $\mathcal{O}(|\hat{v}_l|^{-1})$. The weighted framework resolves this via an exact algebraic cancellation. The directional estimate \eqref{prop:inverse2} ensures that the unperturbed resolvent's response along the perturbation vector $\mathbf{e}_1$ is strictly bounded by $\mathcal{O}(|\hat{v}_l| e^{\xi(l-1)})$. In the operator composition $\widehat{\mathcal{R}^{l}_{0}}(z)(\mathcal{L}^{l}_{\varepsilon,E} - \mathcal{L}^{l}_{0,E})$, this  cancels the rescaled diverging coefficient $\mathcal{O}(|\hat{v}_l|^{-1}e^{-\xi(l-1)})$. 
\end{remark}

\subsection{Fiber decomposition}

Proposition~\ref{prop:improved} ensures the uniform invertibility of $z-\mathcal{L}^{l}_{\varepsilon,E}$ for all $|\varepsilon|\le\varepsilon_\ast$ and $z\in\mathcal{A}'(E)$. This allows us to define the global Riesz spectral projections via the contours in \eqref{contour}:
$$
\widehat{\mathbb P}^l_{\varepsilon,s} := \frac{1}{2\pi i}\int_{\Gamma_s} (z-\mathcal{L}^{l}_{\varepsilon,E})^{-1} dz, \qquad
\widehat{\mathbb P}^l_{\varepsilon,c} := \frac{1}{2\pi i} \int_{\Gamma_c} (z-\mathcal{L}^{l}_{\varepsilon,E})^{-1}dz,
$$
and $\widehat{\mathbb P}^l_{\varepsilon,u} := \mathrm{Id} - \widehat{\mathbb P}^l_{\varepsilon,s} - \widehat{\mathbb P}^l_{\varepsilon,c}$.

For any basepoint $\theta_0\in\Omega$, evaluating these global operators defines the fiberwise projections:
$$
\widehat{\mathbb P}^l_{\varepsilon,\sharp}(\theta_0)v := (\widehat{\mathbb P}^l_{\varepsilon,\sharp}\Phi)(\theta_0), \qquad \sharp \in \{s, c, u\},
$$
where $\Phi$ is any bounded section satisfying $\Phi(\theta_0)=v$. By the same reasoning as in Lemma \ref{prop:fiber}, $\widehat{\mathbb P}^l_{\varepsilon,\sharp}(\theta_0)$  is well-defined, and is independent of the chosen extension $\Phi$. Moreover, we have the following  partially hyperbolic splitting.

\begin{lemma}\label{prop:fibe2}
For each $\sharp \in \{s, c, u\}$, the subspaces $\widehat{\mathcal E}_{\varepsilon,\sharp}^l(\theta) := \operatorname{Ran}\widehat{\mathbb P}^l_{\varepsilon,\sharp}(\theta)$ form an invariant direct sum decomposition of $\mathbb{C}^{2l}$:
$$
\mathbb{C}^{2l} = \widehat{\mathcal E}^{l}_{\varepsilon,s}(\theta) \oplus \widehat{\mathcal E}^{l}_{\varepsilon,c}(\theta) \oplus \widehat{\mathcal E}^{l}_{\varepsilon,u}(\theta),
\qquad
\widehat A_{\varepsilon}^{l}(E,\theta) \widehat{\mathcal E}_{\varepsilon,\sharp}^l(\theta) = \widehat{\mathcal E}_{\varepsilon,\sharp}^l(T\theta).
$$
Furthermore, there exists a constant $C\ge 1$, independent of $l$ and $\varepsilon\leq \varepsilon_*$, such that for all $\theta\in\Omega$ and $n\ge 0$:
\begin{equation*}
\begin{aligned}
\|(\widehat A^{l}_{\varepsilon})_n(E, \theta)\widehat v_s\|_{1} &\le C r_s^n \|\widehat v_s\|_{1}, &&\quad \text{for } \widehat v_s\in \widehat{\mathcal E}^{l}_{\varepsilon,s}(\theta), \\
\|(\widehat A^{l}_\varepsilon)_{-n}(E,\theta)\widehat v_u\|_{1} &\le C r_u^{-n} \|\widehat v_u\|_{1}, &&\quad \text{for } \widehat v_u\in \widehat{\mathcal E}^{l}_{\varepsilon,u}(\theta), \\
\|(\widehat A^{l}_\varepsilon)_{n}(E,\theta)\widehat v_c\|_{1} \le C r_+^{n} \|\widehat v_c\|_{1},\quad
&\|(\widehat A^{l}_\varepsilon)_{-n}(E,\theta)\widehat v_c\|_{1} \le C r_-^{-n} \|\widehat v_c\|_{1}, &&\quad \text{for } \widehat v_c\in \widehat{\mathcal E}^{l}_{\varepsilon,c}(\theta).
\end{aligned}
\end{equation*}
The subspaces $\mathcal E_{\varepsilon,\sharp}^l(\theta) := R_l^{-1}\widehat{\mathcal E}_{\varepsilon,\sharp}^l(\theta)$ yield the corresponding invariant splitting for the original cocycle $A_{\varepsilon}^{l}(E,\theta)$.
\end{lemma}

\begin{proof}
    Same proof as in Lemma \ref{prop:growth} and Lemma \ref{prop:fiber}, we omit the details. 
\end{proof}

\subsection{Dimension of the center fiber}
We now determine the center fiber dimension. The proof relates the perturbed projection to the unperturbed constant-coefficient case.
\begin{lemma}
\label{prop:center-dimension-2m}
For all $l\in[l_0,\infty)$, $|\varepsilon|\le\varepsilon_\ast$, and $\theta\in\Omega$,
$\dim \widehat{\mathcal E}^{l}_{\varepsilon,c}(\theta)=2m(E).$
\end{lemma}
\begin{proof}
Fix $l$ and $\theta\in\Omega$. By definition,
$$
\widehat{\mathbb{P}}^{l}_{\varepsilon,c} = \frac{1}{2\pi i}\int_{\Gamma_+} (z-\mathcal{L}^{l}_{\varepsilon,E})^{-1} dz - \frac{1}{2\pi i}\int_{\Gamma_-} (z-\mathcal{L}^{l}_{\varepsilon,E})^{-1} dz.
$$
The circles $\Gamma_\pm$ remain in the resolvent set for $|\varepsilon|\le\varepsilon_\ast$, so the resolvent is continuous in $\varepsilon$ on these contours. Hence $\varepsilon\mapsto\widehat{\mathbb{P}}^{l}_{\varepsilon,c}(\theta)$ is continuous. Since each $\widehat{\mathbb{P}}^{l}_{\varepsilon,c}(\theta)$ is a projection, its rank equals its trace, which is integer-valued. Continuity of the trace implies the rank is constant in $\varepsilon$, so it suffices to compute it at $\varepsilon=0$.

We claim that for every $z\in \Gamma_+\cup\Gamma_-$ and every $v\in\mathbb C^{2l}$,
$
(z-\mathcal{L}^{l}_{0,E})^{-1}\widetilde v=\widetilde u,
$
where $\widetilde v(\theta)\equiv v$, $\widetilde u(\theta)\equiv
{(z-\widehat A_{0}^{l}(E))^{-1}v}$ denotes the constant section.
Indeed, let
$
u:=(z-\widehat A_{0}^{l}(E))^{-1}v.
$ Then
$
((z-\mathcal{L}^{l}_{0,E})\widetilde u)(\theta)=
(z-\widehat A_{0}^{l}(E))u=v.
$
Therefore
$
(z-\mathcal{L}^{l}_{0,E})\widetilde u=\widetilde v.
$
Since $z\in \Gamma_+\cup\Gamma_-$ lies in the resolvent set of $\mathcal{L}^{l}_{0,E}$,
the solution is unique, so
$
(z-\mathcal{L}^{l}_{0,E})^{-1}\widetilde v=
\widetilde u.
$

By definition of $\widehat{\mathbb{P}}^{l}_{0,c}(\theta)$,
\begin{align*}
\widehat{\mathbb{P}}^{l}_{0,c}(\theta)v=(\widehat{\mathbb{P}}^{l}_{0,c}\widetilde v)(\theta)
&=
\frac{1}{2\pi i}\int_{\Gamma_+}(z-\mathcal{L}^{l}_{0,E})^{-1}\widetilde v(\theta)dz-
\frac{1}{2\pi i}\int_{\Gamma_-}(z-\mathcal{L}^{l}_{0,E})^{-1}\widetilde v(\theta)dz \\
&=\frac{1}{2\pi i}\int_{\Gamma_+}\widetilde u(\theta)dz-
\frac{1}{2\pi i}\int_{\Gamma_-}\widetilde u(\theta)dz\\
&=\frac{1}{2\pi i}\int_{\Gamma_+}(z- \widehat A_{0}^{l}(E))^{-1}vdz-
\frac{1}{2\pi i}\int_{\Gamma_-}(z- \widehat A_{0}^{l}(E))^{-1}vdz.
\end{align*}
Hence
$$
\widehat{\mathbb{P}}^{l}_{0,c}(\theta)
=\frac{1}{2\pi i}\int_{\Gamma_+}(z-\widehat A_0^l(E))^{-1}\,dz
-\frac{1}{2\pi i}\int_{\Gamma_-}(z-\widehat A_0^l(E))^{-1}\,dz,
\qquad \forall\,\theta\in\Omega.
$$
By construction of  $\Gamma_{\pm}$, taking ranks gives
$\operatorname{rank}\widehat{\mathbb{P}}^{l}_{0,c}(\theta)=2m(E)$ for all $\theta\in\Omega.$
\end{proof}
\subsection{Proof of Theorem \ref{thm:sec4-main}}
The invariant splitting and partial hyperbolicity follow from Lemma~\ref{prop:fibe2} and  \eqref{equ:induced}; the center dimension follows from Lemma~\ref{prop:center-dimension-2m}.
\qed

\section{Global Trivialization and the Center Cocycle}\label{trivial-center}

In this section, we prove that the $\boldsymbol{\mathcal M}^c_{V,E,\varepsilon}:=\bigsqcup_{\theta\in\Omega}\mathcal{M}_{V,E,\varepsilon}^c(\theta)$ constitute a globally trivial vector subbundle, a geometric property that rigorously guarantees the existence of the center cocycle.

Recall that $\mathfrak A$ is the admissible Banach algebra introduced in Definition~\ref{def:admissible}. 
If $A=(a_{ij})_{1\le i,j\le r}$ with
$a_{ij}\in\mathcal H_\delta(\mathfrak A)$, we write 
$A\in M_r(\mathcal H_\delta(\mathfrak A))$
and define
$$\|A\|_{\mathfrak A,\delta}:=
\max_{1\le i,j\le r}\|a_{ij}\|_{\mathfrak A,\delta},
\qquad
\|A\|_{\infty,\delta}:=
\sup_{\zeta\in I_\delta}
\sup_{\theta\in\Omega}
\|A(\zeta,\theta)\|.$$
Since $\|f\|_\infty\le \|f\|_{\mathfrak A}$,
$\|A\|_{\infty,\delta}\le r\|A\|_{\mathfrak A,\delta}.$
Finally, let $\mathcal H_{\delta,\xi}(\mathfrak A)$ denote the space of sequences
$u=(u_k)_{k\in\mathbb Z}$ with $u_k\in\mathcal H_\delta(\mathfrak A)$ and
$$
\|u\|_{\mathfrak A,\delta,\xi}:=
\sum_{k\in\mathbb Z}\|u_k\|_{\mathfrak A,\delta}e^{-\xi|k|}
<\infty.
$$

\begin{theorem}\label{thm:centert}
Suppose $W(\cdot)\in \mathfrak A$, $V(\cdot)\in C_h^{\omega}(\T,\R)$. There exists $\varepsilon_{a} > 0$ such that for all $|\varepsilon| \le \varepsilon_{a}$, the center bundle $\boldsymbol{\mathcal M}^c_{V,E,\varepsilon}$ over the complexified base $I_{\delta_{1}}\times \Omega$ satisfies the following:
\begin{enumerate}
    \item $\dim \mathcal M_{V,\zeta,\varepsilon}^c(\theta) = 2m(E)$ for all $(\zeta,\theta) \in I_{\delta_1} \times \Omega$.
    \item $\boldsymbol{\mathcal M}^c_{V,E,\varepsilon}$ is globally trivialized  by a point-wise linearly independent  frame
    $$
    U_\varepsilon^{\infty}(\zeta,\theta) := \bigl( u_{\varepsilon}^{\infty,0}(\zeta,\theta), \dots, u_{\varepsilon}^{\infty,2m-1}(\zeta,\theta) \bigr),
    $$
    where $u_{\varepsilon}^{\infty,p}(\zeta,\theta) \in \mathcal H_{\delta,\xi}(\mathfrak A),p=0,\cdots 2m-1 $.
    \item  There exists invertible matrix  $C_\varepsilon^{\infty }(\zeta,\theta) \in\rm M_{2m(E)}(\mathcal H_\delta(\mathfrak A))$, such that 
     $$
S U^{\infty}_\varepsilon(\zeta,\theta) = U^{\infty}_\varepsilon(\zeta,T\theta) C_{\varepsilon}^{\infty}(\zeta,\theta).
    $$
    Moreover, the center cocycle $C_{\varepsilon}^{\infty}(\zeta,\theta)$ has the companion form:
   \begin{equation}
\label{eq:comicc}
C_{\varepsilon}^{\infty}(\zeta,\theta):=
\begin{pmatrix}
0 & 0 & \cdots & 0 & *\\
1 & 0 & \cdots & 0 & *\\
0 & 1 & \cdots & 0 & *\\
\vdots & \vdots & \ddots & \vdots & \vdots\\
0 & 0 & \cdots & 1 & *
\end{pmatrix},
\end{equation}
\end{enumerate}
where $\delta_1$ is defined in  Lemma \ref{thm:D-abstract}.
\end{theorem}

The proof proceeds in two steps. First, we resolve the finite-range case: utilizing the Algebraic Bridge (Proposition \ref{thm:bridge1}), we construct a global holomorphic frame that explicitly trivializes the truncated space $\mathcal{C}_{V_l,E,\varepsilon}$. Second, we pass this trivialization to the infinite-range limit via the uniform convergence established in the Analytic Bridge (Proposition \ref{thm:ambient-projection}).

\subsection{Global trivialization for finite-range via the Algebraic Bridge}
By the Algebraic Bridge  (Proposition \ref{thm:bridge1}) and Theorem \ref{thm:sec4-main}, we obtain the exact algebraic identification $\mathcal{C}^{2l}(\theta)=\mathcal{E}^{l}_{\varepsilon,c}(\theta)$. Consequently, for  $l \in[ l_0,\infty)$, $\boldsymbol{\mathcal M}^c_{V,E,\varepsilon}$ constitute a vector subbundle of constant fiber dimension $2m(E)$. To globally trivialize this bundle, it thus suffices to construct $2m(E)$ global sections and prove they are everywhere point-wise linearly independent. For notational simplicity, in the sequel we write $m=m(E)$.

To construct a global trivialization for the center bundle, we apply the integral representation of Lemma \ref{recons} to the monomials $F_p(z) := z^p$. For each $l\in[l_0,\infty]$, $|\varepsilon|\le \varepsilon_0$, $\zeta\in I_{\delta_1}$, $\theta\in \Omega$, and $k\in \Z$, we define the sections:
\begin{equation}\label{equ:basis1}
u^{l,p}_{\varepsilon; k}(\zeta,\theta) := \frac{1}{2\pi i}\int_{\Gamma_c} z^{k+p} g^{l}_{\zeta,\varepsilon}(z,T^k\theta) dz \in \mathcal{M}^c_{V_l,\zeta,\varepsilon}(\theta),
\end{equation}
where $l_0, \varepsilon_0$, and $g^{l}_{\zeta,\varepsilon}$ are specified in Lemma \ref{thm:D-abstract}.
The contour integral definition directly yields the index-shift relation 
\begin{equation*}
    u^{l,p}_{\varepsilon;k+1}(\zeta,\theta)=u^{l,p+1}_{\varepsilon; k}(\zeta,T\theta).
\end{equation*}This translates to the exact shift covariance:
\begin{equation}\label{equ:strans}
S u^{l,p}_{\varepsilon}(\zeta,\theta) = u^{l,p+1}_{\varepsilon}(\zeta,T\theta).
\end{equation}
This explicit algebraic construction ensures that the center cocycle takes a companion matrix form,  adapted to the shift dynamics. 
The following lemma establishes the basic estimates for these sections:

\begin{lemma}\label{lem:uest}
Let $R:=\max(r_+,r_-^{-1})$. There exist constants $\widetilde M_U,M_U,C_D,C_U>0$,
independent of $l,\varepsilon,\zeta,\theta$, such that for every
$0\le p\le 2m$,
\begin{equation}\label{equ:basisbound}
\sup_{\substack{l\in[l_0,\infty]\\ |\varepsilon|\le \varepsilon_0}}
\|u^{l,p}_{\varepsilon;k}(\zeta,\theta)\|_{\mathfrak A,\delta_1}
\le\widetilde M_U R^{|k|},
\qquad \sup_{\substack{l\in[l_0,\infty]\\ |\varepsilon|\le \varepsilon_0}}
\|u^{l,p}_{\varepsilon}(\zeta,\theta)\|_{\mathfrak A,\delta_1,\xi}
\le M_U.
\end{equation} 
Moreover, we have the following: 
\begin{eqnarray}
\label{equ:basiscon}
&&\sup_{|\varepsilon|\le \varepsilon_0}
\|u^{l,p}_{\varepsilon}(\zeta,\theta)-
u^{\infty,p}_{\varepsilon}(\zeta,\theta)\|_{\mathfrak A,\delta_1,\xi}
\longrightarrow 0,
\qquad l\to\infty, \\
\label{equ:basiszero}
&&\sup_{l\in[l_0,\infty]}
\|u^{l,p}_{\varepsilon;k}(\zeta,\theta)
-u^{l,p}_{0;k}(\zeta)\|_{\mathfrak A,\delta_1}\le
C_D |\varepsilon| R^{|k|},\\
\label{equ:basisper}
&&\sup_{l\in[l_0,\infty]}
\|u^{l,p}_{\varepsilon}(\zeta,\theta)
-u^{l,p}_{0}(\zeta)\|_{\mathfrak A,\delta_1,\xi}
\le C_U |\varepsilon|.
\end{eqnarray}
\end{lemma}
\begin{proof}
By Lemma \ref{thm:D-abstract},   $\|g^{l}_{\cdot,\varepsilon}(z,\cdot)\|_{\mathfrak A,\delta_1} \le M_D$ on $\Gamma_c$, this directly implies
\begin{equation*}
\|u^{l,p}_{\varepsilon;k}(\zeta,\theta)\|_{\mathfrak A,\delta_1}
\le\frac{|\Gamma_c|}{2\pi}M_D
\sup_{z\in\Gamma_c}|z|^{k+p}
\le \frac{|\Gamma_c|}{2\pi}M_D R^{|k|}R^p,
\end{equation*}
set $\tilde  M_U=\frac{|\Gamma_c|}{2\pi}M_DR^{2m}$, \eqref{equ:basisbound} immediately follows.
For each $k\in\Z$,
$$
u_{\varepsilon;k}^{l,p}(\zeta,\theta)-u_{\varepsilon;k}^{\infty,p}(\zeta,\theta)
=\frac{1}{2\pi i}\int_{\Gamma_c}
z^{k+p}\Bigl(
g_{\zeta,\varepsilon}^{l}(z,T^k \theta)-
g_{\zeta,\varepsilon}^{\infty}(z,T^k \theta)
\Bigr)\,dz.
$$
Since  $g_{\zeta,\varepsilon}^{l}(z,\theta)=(D_{\varepsilon,\zeta}^{(l)}(z))^{-1}\mathbf{1}(\theta)$, by Lemma \ref{thm:D-abstract} (2), for each fixed $k\in\Z$, we have 
\begin{equation}\label{eq:conv}
    \sup_{|\varepsilon|\le \varepsilon_0}
\|u_{\varepsilon;k}^{l,p}(\zeta,\theta)-u_{\varepsilon;k}^{\infty,p}(\zeta,\theta)\|_{\mathfrak A,\delta_1}\to 0
\qquad\text{as }l\to\infty.
\end{equation}
By \eqref{equ:basisbound}, there exist constants
$C>0$, independent of $l,\varepsilon,\zeta,\theta$, such that
$$
\sup_{|\varepsilon|\le \varepsilon_0}
\|u_{\varepsilon;k}^{l,p}(\zeta,\theta)-u_{\varepsilon;k}^{\infty,p}(\zeta,\theta)\|_{\mathfrak A,\delta_1}
\,e^{-\xi|k|}\le
C(Re^{-\xi})^{|k|}.
$$
Since $Re^{-\xi}<1$, \eqref{equ:basiscon} follows from \eqref{eq:conv} and  Lebesgue Dominated Convergence Theorem.

Finally, by Lemma \ref{thm:D-abstract} (3), 
we get
$$
 \sup_{l \in [l_0, \infty]}\|u_{\varepsilon;k}^{l,p}(\zeta,\theta)-
u_{0;k}^{l,p}(\zeta,\theta)\|_{\mathfrak A,\delta_1}
\le\frac{|\Gamma_c|}{2\pi}M_D^2\|W\|_\mathcal{h} |\varepsilon|R^{|k|+p}.
$$
Set
$
C_D:=\frac{|\Gamma_c|}{2\pi}M_D^2\|W\|_\mathcal{h}R^{2m},
$ \eqref{equ:basiszero} and \eqref{equ:basisper}  follow immediately.
\end{proof}

Arranging the first $2m$ sections, we define the candidate frame:
\begin{equation}\label{equ:basis}
    U^{l}_{\varepsilon}(\zeta,\theta):=\bigl(
u^{l,0}_{\varepsilon}(\zeta,\theta),\dots,
u^{l,2m-1}_{\varepsilon}(\zeta,\theta)\bigr),
\end{equation}
then we have the following:

\begin{proposition}\label{thm:sec8-final}
Assume $W(\cdot) \in \mathfrak A$. There exist constants $l_a \ge 1$ and $\varepsilon_a > 0$ such that for every $l \in [l_a, \infty]$ and $|\varepsilon| \le \varepsilon_a$, the following hold:
\begin{enumerate}
    \item $\boldsymbol{\mathcal M}^c_{V,E,\varepsilon}$ admits a global frame
    $
    U^l_\varepsilon(\zeta,\theta)
    $
    on $I_{\delta_1} \times \Omega$.
    \item As a consequence,  for every $l \in [l_a,\infty)$, the columns of $U^l_\varepsilon(\zeta,\theta)$ form a basis of the fiber $\mathcal{M}^c_{V_l,\zeta,\varepsilon}(\theta)$.
\end{enumerate}
\end{proposition}

\begin{proof}
It remains to establish their point-wise linear independence. Let $\Pi_{[0,2m-1]}: \mathcal{X}_\xi \to \C^{2m}$ denote the canonical projection onto the central coordinates $k \in \{0, \dots, 2m-1\}$. For any $l \in [l_0, \infty]$ and $(\zeta,\theta) \in I_{\delta_1} \times \Omega$, we define the $2m \times 2m$  matrix:
$$
B_\varepsilon^l(\zeta,\theta) := \Bigl( \Pi_{[0,2m-1]}u_{\varepsilon}^{l,0}(\zeta,\theta), \dots, \Pi_{[0,2m-1]}u_{\varepsilon}^{l,2m-1}(\zeta,\theta) \Bigr).
$$
The linear independence  reduces entirely to verifying the uniform non-singularity of $B_\varepsilon^l$. We first establish this non-degeneracy for the unperturbed case ($\varepsilon=0$). Remarkably, the non-degeneracy of the contour integrals reduces to a pure polynomial algebra problem, which we resolve via Hermite interpolation.

\begin{lemma}\label{lem:lindep-uniform}
There exists a universal constant $c_* > 0$ such that for every $l \in [l_0, \infty]$ and $\zeta \in \overline{I_{\delta_1}}$, we have
$
|\det B_0^l(\zeta)| \ge c_*.
$
\end{lemma}

\begin{proof}

Fix $\zeta\in \overline{I_{\delta_1}}$ and $l$. 
Assume there exist coefficients
$
\beta_0,\dots,\beta_{2m-1}\in\mathbb C,
$ such that
$
\sum_{p=0}^{2m-1}\beta_p u^{l,p}_{0;k}(\zeta)=0
$
 for every $0\le k\le 2m-1$.
Define the polynomial
$
Q(z):=\sum_{p=0}^{2m-1}\beta_p z^p.
$
Then, by the definition of the vectors $u_{0}^{l,p}(\zeta)$,
$$
0=\sum_{p=0}^{2m-1}\beta_p u^{l,p}_{0;k}(\zeta)=
\frac{1}{2\pi i}\int_{\Gamma_c}\frac{z^k Q(z)}{L_\zeta^l(z)} dz
\qquad 0\le k\le 2m-1.
$$
Hence, for every polynomial $P$ with $\deg P\le 2m-1$,
$$
\frac{1}{2\pi i}\int_{\Gamma_c}\frac{P(z)Q(z)}{L_\zeta^l(z)}dz=0.
$$
Let $\lambda_1,\dots,\lambda_r$ be the zeros of $L_\zeta^l(z)$ inside the center annulus $\{z\in\mathbb C: r_-(E) \le |z| \le r_+(E)\}$, with multiplicities
$
n_1,\dots,n_r,
$
so that
$
n_1+\cdots+n_r=2m.
$
Set
$
R(z)=\frac{Q(z)}{{L_\zeta^l(z)}}.
$
 Hence near $\lambda_j$ we can write
$$
R(z)=\sum_{s=1}^{n_j}c_{j,s}(z-\lambda_j)^{-s}+h_j(z),
$$
where $h_j$ is analytic near $\lambda_j$.
By the residue theorem,
\begin{equation}\label{res-eq}
    0=\frac{1}{2\pi i}\int_{\Gamma_c}P(z)R(z) dz
=\sum_{j=1}^{r}
\operatorname{Res}_{z=\lambda_j}(P(z)R(z))=\sum_{j=1}^{r}\sum_{s=1}^{n_j}
c_{j,s}\frac{P^{(s-1)}(\lambda_j)}{(s-1)!}.
\end{equation}
Fix indices $j_0$ and $s_0$ with $1\le j_0\le r$ and
$1\le s_0\le n_{j_0}$. By Hermite interpolation,
there exists a polynomial $P$ of degree at most $2m-1$ such that 
$$
P^{(s_0-1)}(\lambda_{j_0})=(s_0-1)! \quad \text{and}\quad P^{l}(\lambda_j)=0 \quad\text{for all other }(j,l).
$$
Substituting this $P$ into \eqref{res-eq} gives
$c_{j_0,s_0}=0.$
Since $(j_0,s_0)$ was arbitrary, we conclude that
$c_{j,s}=0$ for all  $1\leq j\le r,1\le s\le n_j.$
Hence $R(z)=Q(z)/L(z)$ has no poles inside $\Gamma_c$ and is therefore
analytic there. Consequently $Q$ must vanish at each $\lambda_j$ with
multiplicity at least $n_j$. The total multiplicity of these zeros is
$n_1+\cdots+n_r=2m.$
But $\deg Q\le 2m-1$, so the only possibility is
$Q\equiv 0.$
Therefore $\beta_0=\cdots=\beta_{2m-1}=0$.

 Then for every fixed $\zeta\in I_{\delta_1}$ and every
$l\in [l_0,\infty]$, one has
$
\det B_0^l(\zeta)\neq 0.
$
For each fixed $l$, the map
$
\zeta\mapsto \det B_0^l(\zeta)
$
is continuous on the compact set $\overline{I_{\delta_1}}$, we have
$
c_l:=\min_{\zeta\in\overline{I_{\delta_1}}}| \det B_0^l(\zeta)|>0.
$ On the other hand, for every $0\le k,p\le 2m-1$,
$$
u_{0;k}^{l,p}(\zeta)
=\frac{1}{2\pi i}\int_{\Gamma_c}\frac{z^{k+p}}{L_{0,\zeta}^{l}(z)}dz\to
\frac{1}{2\pi i}\int_{\Gamma_c}\frac{z^{k+p}}{L_{0,\zeta}(z)}dz
=u_{0;k}^{\infty,p}(\zeta)
$$
uniformly on $\overline{I_{\delta_1}}$, we get
$
\sup_{\zeta\in\overline{I_{\delta_1}}}
\|\det B_0^l(\zeta)-\det B_0^\infty(\zeta)\|\to 0, 
$ as $l\to\infty$.

Since $c_\infty>0$, there exists $l_1\geq l_0$, such that for $l\in[ l_1,\infty]$,
$
|\det B_0^l(\zeta)|\ge \frac{c_\infty}{2}
$ for all $ \zeta\in\overline{I_{\delta_1}}.$
Taking the minimum with the finitely many positive constants $c_l$ for
$
l_0\le l<l_1,
$
we obtain a constant $c_*>0$ such that
$
| \det B_0^l(\zeta)|\ge c_*
$ for all $\zeta\in\overline{I_{\delta_1}}, l\in[l_0,\infty].$
\end{proof}

By the multilinearity of the determinant with respect to its columns, we have the  telescoping sum:
\begin{align*}
\det B_\varepsilon^l(\zeta,\theta)-\det B_0^l(\zeta,\theta) =
\sum_{q=0}^{2m-1}\det\Bigl(
b_0^{l,0},\ldots,b_0^{l,q-1},
b_\varepsilon^{l,q}-b_0^{l,q},
b_\varepsilon^{l,q+1},\ldots,b_\varepsilon^{l,2m-1}
\Bigr)(\zeta,\theta),
\end{align*}
where
$
b_\varepsilon^{l,p}(\zeta,\theta):=
\Pi_{[0,2m-1]}u_\varepsilon^{l,p}(\zeta,\theta),
p=0,\ldots,2m-1.
$
Using \eqref{equ:basisbound}, there exists a constant $M$, such that $$\sup_{l\in[l_0,\infty]}\sup_{|\varepsilon|<\varepsilon_0}\sup_{0\le p\le 2m-1}
\bigl\|\Pi_{[0,2m-1]}u_0^{l,p}(\zeta,\theta)\bigr\|_{\mathfrak A,\delta_1}\leq \tilde M_U \sum_{k=0}^{2m-1}R^{|k|}\leq M,$$
where the vector norm takes the vector $1$-norm.
Combine with \eqref{equ:basiszero}, we obtain
\begin{align*}
\bigl|
\det B_\varepsilon^l(\zeta,\theta)-\det B_0^l(\zeta,\theta)\bigr| 
&\le\sum_{q=0}^{2m-1}C_m M^{2m-1}\bigl\|\Pi_{[0,2m-1]}\bigl(u_\varepsilon^{l,q}(\zeta,\theta)
-u_0^{l,q}(\zeta,\theta)\bigr)\bigr\|_{\mathfrak A,\delta_1}\\
&\leq
\sum_{q=0}^{2m-1}C_m M^{2m-1}C_D |\varepsilon| 
\sum_{k=0}^{2m-1}R^{|k|}
\le\tilde C_m |\varepsilon|.
\end{align*}
Here $C_m>0$ is only depend on $m$, and $\tilde C_m:=\sum_{q=0}^{2m-1}
C_m M^{2m-1}\sum_{k=0}^{2m-1}R^{|k|}.$

Setting $l_a = l_0$ and $\varepsilon_a := \min\{1, \frac{c_*}{2\tilde C_m }\}$, we immediately obtain for all $|\varepsilon| \le \varepsilon_a$:
\begin{equation}\label{equ:det}
    |\det B_\varepsilon^l(\zeta,\theta)| \ge |\det B_0^l(\zeta,\theta)| - \tilde C_m |\varepsilon| \ge \frac{c_*}{2} > 0.
\end{equation}
This uniform non-degeneracy  guarantees that the projected vectors are linearly independent in $\C^{2m}$, which consequently forces $\{u_\varepsilon^{l,p}(\zeta,\theta)\}_{p=0}^{2m-1}$ to be point-wise linearly independent in $\mathcal{X}_\xi$. 
\end{proof}

\subsection{Global trivialization for infinite-range via the Analytic Bridge }

The goal of this subsection is to pass from the finite-range trivializations to the infinite-range center space.  For the global frame $U_{\varepsilon}^{l}(\zeta,\theta)$, we define its Gram matrix by
\begin{equation*}\label{eq:Gram-matrix}
\mathbf{G}^{l}_{\varepsilon}(\zeta,\theta) := \Bigl( \langle u^{l,p}_{\varepsilon}(\zeta,\theta), u^{l,q}_{\varepsilon}(\zeta,\theta) \rangle_\xi \Bigr)_{0\le p,q\le 2m-1},
\end{equation*}
equipped with  $\langle y,z\rangle_\xi := \sum_{k\in\Z} e^{-2\xi|k|}y_k\overline{z_k}$. 

The uniform invertibility of $\mathbf{G}^{l}_{\varepsilon}$  guarantees that the frame admits a uniformly bounded left inverse. This structural non-degeneracy provides the  algebraic mechanism required to extract and stabilize local coordinates as $l \to \infty$.

\begin{lemma}\label{prop:extended}
There exist universal constants $c_0, C_G > 0$ such that for all $l \in [l_a,\infty]$, $|\varepsilon| \le \varepsilon_a$, and $(\zeta,\theta) \in I_{\delta_1} \times \Omega$, the Gram matrix satisfies the following properties:

\begin{enumerate}
    \item $\mathbf{G}_\varepsilon^l(\zeta,\theta)$ is uniformly invertible with 
    \begin{equation}\label{eq:Gbound}
        \|(\mathbf{G}_\varepsilon^l(\zeta,\theta))^{-1}\|_{\infty,\delta_1} \le c_0.
    \end{equation}
    Consequently, the frame  $U_\varepsilon^l(\zeta,\theta)$ admits an exact algebraic left inverse:
    \begin{equation}\label{eq:left-inverse}
        \bigl(\mathbf{G}_\varepsilon^l(\zeta,\theta)\bigr)^{-1} \bigl(\mathcal{W}_\xi U_\varepsilon^l(\zeta,\theta)\bigr)^* \mathcal{W}_\xi \cdot  U_\varepsilon^l(\zeta,\theta) = I_{2m},
    \end{equation}
where $
(\mathcal W_\xi u)_k:=e^{-\xi|k|}u_k.
$
    \item  As $l \to \infty$, the inverse of Gram matrix converges uniformly:
    \begin{equation}\label{equ:gramcon}
        \lim_{l\to\infty} \sup_{|\varepsilon|\le \varepsilon_a} \|(\mathbf{G}_\varepsilon^l)^{-1} - (\mathbf{G}_\varepsilon^\infty)^{-1}\|_{\infty,\delta_1} = 0.
    \end{equation}

    \item The inverse of Gram matrix is uniformly Lipschitz continuous with the perturbation:
    \begin{equation}\label{eq:invG}
        \sup_{l\in[l_a,\infty]} \|(\mathbf{G}^{l}_\varepsilon)^{-1} - (\mathbf{G}^{l}_0)^{-1}\|_{\infty,\delta_1} \le C_G |\varepsilon|.
    \end{equation}
\end{enumerate}
\end{lemma}

\begin{proof}
(1)
By \eqref{equ:det},
there exists $c_*>0$ such that
$
\left|\det(B_\varepsilon^l(\zeta,\theta))\right|\ge \frac{c_*}{2}
$
for every $l\in[l_a,\infty]$, every $|\varepsilon|\le \varepsilon_a$, and every
$(\zeta,\theta)\in{I_{\delta_1}}\times \Omega$.
Since the matrices
$
B_\varepsilon^l(\zeta,\theta)
$
have uniformly bounded entries by \eqref{equ:basisbound}, the lower bound on the determinant implies a uniform
lower bound on the smallest singular value. Hence there exists $\tilde c_0>0$ such that
$$
\|B_\varepsilon^l(\zeta,\theta)a\|\ge \tilde c_0 \|a\|
\qquad
\text{for all }a=(a_0,a_1,\cdots,a_{2m-1})^{\!\top}\in\C^{2m},
$$
uniformly in $l,\varepsilon,\zeta,\theta$.

Now let
$
v_a^l(\zeta,\theta):=\sum_{p=0}^{2m-1}a_pu_\varepsilon^{l,p}(\zeta,\theta).
$
Then
$
\Pi_{[0,2m-1]}v_a^l(\zeta,\theta)=B_\varepsilon^l(\zeta,\theta)a,
$
so
$$
\|v_a^l(\zeta,\theta)\|_\xi
\ge e^{-2\xi m}
\|\Pi_{[0,2m-1]}v_a^l(\zeta,\theta)\|
= e^{-2\xi m}\|B_\varepsilon^l(\zeta,\theta)a\|\ge e^{-2\xi m}\tilde c_0\|a\|.
$$
Therefore
$$
a^*\mathbf G_\varepsilon^l(\zeta,\theta)a
=\left\|\sum_{p=0}^{2m-1}a_pu_\varepsilon^{l,p}(\zeta,\theta)\right\|_\xi^2
=\|v_a^l(\zeta,\theta)\|_\xi^2
\ge e^{-4\xi m}\tilde c_0^2\|a\|^2.
$$
 We obtain
$
\mathbf G_\varepsilon^l(\zeta,\theta)\ge  e^{-4\xi m}\tilde c_0^2 I_{2m}.
$
The bound on the inverse follows immediately.
On the other hand,
it is easy to  check that 
\begin{equation*}
\mathbf{G}^{l}_{\varepsilon}(\zeta,\theta) =\Bigl( \langle u^{l,p}_{\varepsilon}(\zeta,\theta), u^{l,q}_{\varepsilon}(\zeta,\theta) \rangle_\xi \Bigr)_{0\le p,q\le 2m-1}=
\bigl(\mathcal W_\xi U_\varepsilon^l(\zeta,\theta)\bigr)^*
\mathcal W_\xi U_\varepsilon^l(\zeta,\theta).
\end{equation*}
The left inverse \eqref{eq:left-inverse}  follows the definition of $\mathbf G_\varepsilon^l(\zeta,\theta)$.

(2)  Fix $0\le p,q\le 2m-1$. By the definition of the Gram matrix,
$$
\begin{aligned}
\big|(\mathbf G_\varepsilon^l(\zeta,\theta)-\mathbf G_\varepsilon^\infty(\zeta,\theta))(p,q)\big| 
\leq  \|u_\varepsilon^{l,p}\|_{\mathfrak A,\delta_1,\xi} \|u_\varepsilon^{l,q} - u_\varepsilon^{\infty,q}\|_{\mathfrak A,\delta_1,\xi}  + \|u_\varepsilon^{l,p} - u_\varepsilon^{\infty,p}\|_{\mathfrak A,\delta_1,\xi} \|u_\varepsilon^{\infty,q}\|_{\mathfrak A,\delta_1,\xi},
\end{aligned}
$$
by \eqref{equ:basisbound} and \eqref{equ:basiscon}, 
the right-hand side vanishes uniformly in $\varepsilon$ as $l \to \infty$. Since the matrix dimension $2m$ is finite, this component-wise decay immediately guarantees the uniform convergence of the matrix operator norm: 
\begin{equation}\label{equ:gramcon1}
   \sup_{|\varepsilon|\le \varepsilon_a} \|\mathbf G_\varepsilon^l(\zeta,\theta)-\mathbf G_\varepsilon^\infty(\zeta,\theta)\|_{\infty,\delta_1}\to 0,\quad l\to\infty.
\end{equation}
On the other hand, by the resolvent identity
$$
(\mathbf{G}_\varepsilon^l)^{-1} - (\mathbf{G}_\varepsilon^\infty)^{-1} 
= (\mathbf{G}_\varepsilon^l)^{-1} \bigl( \mathbf{G}_\varepsilon^\infty - \mathbf{G}_\varepsilon^l \bigr) (\mathbf{G}_\varepsilon^\infty)^{-1},
$$
by  \eqref{equ:gramcon1} and \eqref{eq:Gbound}, we have
$$
\begin{aligned}
     &\quad\sup_{|\varepsilon|\le \varepsilon_a}\bigl\|
(\mathbf G_\varepsilon^l(\zeta,\theta))^{-1}-
(\mathbf G_\varepsilon^\infty(\zeta,\theta))^{-1}
\bigr\|_{\infty,\delta_1}\\
&
\le \sup_{|\varepsilon|\le \varepsilon_a}
\|\mathbf G_\varepsilon^l(\zeta,\theta)-\mathbf G_\varepsilon^\infty(\zeta,\theta)\|_{\infty,\delta_1}
\|(\mathbf G_\varepsilon^l(\zeta,\theta))^{-1}\|_{\infty,\delta_1}
\|(\mathbf G_\varepsilon^\infty(\zeta,\theta))^{-1}\|_{\infty,\delta_1}
\to 0.
\end{aligned}
$$

(3) The Lipschitz bound \eqref{eq:invG} follows by the same argument, simply replacing the limit convergence \eqref{equ:basiscon} with the perturbation bound \eqref{equ:basisper}. \qedhere

\end{proof}

\smallskip

\subsection*{Proof of Theorem \ref{thm:centert}}
(1)  Let $u \in \mathcal{M}_{V,E,\varepsilon}^c(\theta)$. By definition, there exists a global section $\Phi \in \mathcal{C}_{V,E,\varepsilon}$ evaluating to $\Phi(\theta) = u$. By the Analytic Bridge (Proposition \ref{thm:ambient-projection}), this section is  invariant under the infinite-range projection: $$ 
(\mathbb{ P }_{\varepsilon,c}^{\infty}\Phi)(\theta)= (\mathbb P_{V,E,\varepsilon}^c\Phi)(\theta)=\Phi(\theta)=u.
$$ For $l \ge l_a$, we define the truncated sections $\Phi^l := \mathbb{P}_{\varepsilon,c}^l\Phi$. Proposition \ref{thm:ambient-projection} guarantees $\Phi^l \in \operatorname{Ran}(\mathbb{P}_{\varepsilon,c}^l) = \mathcal{C}_{V_l,E,\varepsilon}$, alongside the uniform norm convergence:
\begin{equation}\label{phicon}
    \|\Phi^l - \Phi\|_\infty \le \|\mathbb{P}_{\varepsilon,c}^l - \mathbb{P}_{\varepsilon,c}^\infty\|_{\mathfrak{B}( \mathcal{BS}_\xi)} \|\Phi\|_\infty \to 0.
\end{equation}
Consequently, the truncated sequence is uniformly bounded: $\sup_{l \ge l_a} \|\Phi^l\|_\infty < \infty$.

Since $\Phi^{l}\in\mathcal{C}_{V_l,E,\varepsilon}$, by Proposition \ref{thm:sec8-final}, there exists
a unique vector $c_l(\theta)\in\C^{2m}$ such that
$
\Phi^{l}(\theta)= U_\varepsilon^{l}(E,\theta)c_l(\theta). 
$ Applying the left inverse \eqref{eq:left-inverse}, we have
$c_l(\theta)=\bigl(\mathbf G_\varepsilon^l(\zeta,\theta)\bigr)^{-1}d_l(\theta)$,
where 
\begin{eqnarray*}{d}_l(\theta)&=& \bigl(\mathcal{W}_\xi U_\varepsilon^l(\zeta,\theta)\bigr)^* \mathcal{W}_\xi \Phi^{l}(\theta)\\
&=&\Bigl(\langle u^{l,0}_{\varepsilon}(E,\theta),\Phi^{l}(\theta)
\rangle_\xi,\dots,\langle
u^{l,2m-1}_{\varepsilon}(E,\theta),
\Phi^{l}(\theta)\rangle_\xi\Bigr)^{\!\top}.
\end{eqnarray*}
We estimate its asymptotic behavior. For $0\leq p\leq 2m-1$, 
$$\begin{aligned}
   &\quad \big|\langle
u^{l,p}_{\varepsilon}(E,\theta),\Phi^{l}(\theta)\rangle_\xi-
\langle u^{\infty,p}_{\varepsilon}(E,\theta),
\Phi(\theta)\rangle_\xi\big|\\&\leq \big|\langle
u^{l,p}_{\varepsilon}(E,\theta)- u^{\infty,p}_{\varepsilon}(E,\theta),\Phi^{l}(\theta)\rangle_\xi\big|+\big|\langle  u^{\infty,p}_{\varepsilon}(E,\theta),
\Phi^{l}(\theta)-\Phi(\theta)\rangle_\xi\big|\\
&\leq\|u^{l,p}_{\varepsilon}(E,\theta)- u^{\infty,p}_{\varepsilon}(E,\theta)\|_{\mathfrak A,\delta_1,\xi}\|\Phi_l\|_\infty+\|u^{\infty,p}_{\varepsilon}(E,\theta)\|_{\mathfrak A,\delta_1,\xi}\|\Phi^{l}-
\Phi\|_\infty \rightarrow 0,
\end{aligned}$$
where we used \eqref{equ:basisbound}, \eqref{equ:basiscon}, \eqref{phicon}.
Thus, the pairing vectors converge pointwise: $d_l(\theta) \to d_\infty(\theta) $.
Coupled with the uniform metric convergence $\bigl(\mathbf{G}_\varepsilon^l\bigr)^{-1} \to \bigl(\mathbf{G}_\varepsilon^\infty\bigr)^{-1}$ established in \eqref{equ:gramcon}, we have
$
c_l(\theta) \to c_\infty(\theta) := \bigl(\mathbf{G}_\varepsilon^\infty(E,\theta)\bigr)^{-1} d_\infty(\theta).
$
Passing to the limit in  $\Phi^l(\theta) = U_\varepsilon^l(E,\theta) c_l(\theta)$ enforces:
$
\Phi(\theta) = U_\varepsilon^\infty(E,\theta) c_\infty(\theta).
$
It follows that $u = U_\varepsilon^\infty(E,\theta) c_\infty(\theta)$. Therefore, $\dim{{\mathcal{M}_{V,E,\varepsilon}^c(\theta)}}\leq 2m(E),$ and since $U_\varepsilon^{\infty}(E,\theta)$ is linearly independent, we have $\dim{{\mathcal{M}_{V,E,\varepsilon}^c(\theta)}}= 2m(E).$ 

(2) Consequently, for all $l \in [l_a, \infty]$, the global holomorphic frame $U_\varepsilon^l(\zeta,\theta)$ exactly trivializes the truncated center bundle:
$$
\mathcal{M}_{V_l,E,\varepsilon}^c(\theta) = \mathrm{span}\{u_\varepsilon^{l,0}(\zeta,\theta), \dots, u_\varepsilon^{l,2m-1}(\zeta,\theta)\}.
$$

(3) The invariance of the center bundle under the shift dynamics $S$ guarantees the existence of a Center Cocycle $C_\varepsilon^l(\zeta,\theta) $ such that:
\begin{equation*}\label{equ:s}
S U_\varepsilon^l(\zeta,\theta) = U_\varepsilon^l(\zeta,T\theta) C_\varepsilon^l(\zeta,\theta).
\end{equation*}
By the shift covariance \eqref{equ:strans}, this cocycle necessarily has the  companion form:
\begin{equation}\label{eq:defAc}
C_\varepsilon^l(\zeta,\theta) :=
\begin{pmatrix}
0 & 0 & \cdots & 0 & -a^l_{\varepsilon,0}(\zeta,\theta)\\
1 & 0 & \cdots & 0 & -a^l_{\varepsilon,1}(\zeta,\theta)\\
0 & 1 & \cdots & 0 & -a^l_{\varepsilon,2}(\zeta,\theta)\\
\vdots & \vdots & \ddots & \vdots & \vdots\\
0 & 0 & \cdots & 1 & -a^l_{\varepsilon,2m-1}(\zeta,\theta)
\end{pmatrix}.
\end{equation}
The final column is explicitly determined by the shift of the final frame section \eqref{equ:strans}. Specifically, there exists a unique coefficient vector $a_\varepsilon^l(\zeta,\theta) \in \C^{2m}$ such that:
\begin{equation}\label{eq:defa}
u_\varepsilon^{l,2m}(\zeta,T\theta) = - U_\varepsilon^l(\zeta,T\theta) a_\varepsilon^l(\zeta,\theta).
\end{equation}
 By \eqref{equ:det}, the finite-dimensional block $B_\varepsilon^l$ is  invertible, enforcing:
$$
a^{l}_{\varepsilon}(\zeta,\theta)
=-\bigl(B_\varepsilon^l(\zeta,T\theta))^{-1}
\Pi_{[0,2m-1]}u^{l,2m}_\varepsilon(\zeta,T\theta).
$$
Therefore
$
a_{\varepsilon,i}^l\in\mathcal H_{\delta_1}(\mathfrak A),0\leq i\leq 2m-1.
$
It follows that
$
C_\varepsilon^l\in\rm M_{2m}(\mathcal H_{\delta_1}(\mathfrak A)).
$
\qed

\subsection{Gap convergence to the Canonical Center Bundle}
We have constructed the CCB
$\boldsymbol{\mathcal M}^c_{V,E,\varepsilon}$,
directly from the shift dynamics. As a coordinate-free geometric object, it is independent of any finite-range approximation. The purpose of this subsection is to show that the center structures arising from finite truncations converge to this bundle. More precisely, after embedding the finite-range center spaces into a common Banach space, we prove that they converge in the gap topology to the CCB. Thus the CCB appears as the canonical geometric limit of the finite-range center dynamics.

For each finite $l$, let
$
\mathcal E_{\varepsilon,c}^l(E,\theta)\subset \mathbb C^{2l}
$
be the center subspace of the companion cocycle
$A_\varepsilon^l(E,\theta)$ defined in Theorem \ref{thm:sec4-main}. We regard it as a subspace of $\mathcal X_\xi$ by
the zero-extension map
$$
\iota_l:\mathbb C^{2l}\rightarrow \mathcal X_\xi,
\qquad
\iota_l(a_{-l},\ldots,a_{l-1})=(\ldots,0,a_{-l},\ldots,a_{l-1},0,\ldots),
$$
and write
$$
\widetilde{\mathcal E}_{\varepsilon,c}^l(\theta):=
\iota_l\mathcal E_{\varepsilon,c}^l(\theta)
\subset \mathcal X_\xi .
$$

We use the standard gap topology on closed subspaces. 
\begin{definition}\cite{kato}
Let $X$ be a Banach space, and let $M,N\subset X$ be closed subspaces. Define
$$
\delta(M,N):=\sup_{\substack{u\in M\\ \|u\|_X=1}}
\operatorname{dist}(u,N),\qquad\operatorname{dist}(u,N):=\inf_{v\in N}\|u-v\|_X .
$$
The gap distance between $M$ and $N$ is
$$
\widehat\delta(M,N):=\max\{\delta(M,N),\delta(N,M)\}.
$$
We say that $M_l$ converges to $M_\infty$ in the gap topology if
$\widehat\delta(M_l,M_\infty)\to 0 .
$
\end{definition}

The gap topology is the natural geometric topology on the Grassmannian of closed subspaces of a Banach space, which is independent of the choice of coordinates or frames. Consequently, gap convergence provides an intrinsic notion of bundle convergence, describing the convergence of the entire family of subspaces as geometric objects rather than merely the convergence of  bases, or individual vectors.

\begin{proposition}
\label{prop:gap}
For every fixed $E$,
$$\sup_{\theta\in\Omega}\widehat\delta\left(
\widetilde{\mathcal E}_{\varepsilon,c}^l(\theta),
\mathcal M_{V,E,\varepsilon}^{c}(\theta)\right)
\rightarrow 0\qquad (l\to\infty).
$$
in the gap topology of closed subspaces of $\mathcal X_\xi$.
\end{proposition}
\begin{proof}
By the Algebraic Bridge  (Proposition \ref{thm:bridge1}), the space
$\widetilde{\mathcal E}_{\varepsilon,c}^l(E,\theta)$ is spanned by 
$
\widetilde u_{\varepsilon}^{l,0}(E,\theta),\ldots,
\widetilde u_{\varepsilon}^{l,2m-1}(E,\theta)\in \mathcal X_\xi .
$
Here
$
\widetilde u_{\varepsilon}^{l,p}(E,\theta):=
\iota_l\Pi_{[-l,l-1]} u_{\varepsilon}^{l,p}(E,\theta), 0\le p\le 2m-1,
$
where $\Pi_{[-l,l-1]}$ denotes the projection onto the coordinates
$\{-l,\ldots,l-1\}$. 
Define the linear maps $\widetilde U_{\varepsilon}^{l}(E,\theta),U_{\varepsilon}^{\infty}(E,\theta):\C^{2m}\to \mathcal X_\xi$ by
$$
\widetilde U_{\varepsilon}^{l}(E,\theta)c
:=\sum_{p=0}^{2m-1}c_p\widetilde u_{\varepsilon}^{l,p}(E,\theta),\quad U_{\varepsilon}^{\infty}(E,\theta)c
:=\sum_{p=0}^{2m-1}c_p u_{\varepsilon}^{\infty,p}(E,\theta).
$$
Thus
$
\widetilde{\mathcal E}_{\varepsilon,c}^l(E,\theta)
=\operatorname{Ran}\widetilde U_{\varepsilon}^{l}(E,\theta),
\mathcal M_{V,E,\varepsilon}^{c}(\theta)
=\operatorname{Ran}U_{\varepsilon}^{\infty}(E,\theta).
$
By Lemma \ref{lem:uest}, we have
$
\sup_{\theta\in\Omega}\left\|\widetilde U_{\varepsilon}^{l}(E,\theta)-
U_{\varepsilon}^{\infty}(E,\theta)\right\|_{\mathcal B(\mathbb C^{2m},\mathcal X_\xi)}
\rightarrow 0 .
$
Moreover, the columns of $U_{\varepsilon}^{\infty}(E,\theta)$ form a basis of
$\mathcal M_{V,E,\varepsilon}^{c}(\theta)$. Hence
$U_{\varepsilon}^{\infty}(E,\theta)$ is injective. Since the domain is
finite-dimensional, there exists $\sigma>0$ such that for all $c\in  \C^{2m}$,
$
\|U_{\varepsilon}^{\infty}(E,\theta)c\|_{\xi}\ge \sigma |c|.
$
Therefore, for all sufficiently large $l$,
$
\|\widetilde U_{\varepsilon}^{l}(E,\theta)c\|_{\xi}\ge \frac{\sigma}{2}|c|.
$

We first prove
$
\delta\left(\widetilde{\mathcal E}_{\varepsilon,c}^l(E,\theta),\mathcal M_{V,E,\varepsilon}^{c}(\theta)\right)\to 0 .
$
Take
$
w_l\in \widetilde{\mathcal E}_{\varepsilon,c}^l(E,\theta),
\|w_l\|_{\xi}=1 .
$
Then there exists $c_l\in\mathbb C^{2m}$ such that
$
w_l=\widetilde U_{\varepsilon}^{l}(E,\theta)c_l .
$
The uniform lower bound above gives
$$
1=\|w_l\|_{\xi}
=\|\widetilde U_{\varepsilon}^{l}(E,\theta)c_l\|_{\xi}
\ge \frac{\sigma}{2}|c_l|,
$$
and hence
$|c_l|\le \frac{2}{\sigma}.$
Now set
$
v_l:=U_{\varepsilon}^{\infty}(E,\theta)c_l
\in\mathcal M_{V,E,\varepsilon}^{c}(\theta).
$
Then
$
\operatorname{dist}\left(w_l,\mathcal M_{V,E,\varepsilon}^{c}(\theta)
\right)\le\|w_l-v_l\|_{\xi}
$
and therefore
$$
\begin{aligned}
\|w_l-v_l\|_{\xi}&=\left\|
\left(\widetilde U_{\varepsilon}^{l}(E,\theta)-U_{\varepsilon}^{\infty}(E,\theta)
\right)c_l\right\|_{\xi}\le\left\|\widetilde U_{\varepsilon}^{l}(E,\theta)
-U_{\varepsilon}^{\infty}(E,\theta)
\right\|_{\mathcal B(\mathbb C^{2m},\xi)}
|c_l|  \\
&\le
\frac{2}{\sigma}\left\|
\widetilde U_{\varepsilon}^{l}(E,\theta)-U_{\varepsilon}^{\infty}(E,\theta)
\right\|_{\mathcal L(\mathbb C^{2m},\mathcal X_\xi)}
\rightarrow 0 .
\end{aligned}
$$
Taking the supremum over all unit vectors
$w_l\in\widetilde{\mathcal E}_{\varepsilon,c}^l(E,\theta)$ yields
$
\delta\left(\widetilde{\mathcal E}_{\varepsilon,c}^l(E,\theta),\mathcal M_{V,E,\varepsilon}^{c}(\theta)\right)\to 0 .
$

Conversely, we prove
$
\delta\left(\mathcal M_{V,E,\varepsilon}^{c}(\theta),\widetilde{\mathcal E}_{\varepsilon,c}^l(E,\theta)\right)\to 0 .
$
Take
$
w\in\mathcal M_{V,E,\varepsilon}^{c}(\theta),
\|w\|_{\xi}=1 .
$
Since the columns of $U_{\varepsilon}^{\infty}(E,\theta)$ form a basis of
$\mathcal M_{V,E,\varepsilon}^{c}(\theta)$, there exists a unique
$c\in\mathbb C^{2m}$ such that
$
w=U_{\varepsilon}^{\infty}(E,\theta)c .
$
The lower bound for $U_{\varepsilon}^{\infty}(E,\theta)$ gives
$$
1=\|w\|_{\xi}
=\|U_{\varepsilon}^{\infty}(E,\theta)c\|_{\xi}\ge \sigma |c|,
$$
hence
$
|c|\le \frac{1}{\sigma}.
$
Define
$
w_l:=\widetilde U_{\varepsilon}^{l}(E,\theta)c\in\widetilde{\mathcal E}_{\varepsilon,c}^l(E,\theta).
$
Then
$$
\operatorname{dist}\left(w,\widetilde{\mathcal E}_{\varepsilon,c}^l(E,\theta)\right)\le\|w-w_l\|_{\xi},
$$
and
$$
\|w-w_l\|_{\xi}=\left\|\left(U_{\varepsilon}^{\infty}(E,\theta)-
\widetilde U_{\varepsilon}^{l}(E,\theta)\right)c\right\|_{\xi}\le\frac{1}{\sigma}
\left\|U_{\varepsilon}^{\infty}(E,\theta)-
\widetilde U_{\varepsilon}^{l}(E,\theta)\right\|_{\mathcal B(\mathbb C^{2m},\mathcal X_\xi)}
\rightarrow 0 .
$$
Taking  supremum over all unit vectors
$w\in\mathcal M_{V,E,\varepsilon}^{c}(\theta)$ gives
$
\delta\left(\mathcal M_{V,E,\varepsilon}^{c}(\theta),\widetilde{\mathcal E}_{\varepsilon,c}^l(E,\theta)\right)\to 0 .
$

Combining the two one-sided estimates, we obtain
$
\widehat\delta\left(\widetilde{\mathcal E}_{\varepsilon,c}^l(E,\theta),
\mathcal M_{V,E,\varepsilon}^{c}(\theta)\right)\to 0 .
$
This proves the convergence in the gap topology.
\end{proof}

\subsection{Center Cocycle}
We now prove several basic properties of the center cocycle, which will be crucial in the subsequent sections.

\begin{lemma}
    \label{thm:center}
  Suppose $W(\cdot)\in \mathfrak A$. For every $l\in[ l_a,\infty]$ and every
$|\varepsilon|\le \varepsilon_{a}$, 
there exists an invertible  matrix-valued function
$
C_\varepsilon^l(\zeta,\theta)\in \rm M_{2m}(\mathcal H_{\delta_1}(\mathfrak A))  ,
$
such that
\begin{equation}\label{shift}
SU_\varepsilon^l(\zeta,\theta)
=U_\varepsilon^l(\zeta,T\theta)C_\varepsilon^l(\zeta,\theta),
\end{equation}
and there exists  constant $C$, such that the following hold: 

\begin{enumerate}
\item  $\sup_{|\varepsilon|\leq \varepsilon_a}\|C_{\varepsilon}^{l}(\zeta,\theta)-C_{\varepsilon}^{\infty}(\zeta,\theta)\|_{\mathfrak A,\delta_{1}}\to 0.$

\item
$
\|C_\varepsilon^l(\zeta,\theta)\|_{\mathfrak A,\delta_{1}}\le C,
\|(C_\varepsilon^l(\zeta,\theta))^{-1}\|_{\mathfrak A,\delta_{1}}\le C.
$

\item
 For all $|\varepsilon|<\varepsilon_a$,
$
\|C_\varepsilon^l(\zeta,\theta)-C_0^l(\zeta)\|_{\mathfrak A,\delta_{1}}
\le C|\varepsilon|.
$
\end{enumerate}
\end{lemma}

\begin{proof}
    (1) Applying the left inverse \eqref{eq:left-inverse} to \eqref{shift}, we have 
\begin{equation}\label{equ:cenrep}
    C_{\varepsilon}^{l}(\zeta,\theta)=\bigl(\mathbf{G}^{l}_\varepsilon(\zeta,T\theta)\bigr)^{-1}U_\varepsilon^{l}(\zeta,T\theta)^*\mathcal W_{\xi}SU_\varepsilon^{l}(\zeta,\theta).
\end{equation}
By  \eqref{equ:basisbound}, \eqref{equ:basiscon}, \eqref{equ:gramcon},  and triangle inequality, we get
  $\sup_{|\varepsilon|\leq \varepsilon_a}\bigl\|C_{\varepsilon}^{l}(\zeta,\theta)-C^{\infty}_\varepsilon(\zeta,\theta)\bigr\|_{\mathfrak A,\delta_1}\to0.$

(2) By \eqref{equ:cenrep},  the boundedness of $\|C_\varepsilon^{l}(\zeta,\theta)\|_{\mathfrak A,\delta_1}$ follows   \eqref{equ:basisbound} and \eqref{eq:Gbound}.
Next, we prove $C_\varepsilon^l(\zeta,\theta)$ is invertible.
Let
$
c=(c_0,\dots,c_{2m-1})^\top\in\C^{2m}
$
satisfy
$
C_\varepsilon^{l}(\zeta,\theta)c=0,
$
which implies that
$
SU_\varepsilon^{l}(\zeta,\theta)c=U_\varepsilon^{l}(\zeta,T\theta)C_\varepsilon^{l}(\zeta,\theta)c=0.
$
Since $S$ is injective on $\mathcal{X}_{\xi}$, it follows that
$
U_\varepsilon^{l}(\zeta,\theta)c=0.
$
Because the columns of $U_\varepsilon^{l}(\zeta,\theta)$ are linearly independent, we must have
$c=0,$
and therefore $C_\varepsilon^{l}(\zeta,\theta)$ is  invertible.
Moreover, applying the left inverse \eqref{eq:left-inverse} to  $ U_\varepsilon^l(\zeta,\theta) C_\varepsilon^l(\zeta,\theta)^{-1}=
S^{-1}\,U_\varepsilon^l(\zeta,T\theta)\,
$, we obtain $$C_\varepsilon^{l}(\zeta,\theta)^{-1}=\bigl(\mathbf G_\varepsilon^{(l)}(\zeta,\theta)\bigr)^{-1}
\bigl(U_\varepsilon^{l}(\zeta,\theta)\bigr)^*\mathcal W_{\xi}
S^{-1}
U_\varepsilon^{l}(\zeta,T\theta).$$ 
By  \eqref{equ:basisbound} and \eqref{eq:Gbound}, it follows that $\|C_\varepsilon^{l}(\zeta,\theta)^{-1}\|_{\mathfrak A,\delta_1}\leq C.$

(3) follows from  \eqref{equ:basisbound}, \eqref{equ:basisper}, \eqref{eq:invG}, \eqref{equ:cenrep}  and triangle inequality.

\end{proof}

The companion structure of the center cocycle provides a convenient way to compute the spectrum  of $C_{0}^{\infty}(E)$.
\begin{corollary}\label{cor:zero}
The spectrum of the unperturbed center cocycle is exactly
$
\sigma\bigl(C_{0}^{\infty}(E)\bigr) = \{\mu_{1}(E),\dots,\mu_{2m}(E)\},
$
where $\mu_{1}(E),\dots,\mu_{2m}(E)$ are the roots of $L_{E}(z)$ in the annulus $\{z\in\mathbb{C}: r_-(E) \le |z| \le r_+(E)\}$, counted with algebraic multiplicities.
\end{corollary}
\begin{proof}
Define the monic polynomial
$$
Q_E(z):=\prod_{j=1}^{2m}\bigl(z-\mu_{j}(E)\bigr)=
z^{2m}+\beta_{2m-1}(E)z^{2m-1}+\cdots+\beta_{1}(E)z+\beta_{0}(E).
$$
For every $k\in\Z$, using the definition of $u^{\infty,p}_{0}(E)$ and of the coefficients of
$Q_E$, we have
\begin{align*}
u^{\infty,2m}_{0;k}(E)
+\sum_{q=0}^{2m-1}\beta_{q}(E) u^{\infty,q}_{0;k}(E)
&=\frac{1}{2\pi i}\int_{\Gamma_c}\frac{z^{k+2m}+\sum_{q=0}^{2m-1}\beta_{q}(E)z^{k+q}}{L_{E}(z)} dz\\
&=\frac{1}{2\pi i}\int_{\Gamma_c}z^k \frac{Q_E(z)}{L_{E}(z)} dz.
\end{align*}
By construction, the zeros of $Q_E$ are precisely the zeros of
$L_{E}$ lying in $\{z\in\mathbb C: r_-(E) \le |z| \le r_+(E)\}$, with the same algebraic multiplicities.
 Cauchy's theorem on the annulus gives
$
\int_{\Gamma_c}z^k \frac{Q_E(z)}{L_{E}(z)} dz=0.
$
Thus
$
u^{\infty,2m}_{0;k}(E)+\sum_{q=0}^{2m-1}\beta_{q}(E) u^{\infty,q}_{0;k}(E)=0.
$
On the other hand, by \eqref{eq:defa} and  linear independence of
$
u^{\infty,0}_{0}(E),\dots,u^{\infty,2m-1}_{0}(E),
$
we obtain
$$
\alpha^{\infty}_{0,q}(E)=\beta_{q}(E),
\qquad q=0,\dots,2m-1.
$$
And characteristic polynomial of
the matrix $C_{0}^{\infty}(E)$  is
$$
\det\bigl(\lambda -C_{0}^{\infty}(E)\bigr)
=\lambda^{2m}+\alpha^{\infty}_{0,2m-1}(E)\lambda^{2m-1}
+\cdots+\alpha^{\infty}_{0,1}(E)\lambda+\alpha^{\infty}_{0,0}(E).
$$
 We get
$\det\bigl(\lambda -C_{0}^{\infty}(E)\bigr)=Q_E(\lambda).$
Therefore,
$
\sigma\bigl(C_{0}^{\infty}(E)\bigr)=\{\mu_{1}(E),\dots,\mu_{2m}(E)\}.
$

\end{proof}

\subsection*{Proof of Theorem \ref{thm:Johnson_generalization}:} 
\textbf{``Only if'' part:}
For every sufficiently large $l$, there exists
$
 E_l\in \sigma(\mathbf{L}_{V_l,\varepsilon W,T,\theta})
$
such that
$
| E_l-E|<C_0e^{-h l}.
$
Suppose, by contradiction, that
$(T, C^\infty_\varepsilon(E,\theta))
$
is uniformly hyperbolic.
By Theorem \ref{thm:centert} and Lemma \ref{thm:center}, the cocycle $C^\infty_\varepsilon(\zeta,\theta)$ is analytic in $\zeta$,
and
$
\|C^l_\varepsilon(\zeta,\theta)-C^\infty_\varepsilon(\zeta,\theta)\|_{\mathfrak A,\delta_1}
\to 0 ,l\to\infty.
$
Hence we have
\begin{eqnarray*}
\sup_{\theta\in\Omega}\|C^l_\varepsilon( E_l,\theta)-C^\infty_\varepsilon(E,\theta)\|
&\leq&
\sup_{\theta\in\Omega}\|C^l_\varepsilon( E_l,\theta)-C^\infty_\varepsilon( E_l,\theta)\| +\sup_{\theta\in\Omega}
\|C^\infty_\varepsilon( E_l,\theta)-C^\infty_\varepsilon(E,\theta)\| \\
&\leq&
C\|C^l_\varepsilon(\zeta,\theta)-C^\infty_\varepsilon(\zeta,\theta)\|_{\mathfrak A,\delta_1}
+C| E_l-E|\to 0.
\end{eqnarray*}
Since uniform hyperbolicity is open, it follows that for all sufficiently large $l$,
$(T, C^l_\varepsilon( E_l,\theta))
$
is uniformly hyperbolic.
By the Algebraic Bridge (Proposition \ref{thm:bridge1}), we have a fiberwise linear isomorphism
$$
 \Pi_{[-l,l-1]}(\theta):\mathcal M^c_{V_l,E_l,\varepsilon}(\theta) \xrightarrow{\cong} \mathcal{C}_{2l}(\theta) =\mathcal E^l_{\varepsilon,c}(\theta)
,
$$
and we have
    $$
    \Pi_{[-l,l-1]}(T\theta)\circ \mathscr T\big|_{\mathcal M_{V_l,E_l,\varepsilon}^c(\theta)}
    =  A_{\varepsilon}^{l}(E_l,\theta) \circ \Pi_{[-l,l-1]}(\theta).
    $$ 
It follows that $$ A_{\varepsilon}^{l}(E_l,\theta)\Pi_{[-l,l-1]}(\theta)U_\varepsilon^l(E_l,\theta)=\Pi_{[-l,l-1]}(T\theta)U_\varepsilon^l(E_l,T\theta)C^l_\varepsilon(E_l,\theta).$$
Therefore the cocycle $C^l_\varepsilon( E_l,\theta)$ is precisely the restriction of
$A^l_\varepsilon( E_l,\theta)$ to the center subspace
$\mathcal E^l_{\varepsilon,c}(\theta)$, and the center restriction
is uniformly hyperbolic.

By Theorem~\ref{thm:sec4-main}, the  cocycle
$(T,A^l_\varepsilon( E_l,\theta))$ is partially hyperbolic, and  admits an invariant splitting
$
\mathbb C^{2l}=\mathcal E^l_{\varepsilon,s}(\theta)\oplus\mathcal E^l_{\varepsilon,c}(\theta)
\oplus\mathcal E^l_{\varepsilon,u}(\theta),
 \theta\in\Omega.
$
 It immediately implies that $(T,A^l_\varepsilon( E_l,\theta))$ is uniformly hyperbolic. This directly contradicts the classical Johnson's theorem for finite-range operators \cite{HaroPuig}, which asserts that $E_l \in \sigma(\mathbf{L}_{V_l,\varepsilon W,T,\theta})$ if and only if $(T, A_{\varepsilon}^{l}(E_l,\theta))$ is not uniformly hyperbolic. This completes the ``Only if'' part.

\medskip

\noindent
\textbf{``If'' part:}
First by the Sacker-Sell theory \cite{SS1,SS}, there exist $\theta_0\in\Omega$ and
$v_0\in\C^{2m}\setminus\{0\}$ such that
\begin{equation}\label{bound}
    \|(C^\infty_\varepsilon)_n(E,\theta_0)v_0\|\leq C,\quad 
 \forall n\in\Z.
\end{equation}
By trivialization of CCB (Theorem~\ref{thm:centert}), there exists
$$
U^\infty_\varepsilon(\zeta,\theta)
=\bigl(
u^{\infty,0}_\varepsilon(\zeta,\theta),\dots,
u^{\infty,2m-1}_\varepsilon(\zeta,\theta)
\bigr),
$$
such that
$
u^{\infty,p}_\varepsilon(\zeta,\theta)\in \mathcal M^c_{\zeta,V,\varepsilon}(\theta),
0\le p\le 2m-1,
$
and
$
SU^\infty_\varepsilon(\zeta,\theta)
=U^\infty_\varepsilon(\zeta,T\theta)\,
C^\infty_\varepsilon(\zeta,\theta).
$
Iterating this identity, we obtain
\begin{equation}\label{lin}
S^nU^\infty_\varepsilon(E,\theta_0)v_0
=U^\infty_\varepsilon(E,T^n\theta_0)
(C^\infty_\varepsilon)_n(E,\theta_0)v_0.
\end{equation}

Let
$
u:=U^\infty_\varepsilon(E,\theta_0)v_0.
$
Since each
$
u^{\infty,p}_\varepsilon(E,\theta)\in \mathcal M^c_{E,V,\varepsilon}(\theta),
$
it follows that
$
u\in \mathcal M^c_{E,V,\varepsilon}(\theta_0).
$
In particular, $u$ is a nontrivial solution of
$
\mathbf{L}_{V,\varepsilon W,T,\theta_0}u=Eu.
$
Now, by \eqref{lin}, the $n$-th coordinate of $u$ is given by
$$
u_n=\bigl(S^nU^\infty_\varepsilon(E,\theta_0)v_0\bigr)_0 =
\sum_{k=0}^{2m-1}u^{\infty,k}_{\varepsilon;0}(E,T^n\theta_0)
\bigl((C^\infty_\varepsilon)_n(E,\theta_0)v_0\bigr)_k.
$$
By \eqref{equ:basisbound}, there exists $\widetilde M_U$ such that 
we have
$
\sup_{\theta\in\Omega}|u^{\infty,p}_{\varepsilon;0}(E,\theta)|\leq\widetilde M_U,
 0\le p\le 2m-1.
$
Together with \eqref{bound}, this yields
$$
|u_n|\leq\sum_{k=0}^{2m-1}
\sup_{\theta\in\Omega}|u^{\infty,k}_{\varepsilon;0}(E,\theta)|
\|(C^\infty_\varepsilon)_n(E,\theta_0)v_0\|
\leq 2m\widetilde M_U C.
$$
Hence $u$ is a nontrivial bounded solution of
$
\mathbf{L}_{V,\varepsilon W,T,\theta_0}u=Eu.
$
By the standard Weyl's criterion, we conclude that
$
E\in \sigma(\mathbf{L}_{V,\varepsilon W,T,\theta_0}).
$
Since the spectrum is phase-independent, it follows that
$
E\in \sigma(\mathbf{L}_{V,\varepsilon W,T,\theta}).
$
This completes the proof.\qed

\section{The Center Thouless Formula and Spectral Regularity}

\subsection{The Center Thouless formula}

Having established the global trivialization of the center bundle, we proceed to activate the left spectral branch of our framework (see Figure \ref{fig:logic_architecture}). The objective of this subsection is to establish the Center Thouless Formula  (Proposition \ref{prop:thouless}). From the perspective of potential theory, the limit \eqref{eq:com} formulates an explicit analytic cancellation. We first record the weak convergence of the density of
states measures, which will be used to pass the logarithmic potential term in
the finite-range Thouless formula to the limit.

\begin{lemma}\label{lemma:IDSweak}
Let $d \widehat{\mathcal{N}}$ and $d\widehat{\mathcal{N}}_l$ denote the density of states measures associated with
$\mathbf{L}_{V,\varepsilon W,\alpha,\theta}$ and $\mathbf{L}_{V_l,\varepsilon W,\alpha,\theta}$, respectively.
Then
$
d\widehat{\mathcal{N}}_l \overset{*}{\rightharpoonup} d \widehat{\mathcal{N}}, l\to\infty,
$
that is,
$$
\int_{\mathbb R}\varphi(E)d\widehat{\mathcal{N}}_l(E)
\longrightarrow
\int_{\mathbb R}\varphi(E)d \widehat{\mathcal{N}}(E), \varphi\in C_b(\mathbb R),
$$
where $C_b(\mathbb R)$ denotes the space of bounded continuous functions on $\mathbb R$.
\end{lemma}
\begin{proof}Since $V$ is analytic, its Fourier truncation $V_l$ converges uniformly to $V$ on $\mathbb{T}$. This implies  the uniform operator norm convergence $\sup_{\theta\in\mathbb{T}^d} \|\mathbf{L}_{V_l,\varepsilon W,\alpha,\theta} - \mathbf{L}_{V,\varepsilon W,\alpha,\theta}\| \to 0$ as $l \to \infty$. By the standard continuous functional calculus for uniformly convergent families of bounded self-adjoint operators, for any $\varphi \in C_b(\mathbb{R})$, we obtain the uniform operator limit:$\sup_{\theta\in\mathbb{T}^d} \|\varphi(\mathbf{L}_{V_l,\varepsilon W,\alpha,\theta}) - \varphi(\mathbf{L}_{V,\varepsilon W,\alpha,\theta})\| \to 0.$ This direct implies that :$$\begin{aligned}
\left|\int_{\mathbb{R}}\varphi(E)d\widehat{\mathcal{N}}_l(E) - \int_{\mathbb{R}}\varphi(E)d\widehat{\mathcal{N}}(E)\right|
&\leq \int_{\mathbb{T}^d}\left|\bigl\langle\delta_0,\bigl(\varphi(\mathbf{L}_{V_l,\varepsilon W,\alpha,\theta}) - \varphi(\mathbf{L}_{V,\varepsilon W,\alpha,\theta})\bigr)\delta_0\bigr\rangle\right|d\theta\\
&\leq \sup_{\theta\in\mathbb{T}^d}\|\varphi(\mathbf{L}_{V_l,\varepsilon W,\alpha,\theta}) - \varphi(\mathbf{L}_{V,\varepsilon W,\alpha,\theta})\| \to 0.
\end{aligned}$$\end{proof}

\subsection*{Proof of Proposition \ref{prop:thouless}}
Let $m := m(E)$. By the classical Thouless formula for the finite-range operator \cite{HaroPuig}, the sum of the top $l$ Lyapunov exponents of the companion cocycle $(\alpha,A^l_\varepsilon(\zeta,\theta))$ satisfies, for all $l \ge l_a$ and $\zeta \in I_{\delta_1}$:
$$\sum_{i=1}^{l-m}L_{i,l}(\zeta) + \sum_{j=1}^{m}L_{l-m+j,l}(\zeta)
=
\int_\Sigma \log|\zeta-z|\,d\widehat{\mathcal{N}}_l(z) - \log|\hat{v}_l|,$$
where $\widehat{\mathcal{N}}_l$ denotes the integrated density of states of $\mathbf{L}_{V_l,\varepsilon W,\alpha,\theta}$. 

By the partial hyperbolicity established in Theorem \ref{thm:sec4-main}, the cocycle admits the invariant splitting $\mathbb{C}^{2l} = \mathcal{E}^l_{\varepsilon,u}(\theta) \oplus \mathcal{E}^l_{\varepsilon,c}(\theta) \oplus \mathcal{E}^l_{\varepsilon,s}(\theta)$. The Algebraic Bridge  (Proposition \ref{thm:bridge1}) explicitly identifies the center cocycle $C^l_\varepsilon(\zeta,\theta)$ as the exact restriction of $A^l_\varepsilon(\zeta,\theta)$ to the center bundle $\mathcal{E}^l_{\varepsilon,c}(\theta)$.
More precisely, $$A^l_\varepsilon(\zeta,\theta)|_{\mathcal{E}^l_{\varepsilon,c}(\theta)}=\Pi_{[-l,l-1]}(\theta+\alpha)U_\varepsilon^l(E_l,\theta+\alpha)C^l_\varepsilon(E_l,\theta)\big(\Pi_{[-l,l-1]}(\theta)U_\varepsilon^l(E_l,\theta))\big)^{-1},$$
where $\big(\Pi_{[-l,l-1]}U_\varepsilon^l(\zeta,\theta)\big)^{-1}$ denotes the inverse from
$\mathcal E_{\varepsilon,c}^l(\zeta,\theta)$ to $\mathbb C^{2m}$.
Consequently, by the conjugate invariance of the Lyapunov exponent, the $m$ smallest non-negative exponents precisely constitute the Lyapunov spectrum of the center cocycle. In particular,
$$\sum_{j=1}^{m}L_{l-m+j,l}(\zeta)
=\sum_{j=1}^{m}L_j\bigl(C_{\varepsilon}^{l}(\zeta,\theta)\bigr)
=: F_l(\zeta),$$
which is subharmonic on $I_{\delta_1}$.

Simultaneously, the restriction of $A^l_{\varepsilon}(\zeta,\theta)$ to the unstable bundle $\mathcal{E}^l_{\varepsilon,u}(\theta)$ defines a uniformly expanding analytic cocycle. Therefore, its sum of Lyapunov exponents, $\sum_{i=1}^{l-m}L_{i,l}(\zeta)$, is rigorously harmonic on $I_{\delta_1}$ \cite{positive,AJS}. This guarantees that the term$$H_l(\zeta) := -\log|\hat{v}_l| - \sum_{i=1}^{l-m}L_{i,l}(\zeta)$$is harmonic on $I_{\delta_1}$. The truncated Thouless formula thus can be rewritten as:
\begin{equation}\label{eq:finitethouless}F_l(\zeta) = \int_{\Sigma}\log|\zeta-z|d\widehat{\mathcal{N}}_l(z) + H_l(\zeta).\end{equation}

We now pass to the infinite-range limit. By Lemma \ref{lemma:IDSweak}, the weak-$*$ convergence $d\widehat{\mathcal{N}}_l \overset{*}{\rightharpoonup} d\widehat{\mathcal{N}}$ ensures that for every $\zeta \in I_{\delta_1} \setminus \Sigma$, the logarithmic potentials converge:$$\int_{\Sigma}\log|\zeta-z|\,d\widehat{\mathcal{N}}_l(z) \longrightarrow \int_{\Sigma}\log|\zeta-z|\,d\widehat{\mathcal{N}}(z).$$

Furthermore, by Lemma \ref{thm:center}, we have
$\sup_{|\varepsilon|\leq \varepsilon_a}\|C_{\varepsilon}^{l}(\zeta,\theta)-C_{\varepsilon}^{\infty}(\zeta,\theta)\|_{h,\delta_{1}}\to 0.$
If $\alpha\in \R\backslash\Q$,  the continuity of
Lyapunov exponents for analytic cocycles \cite{AJS} implies that
$$
F_l(\zeta)\rightarrow F_\infty(\zeta):=
\sum_{i=1}^{m}L_i\bigl(C_\varepsilon^\infty(\zeta,\theta)\bigr).
$$ 
If $\alpha\in\mathbb{T}^d$ and $\alpha$ is Diophantine, the same conclusion holds by
\cite{DK}. In both cases, the convergence is pointwise in $\zeta\in I_{\delta_1}$.
Moreover, the uniform convergence of the cocycles on compact subsets of
$I_{\delta_1}$ implies that the family $\{F_l\}$ is locally uniformly bounded. 
Evaluating $H_l(\zeta) = F_l(\zeta) - \int_{\Sigma}\log|\zeta-z|\,d\widehat{\mathcal{N}}_l(z)$, we observe that $\{H_l\}$ converges pointwise on the $I_{\delta_1} \setminus \Sigma$. Since the logarithmic potential is uniformly bounded from above on compact sets, its negation is bounded from below. Combined with the local boundedness of the convergent sequence $\{F_l\}$, the harmonic sequence $\{H_l\}$ is locally uniformly bounded from below. By Harnack's principle, the pointwise convergence upgrades to local uniform convergence on the entire domain $I_{\delta_1}$, yielding a harmonic limit $h_\varepsilon(\zeta)$.Taking the limit in \eqref{eq:finitethouless} yields
$$F_\infty(\zeta) = \int_{\Sigma}\log|\zeta-z|\,d\widehat{\mathcal{N}}(z) + h_\varepsilon(\zeta)$$
for all $\zeta \in I_{\delta_1} \setminus \Sigma$. Since both $F_\infty(\zeta)$ and the logarithmic potential are subharmonic on $I_{\delta_1}$ and differ exactly by a harmonic function almost everywhere, this identity extends uniquely to the entire domain $I_{\delta_1}$, completing the proof. \qed

\smallskip

Following the logical flow of Figure \ref{fig:logic_architecture}, the establishment of the Center Thouless Formula  unlocks the spectral regularity of the operator. We now deploy this  identity to prove the absolute and H\"older continuity of the IDS, which will ultimately provide the  measure-theoretic protection  in the localization proof. 

\subsection{Quantitative Regularity: H\"older Continuity.}
Before turning to the H\"older continuity of the IDS, we first prove that, for $\alpha \in \mathbb{R}\setminus\mathbb{Q}$, the quantity $\nu(E)$ coincides with Avila's acceleration.
\begin{lemma}\label{lem:acce}
    Assume $\alpha\in \mathbb{R}\setminus\mathbb{Q}$. Then for every $E$,
the quantity $\nu(E)$ defined in \eqref{acce} coincides with Avila's acceleration $\omega(E)$.
\end{lemma}
\begin{proof}
    For the Schr\"odinger cocycle $S_E^V(\theta)=\begin{pmatrix}
    E-V(\theta)&-1\\1&0
\end{pmatrix}$, we have the following quantitative version of Avila's Global theory.

\begin{theorem}\cite{GJYZ}\label{quanglobal}
    For $\alpha\in\R\backslash\Q$, $V\in C_h^\omega(\T,\R)$. For every $E\in\mathbb R$, there exist $p(E)\in[1,\infty]$ and non negative $\{{L}_i(E)\}_{i=1}^{p(E)}$, such that   ${L}_i(E):=\lim_{l\to\infty}L_i^l(E), 1\leq i\leq p(E),$ where $0\leq L_1^l(E)\leq\cdots\leq L_l^l(E)$ is the Lyapunov exponent of cocycle $(\alpha, A_{\varepsilon}^{l}(E,\theta)).$ And $$\omega(E)=\begin{cases}
        0,   &{L}_1(E)>0\\ 
        \#\{j:{L}_j(E)=0\},&{L}_1(E)=0
    \end{cases}.$$
\end{theorem}

By the Algebraic Bridge  (Proposition \ref{thm:bridge1})  identifies the center cocycle $C^l_\varepsilon(\zeta,\theta)$ as the exact restriction of $A^l_\varepsilon(\zeta,\theta)$ to the center bundle $\mathcal{E}^l_{\varepsilon,c}(\theta)$, it follows that  $L_{l-j+1}(A^{l}_{\varepsilon}(\theta,E))=L_{m(E)-j+1}((C_{\varepsilon}^{l}(E,\theta))$ for $1\leq j\leq m(E)$. By Lemma \ref{thm:center} (3),   $C_{\varepsilon}^{l}(E,\theta)\rightarrow C_{\varepsilon}^{\infty}(E,\theta)$ in analytic topology, by the continuity of Lyapunov exponent \cite{AJS}, we have $$L_j(E)=\lim_{l\to\infty}L_{l-j+1}(A^{l}_{\varepsilon}(\theta,E))=L_{m(E)-j+1}(C_{\varepsilon}^{\infty}(E,\theta)).$$
 By Theorem \ref{thm:sec4-main}, $L_{l-i+1}(A^{l}_{\varepsilon}(\theta,E))\geq \log r_u>0$ for all $i\in[m(E)+1,l]$ and $l\in[l_0,\infty).$ Therefore, $\limsup_{i\to\infty} L_i(A^{l}_{\varepsilon}(\theta,E))\geq \log r_u>0$ for $i\in[m(E)+1,\infty)$. By Theorem \ref{quanglobal}, it follows that
$$\omega(E)=\#\{1\leq j\leq m(E):\,L_j(E)=0\}=\#\Bigl\{1\leq j\leq m(E):
 L_j\bigl(C_{\varepsilon}^{\infty}(E,\theta)\bigr)=0 \Bigr\}.$$
Furthermore, by symplectic structure of $A_{\varepsilon}^{l}(E,\theta)$,
we have $L_{l-j+1}(A^{l}_{\varepsilon}(\theta,E))=-L_{l+j}(A^{l}_{\varepsilon}(\theta,E))$, since
$
L_{l-j+1}(A^{l}_{\varepsilon}(\theta,E))=L_{m(E)-j+1}((C_{\varepsilon}^{l}(E,\theta))$, we have
$$
L_j\bigl(C_{\varepsilon}^{l}(E,\theta)\bigr)
=-L_{2m-j+1}\bigl(C_{\varepsilon}^{l}(E,\theta)\bigr),\quad 1\leq j\leq m(E).
$$
Passing to the limit as $l\to\infty$, we obtain
$
L_{m+j}\bigl(C_{\varepsilon}^{\infty}(E,\theta)\bigr)=0,
 1\leq j\leq \omega(E),
$
and
$
L_{m+j}\bigl(C_{\varepsilon}^{\infty}(E,\theta)\bigr)<0,
 \omega(E)+1\leq j\leq m.
$
Therefore, $2\omega(E)=\#\bigl\{1\leq j\leq 2m(E): L_j(\alpha, C_{\varepsilon}^{\infty}) = 0\bigr\}.$
\end{proof}

We  recall the H\"older continuity of Lyapunov exponents.
\begin{theorem}[\cite{GYZ22}]\label{thm:holder-GYZ}
Let $\alpha\in \mathrm{DC}_d(\kappa,\tau)$ and let
$A_0\in \mathrm{GL}(q,\mathbb C)$
be a constant matrix.  For every $\epsilon>0$, there exists
$\varepsilon_1=\varepsilon_1(\alpha,h,q,\|A_0\|,\|A_0^{-1}\|,\epsilon)>0$
such that, if
$F\in C_h^\omega(\mathbb T^d,\mathfrak{gl}(q,\mathbb C))$
satisfies
$\|F\|_h\le \varepsilon_1,$  and if we set
$A:=A_0+F,$ then
for every
$\widetilde A\in C_h^\omega(\mathbb T^d,\mathrm{GL}(q,\mathbb C))$, one has
$$
\left|L_i(\alpha,A)-L_i(\alpha,\widetilde A)\right|
\le C_0\|\widetilde A-A\|_h^{\frac1{q}-\epsilon},
\qquad 1\le i\le q.
$$where $C_0$ is a constant depending on $\|A\|, d, \kappa, \tau, q, \epsilon$.
Moreover, let
$$
L_1(\alpha,A)=\cdots=L_{n_1}(\alpha,A)
>\cdots> L_{n_{k-1}+1}(\alpha,A)=\cdots=L_{n_{k-1}+n_k}(\alpha,A)
$$
be the Lyapunov exponent, and write
$
N_k=n_1+\cdots+n_k
$
with $N_0=0$.
 Then
$$
\left|
\sum_{\ell=N_{k-1}+1}^{N_k}\bigl(L_\ell(\alpha,\widetilde A)-L_\ell(\alpha,A)\bigr)
\right|\le C_1\|\widetilde A-A\|_h^{1/2},
$$
and, for every $N_{k-1}<\ell\le N_k$,
$$
\left|L_\ell(\alpha,\widetilde A)-L_\ell(\alpha,A)\right|
\le C_1\|\widetilde A-A\|_h^{\frac1{n_k}-\epsilon}.
$$
where $C_1$ is a constant depending on $A, d, \kappa, \tau, q, \epsilon$.
\end{theorem}

\subsection*{Proof of Theorem \ref{thm:IDSholder}}
Let $I=[\min V(\theta)-1,\max V(\theta)+1]$. Take $\varepsilon$ sufficiently small such that $\sigma(\mathbf{L}_{V,\varepsilon W,\alpha,\theta})\subset I.$
By Lemma \ref{prop:assume}, there exist finitely many open intervals $J^1,\dots,J^k$ covering $I$, radii $e^{-h}<r_-^{(j)}<1<r_+^{(j)}<e^h$, numbers $\eta_-^{(j)}>0,\eta_+^{(j)}>0$, integers $m_j\ge 0$ for $1\leq j\leq k$, and constant $\delta_0>0$ such that the following hold for every $j\in\{1,\dots,k\}$:
For every $E\in \overline{J^j}$, $|L_{E}(z)|\ge \delta_0$ for all $z$ satisfying 
    $$
    r_-^{(j)}e^{-\eta_-^{(j)}}\le |z|\le r_-^{(j)}e^{\eta_-^{(j)}},\quad r_+^{(j)}e^{-\eta_+^{(j)}}\le |z|\le r_+^{(j)}e^{\eta_+^{(j)}}
    $$
and
    $$
    \#\{z:\ r_-^{(j)}<|z|<r_+^{(j)},\ L_{E}(z)=0\}=2m_j,
    $$
    counting multiplicities.
For each fixed $j$, Theorem \ref{thm:centert} gives a number
$\varepsilon_{a,j}>0$ such that, whenever
$
|\varepsilon|<\varepsilon_{a,j},
$ there exists a center cocycle
$
C_{\varepsilon}^{j,\infty}(E,\theta)
\in C^\omega(J^j_{\delta_1}\times\mathbb T_h^d,\GL(2m_j,\mathbb C)),
$ where $\delta_1=\frac{\delta_0}{4}.$
By Lemma \ref{thm:center} (3), the
perturbed cocycle is uniformly close to the unperturbed one,  there exists $C_j$ such that
$
\|C_\varepsilon^{j,\infty}(E,\theta)-C_0^{j,\infty}(E)\|_{h,\delta_1}
\le C_j|\varepsilon|.
$
Hence, there exists $\varepsilon_{r,j}$, such that for  $|\varepsilon|<\varepsilon_{r,j}$, cocycle
$ C_\varepsilon^{j,\infty}(E,\cdot)$
satisfies all assumptions of Theorem \ref{thm:holder-GYZ} on $I_{\delta_1}$.

By Proposition \ref{prop:thouless}, we have $$\sum_{i=1}^m L_i(C_{\varepsilon}^{j,\infty}(\zeta,\theta))=\int_{I_{\delta_{1}}} \log|\zeta-t|d{\widehat{\mathcal{N}}}(t)+h_\varepsilon(\zeta),\qquad \zeta\in J^j_{\delta_{1}}.$$
Since $h_\varepsilon$ is harmonic on $I_{\delta_1}$, it follows that
$$
\begin{aligned}
|d{\widehat{\mathcal{N}}}(E-\varepsilon,E+\varepsilon)|
&\leq\sum_{i=1}^{m_j}\bigl|
L_i\bigl(C_{\varepsilon}^{j,\infty}(E+i\varepsilon,\theta)\bigr)
-L_i\bigl(C_{\varepsilon}^{j,\infty}(E,\theta)\bigr)\bigr|
+C\varepsilon\\
&\leq
\sum_{i=1}^{m_j-\nu(E)}\bigl|
L_i\bigl(C_{\varepsilon}^{j,\infty}(E+i\varepsilon,\theta)\bigr)
-L_i\bigl(C_{\varepsilon}^{j,\infty}(E,\theta)\bigr)\bigr|\\
&\quad+
\sum_{i=m_j-\nu(E)+1}^{m_j}
\bigl|L_i\bigl(C_{\varepsilon}^{j,\infty}(E+i\varepsilon,\theta)\bigr)-
L_i\bigl(C_{\varepsilon}^{j,\infty}(E,\theta)\bigr)\bigr|
+C\varepsilon\\
&\leq
C\varepsilon^{1/2}+C\varepsilon^{\frac{1}{2\nu(E)}-\epsilon}
+C\varepsilon\leq C\varepsilon^{\frac{1}{2\nu(E)}-\epsilon}. 
\end{aligned}
$$
Let $\varepsilon_3=\min_j\{\varepsilon_{r,j}\}$, by Proposition \ref{IDSeq}, the conclusion follows. \qed
\subsection{Absolute Continuity of the IDS}

First let's state the full measure reducibility of the center cocycle:\begin{theorem}\label{main:reduce}
Let $\alpha\in \mathrm{DC}_d$.  There exists $\varepsilon_{r}=\varepsilon_r(\alpha,V,W)>0,$ $h=h(\alpha,V,W)>0$  such that for all $|\varepsilon|<\varepsilon_r$, , there exists $\Theta\subseteq \sigma(\mathbf{L}_{V,\varepsilon W,\alpha,\theta})$ with Hausdorff dimension zero, such that if $E\in\sigma(\mathbf{L}_{V,\varepsilon W,\alpha,\theta})\backslash \Theta$,  there exists $m_E\in \mathbb{N}$, $B(E,\cdot)\in C_{h}^{\omega}(\T^{d},\GL(2m_E,\C))$,  such that
$$
B(E,\theta+\alpha)^{-1}C_{\varepsilon}^\infty(E,\theta)B(E,\theta)=\Lambda(E),
$$  where $\Lambda(E)\in\GL(2m_E,\C)$ is a diagonal matrix. More precisely, there exist finitely many open intervals $J^1,\ldots,J^k$ covering $\sigma(\mathbf L_{V,\varepsilon W,\alpha,\theta})$
and $m_E$ is constant on $J^j,1\leq j\leq k$.
\end{theorem}

To prove this, let us first recall the KAM theory of $\GL(q,\C)$ cocycles in \cite{WXYZ}.
 Let $A: I\rightarrow \GL(q,\C)$ be analytic in $\lambda\in I$.
Let $\Sigma(\lambda):=\Sigma(A(\lambda))$ be the set of eigenvalues of $A(\lambda)\in \GL(q,\C)$, and for any $u\in\T$, let
\begin{equation}\label{trans-func-2}
g(\lambda,u)=\prod_{\sigma_i,\sigma_j\in \Sigma(\lambda),\, {i\neq j}}(\sigma_i-e^{2\pi \ii u}\sigma_j).
\end{equation}
\begin{definition}
We say that $A(\lambda)$ satisfies the $non$-$degeneracy$ condition on an interval $I$, if there exists {$r\geq1$}, $c>0$ such that for $\forall u\in\T$, the following inequality holds for all $\lambda\in I$,
\begin{equation}\label{def-non-degeneracy}
\max_{0\leq l\leq r}|\frac{\partial^{l}g(\lambda,u)}{\partial\lambda^{l}}|\geq c,
\end{equation}
where $g$ is defined as in (\ref{trans-func-2}).
\end{definition}

 \begin{theorem}\cite{WXYZ}\label{Thm2}
Let $h>0,$ $\delta>0$, $\alpha\in\mathrm{DC}_{d}$, and $I\subset\R$ be an interval. 
Suppose that $A\in C_{\delta}^{\omega}(I,\mathrm{GL}(q,\C))$, and  it is non-degenerate on $I$ in the sense of (\ref{def-non-degeneracy}) with some {$r\geq1$}, $c>0$. 
Then there exists $\varepsilon_{0}=\varepsilon_0(\alpha,d,q,\delta,h,r,c,|A|_\delta,|A^{-1}|_\delta)>0$, and $\Theta\subseteq I$ with Hausdorff dimension zero, such that  if $F\in C^{\omega}_{h,\delta}(\T^{d}\times I,\mathfrak{gl}(q,\C))$ satisfying $\|F\|_{h,\delta}\leq\varepsilon_{0}$ and   $\lambda\in I\backslash\Theta$, the cocycle $(\alpha,  A(\lambda)+F(\cdot, \lambda))$ is reducible, i.e., there exists $B_{\lambda}\in C_{\frac{h}{4}}^{\omega}(\T^{d},\GL(q,\C))$ such that
$$
B_\lambda^{-1}(\cdot+\alpha)(A(\lambda)+F(\cdot, \lambda))B_\lambda(\cdot)=\tilde{A}(\lambda)\in \mathrm{GL}(q,\C).
$$
In addition, the $\tilde{A}(\lambda)$ has simple eigenvalues.
\end{theorem}

\subsection*{Proof of Theorem \ref{main:reduce}}

Without loss of generality, we may assume that $V$ has minimal positive period $1$.
By the finite-covering argument in the proof of Theorem \ref{thm:IDSholder}, there exist finitely many open intervals $J^1,\dots,J^k$ covering $I=[\min V(\theta)-1,\max V(\theta)+1]$.
For each fixed $1\leq j\leq k$, Theorem \ref{thm:centert} gives a number
$\varepsilon_{a,j}>0$ such that, whenever
$
|\varepsilon|<\varepsilon_{a,j},
$ there exists a center cocycle
$
C_{\varepsilon}^{j,\infty}(E,\theta)
\in C^\omega(J^j_{\delta_1}\times\mathbb T_h^d,\GL(2m_j,\mathbb C)),
$ where $\delta_1=\frac{\delta_0}{4}.$ Furthermore by Lemma \ref{thm:center} (3), the
perturbed cocycle is uniformly close to the unperturbed one: there exists $C_j$ such that
$
\|C_\varepsilon^\infty(E,\theta)-C_0^\infty(E)\|_{h,\delta_1}
\le C_j|\varepsilon|.
$
To apply Theorem \ref{Thm2}, the key reduces to prove $C_0^\infty(E)$ satisfy the non-degeneracy condition \eqref{def-non-degeneracy}. Note  Corollary \ref{cor:zero} shows that $
\sigma\bigl(C_{0}^{\infty}(E)\bigr)
=\{\mu_{1}(E),\dots,\mu_{2m}(E)\},
$
where $
\mu_{1}(E),\dots,\mu_{2m}(E)
$ is the zeros of $L_{E}(z)$ in $\{z\in\mathbb C: r_-(E) \le |z| \le r_+(E)\}$. 
The crucial observation is the  following: 
\begin{lemma}\label{lem:G-analytic}
Let $\mu_1(E),\dots,\mu_{2m}(E)$
be the zeros of $L_{E}(z)$ in $\{z\in\mathbb C: r_-(E) \le |z| \le r_+(E)\}$, counted with multiplicity. For $\omega\in\mathbb C$ define
$$
G(E,\omega):=\prod_{\substack{1\le i,j\le 2m\\ i\neq j}}\bigl(\mu_i(E)-\omega \mu_j(E)\bigr).
$$
Then we have:
\begin{enumerate}
\item $(E,u)\longmapsto G(E,e^{2\pi i u})$ is real-analytic on $I\times\mathbb T$.
\item Assume in addition that $V(\theta)$ has minimal positive period $1$. Then there exist an integer $r\ge 0$ and a constant $c_\infty>0$ such that
$
\max_{0\le k\le r}\left|\partial_E^k G(E,e^{2\pi i u})\right|
\ge 2c_\infty, (E,u)\in I\times\mathbb T.
$
\end{enumerate}
\end{lemma}
\begin{proof}
(1): For $n\ge 1$, define 
$s_n(E):=\sum_{j=1}^{2m}\bigl(\mu_j(E)\bigr)^n.$
By the residue theorem,
$$
s_n(E)=\frac{1}{2\pi i}\int_{\Gamma_c}z^n \frac{\partial_z L_{E}^{}(z)}{L_{E}^{}(z)} dz.
$$
Since $L_{E}^{}(z)=V(z)-E$ is analytic in $(E,z)$ and $|L_{E}^{}(z)|\ge \delta_0$ on $I\times \Gamma_c$, the integrand is analytic in $E$ and uniformly bounded on $I\times\Gamma_c$. Therefore $s_n(E)$ is analytic in $E$.
Now $G(E,\omega)$ is a symmetric polynomial in the roots
$
\mu_1(E),\dots,\mu_{2m}(E),
$
with coefficients polynomially depending on $\omega$. By the fundamental theorem of symmetric polynomials and Newton identities, there exists a polynomial $ P$ such that
$$
G(E,\omega)= P\bigl(\omega,s_1(E),\dots,s_{2m}(E)\bigr).
$$
Since each $s_n(E)$ is analytic, it follows that $(E,\omega)\mapsto G(E,\omega)$ is analytic.

(2) We first show  for every fixed $u\in\mathbb T$,
$E\longmapsto G(E,e^{2\pi i u})$
is not identically zero on $I$. Fix $\omega=e^{2\pi i u}$.

{\bf Case 1: $\omega=1$.}
Then
$
G(E,1)=\prod_{i\neq j}(\mu_i(E)-\mu_j(E)).
$
If $G(E,1)\equiv 0$ on $I$, then for every $E\in I$ the annular zeros of $L_{E}^{}$ contain a repeated root. A repeated annular root $z$ satisfies
$V(z)-E=0, V'(z)=0.$
Thus every $E\in I$ belongs to the set
$$\{V(z): r_-(E)\leq|z|\leq r_+(E),\ V'(z)=0\}.$$
Since $V$ is non-constant and analytic in a neighborhood of $\overline{A_c}$, the set
$$\{z:r_-(E)\leq|z|\leq r_+(E), V'(z)=0\}$$
is finite. This is a contradiction. Hence $G(\cdot,1)\not\equiv 0$.

{\bf Case 2: $\omega\neq 1$.}
Assume $G(E,\omega)\equiv 0$.
Let $\mathcal C\subset I$ be the set of energies for which $L_{E}$ has a multiple zero in $r_-(E)\leq|z|\leq r_+(E)$. As in Case 1, $\mathcal C$ is finite. Hence there exists a non-empty open interval $J\subset I\setminus \mathcal C$ such that for every $E\in J$ all annular zeros of $L_{E}^{}$ are simple. By the implicit function theorem,  after shrinking $J$ if necessary, there exists an open complex neighborhood $U\subset\mathbb C$ of $J$ such that the annular zeros can be labeled analytically $\mu_1(E),\dots,\mu_{2m}(E), E\in U$.
Since $G(E,\omega)=0$ for every $E\in J$, for each $E\in J$ there exists at least one pair $(i,j)$, $i\neq j$, such that
$
\mu_i(E)-\omega \mu_j(E)=0.
$
There are only finitely many such pairs. Hence by the pigeonhole principle, there exist indices $(i_0,j_0)$, $i_0\neq j_0$, such that the set
$
\{E\in J: \mu_{i_0}(E)-\omega \mu_{j_0}(E)=0\}
$
has an accumulation point in $J$. Since $E\mapsto \mu_{i_0}(E)-\omega \mu_{j_0}(E)$
is analytic, the identity theorem yields $\mu_{i_0}(E)\equiv \omega \mu_{j_0}(E), E\in U.$
Since $V(\mu_{i_0}(E))=E, V(\mu_{j_0}(E))=E.$
Therefore $$
V(\omega \mu_{j_0}(E))=V(\mu_{i_0}(E))=E=V(\mu_{j_0}(E)), E\in J.
$$
Set
$
b(E):=\mu_{j_0}(E).
$
Differentiating $V(b(E))=E$ gives $V'(b(E))  b'(E)=1,$ hence $b'(E)\neq 0$ on $U$. Thus $b(U)$ is an open subset of the annulus. On this open set we have $V(\omega z)=V(z).$
By analytic continuation, this identity holds throughout the connected domain of analyticity of $V$. 
Writing $\omega=e^{2\pi i u}$, we obtain on the unit circle $V(\theta+u)=V(\theta),$ so $V$ has a positive period strictly smaller than $1$,
contradicting the assumption that $V$ has minimal positive period $1$.

Now we prove (2). We claim that there exists $r\ge 0$ such that for every $(E,u)\in I\times\mathbb T$,
$
\max_{0\le k\le r}\left|\partial_E^k G(E,e^{2\pi i u})\right|>0.
$
Suppose not. Then for each $n\ge 0$ there exists $(E_n,u_n)\in I\times\mathbb T$ such that
$
\partial_E^k G(E_n,e^{2\pi i u_n})=0, 0\le k\le n.
$
Passing to a subsequence, we may assume
$
(E_n,u_n)\to(E_*,u_*)\in I\times\mathbb T.
$
By continuity of the derivatives,
$\partial_E^k G(E_*,e^{2\pi i u_*})=0, k\ge 0.$
Hence the analytic function $ G(E,e^{2\pi i u_*})$
has all derivatives vanishing at $E_*$, so it is identically zero on $I$, contradiction. The claim follows.

Now define $H(E,u):=\max_{0\le k\le r}\left|\partial_E^k G(E,e^{2\pi i u})\right|.$
Then $H$ is continuous and strictly positive on the compact set $I\times\mathbb T$. Therefore
$2c_\infty:=\min_{(E,u)\in I\times\mathbb T} H(E,u)>0.$
This proves (2).
\end{proof}

Hence by Corollary \ref{cor:zero} and Lemma \ref{lem:G-analytic}, there exists $\varepsilon_j$, such that $C_0^\infty(E)$ satisfy the non-degeneracy condition \eqref{def-non-degeneracy}. Consequently  the cocycle
$ C_\varepsilon^{j,\infty}(E,\cdot)$
satisfies all assumptions of Theorem \ref{Thm2} on $I_{\delta_1}$. Applying Theorem \ref{Thm2} on each $J^j$, we obtain a Hausdorff dimension zero set $\Theta_j$, such that for every
$E\in J^j\backslash\Theta_j,$
there exist
$B_j(E,\cdot)\in C_h^\omega(\mathbb T^d,\GL(2m_j,\mathbb C))$
and a diagonal matrix
$\Lambda_j(E)\in \GL(2m_j,\mathbb C)$
satisfying
$$
B_j(E,\theta+\alpha)^{-1}C_{\varepsilon}^{j,\infty}(E,\theta)B_j(E,\theta)=\Lambda_j(E).
$$
Let $\varepsilon_r=\min_{1\leq j\leq k}\{\varepsilon_j\}$ and $\Theta=\cup \Theta_j$, 
the conclusion follows. \qed

\subsection*{Proof of Theorem \ref{thm:IDSac}}
    Using Theorem \ref{main:reduce}, 
fix
$E\in\sigma(\mathbf L_{V,\varepsilon W,\alpha,\theta})\setminus\Theta .$
Choose $j$ such that $E\in J^j$.  Then for $\epsilon>0$, with $E+i\epsilon\in J^j_{\delta_1}$, we have
$$
B_{j,E}(\theta+\alpha)^{-1}C_{\varepsilon}^{j,\infty}(E+i\epsilon,\theta)B_{j,E}(\theta)=D_{j,E}+F_{j,E,\epsilon}(\theta),
$$
where, by analyticity in $E$,
$\|F_{j,E,\epsilon}\|_0\le C_j(E)\epsilon.$

 Therefore, we have  
\begin{align*}
\sum_{i=1}^mL_i(C_{\varepsilon}^{j,\infty}(E+i\epsilon,\theta))&=\lim_{n\rightarrow\infty}\frac{1}{n}\int_{\T^d}\ln\|\Lambda^{m}(D_{j,E}(I+D_{j,E}^{-1} F_{j,E,\epsilon}))(\theta;n)\|d\theta \\
&\leq\lim_{n\rightarrow\infty}\frac{1}{n}\int_{\T^d}\ln\|\Lambda^{m}D_{j,E}\|^{n}d\theta+
\lim_{n\rightarrow\infty}\frac{1}{n}\int_{\T^d}\ln\|\Lambda^{m}(I+D_{j,E}^{-1}F_{j,E,\epsilon})\|^{n}d\theta\\
&\leq\sum_{i=1}^mL_i(C_{\varepsilon}^{j,\infty}(E,\theta))+m\ln\|I+D_{j,E}^{-1}F_{j,E,\epsilon}\|_{0}\\
&\leq \sum_{i=1}^mL_i(C_{\varepsilon}^{j,\infty}(E,\theta))+m\tilde C(E)\epsilon.
\end{align*}
Thus by  Proposition \ref{prop:thouless}, we have
$$
\begin{aligned}
   {\widehat{\mathcal{N}}}(E+\epsilon)-{\widehat{\mathcal{N}}}(E-\epsilon)&<\frac{1}{2}\int_{J^j_{\delta_1}\cap\R} \log(1+\frac{\epsilon^{2}}{(E-t)^{2}})d {\widehat{\mathcal{N}}}(t)\\
    &=\sum_{i=1}^mL_i(C_{\varepsilon}^{\infty}(E+i\epsilon,\theta))-\sum_{i=1}^mL_i(C_{\varepsilon}^{\infty}(E,\theta))-h_\varepsilon(E+i\epsilon)+h_\varepsilon(E)\\
    &\leq m\tilde C(E)\epsilon +\hat{C}(E)\epsilon.
\end{aligned}
$$
Hence,  for any $E\notin\Theta$, $\mathcal {N}(E)$ is Lipschitz continuous, so the singular part is supported  only  on $\Theta$.  By Theorem \ref{thm:holder-GYZ}, there exists $C=C(V,W,\alpha)>0$, such that  $d{\widehat{\mathcal{N}}}(U)\leq C|U|^{\frac{1}{4m(E)}}$ for any open interval $U$. 
Now due to $\Theta$ is a set with Hausdorff  dimension zero, then  for any $\zeta>0$ we can find  a cover of $\Theta$, denoted by $\{U_{i}\}_{i=1}^{\infty}$, such that
$\Sigma_{i=1}^\infty\textrm{diam} (U_i)^{\frac{1}{4m(E)}}<C_*^{-1}\zeta.$
Then we have
$$
d\widehat{\mathcal {N}}(\Theta)\leq\Sigma_{i=1}^\infty d{\widehat{\mathcal{N}}}(U_{i})\leq C_*\Sigma_{i=1}^\infty\textrm{diam}(U_i)^{\frac{1}{4m(E)}}<\zeta.
$$
Then by the arbitrariness of $\zeta>0$, we get $d\widehat{\mathcal{N}}(\Theta)=0$ and by Proposition \ref{IDSeq} the result follows. \qed

\section{Anderson localization}

The core proof of Anderson localization  is the following Proposition \ref{prop:rigidity}, which can be viewed as a rigidity type
result in dynamical systems:

\begin{proposition}\label{prop:rigidity}
Let $h>0$, $\varepsilon>0$, and let $\alpha \in \mathrm{DC}_d$. Denote by $\mathcal{R}_c$ the set of energies $E$ for which the $2m$-dimensional center cocycle $(\alpha, C_\varepsilon^\infty(E, \cdot))$ is $C_h^\omega$-reducible, that is, there exists $B \in C_h^\omega(\mathbb{T}^d, \mathrm{GL}(2m, \mathbb{C}))$ such that
\begin{equation*}\label{eq:icc-reducibility}
B(\theta+\alpha)^{-1} C_\varepsilon^\infty(E, \theta) B(\theta) = \Lambda := 
\diag\bigl(e^{2\pi i\rho_1},\dots,e^{2\pi i\rho_{2m}}\bigr).
\end{equation*}
Assume $E \in \mathcal{R}_c$ and $x \in \mathbb{T}$. If $E$ is an eigenvalue of the operator $\mathbf{H}_{\varepsilon W,V,\alpha,x}$ with an eigenfunction $u \in \ell^1(\mathbb{Z}^d)$, then the following hold:
\begin{enumerate}
    \item there exist an index $1 \le j \le 2m$ and $k \in \mathbb{Z}^d$ such that 
    $\rho_j - x = \langle k, \alpha \rangle \mod \mathbb{Z},$
    and in particular, $\Im \rho_j = 0$;
    \item the corresponding eigenfunction $u \in \ell^1(\mathbb{Z}^d)$ decays exponentially: $|u_n|\le C  e^{-2\pi h|n|}.$

\end{enumerate}
\end{proposition}
\begin{proof}
     Let $u=(u_n)_{n\in\mathbb{Z}^d}\in\ell^1(\mathbb{Z}^d)$ satisfy
$\mathbf{H}_{\varepsilon W,V,\alpha,x}u=Eu$.
Define its Fourier transform by
$\widehat u(\theta)=\sum_{n\in\mathbb{Z}^d} u_n  e^{2\pi i\langle n,\theta\rangle}\in C^0(\mathbb{T}^d,\C).
$
By the standard Aubry duality, for  every $\theta\in\mathbb{T}^d$ the sequence
$w_n(\theta):=\widehat u(\theta+n\alpha) e^{2\pi i n x}, n\in\mathbb{Z},$
is a nontrivial solution of the dual equation
$\mathbf{L}_{V,\varepsilon W,\alpha,\theta} w(\theta)=E w(\theta).
$ Embedding this into the global phase space, it follows that 
    $ w(\theta)  \in C^0(\mathbb{T}^d, \mathcal{X}_\xi) $
    and therefore $w(\theta) \in \mathcal{M}^c_{V,E,\varepsilon}(\theta)$ for all $\theta$.  Meanwhile, by the definition of the dual equation and the shift operator $S$, 
    $$(S w)_n(\theta) = w_{n+1}(\theta)= \widehat u(\theta+(n+1)\alpha) e^{2\pi i (n+1) x}=e^{2\pi i x}w_n(\theta+\alpha),$$
    which means $w$ satisfies the exact covariance relation
    \begin{equation}\label{eq:w-cov}
        S w(\theta) = e^{2\pi i x} w(\theta+\alpha).
    \end{equation}

According to the global trivialization of the center bundle (Theorem \ref{thm:centert}), there exists a unique coefficient vector $d(\cdot) \in C^0(\mathbb{T}^d, \mathbb{C}^{2m})$ such that 
    $ w(\theta) = U_\varepsilon^{\infty}(E,\theta) d(\theta),$
    where $U_\varepsilon^{\infty}(E,\theta)$ is the analytic frame defined in Theorem \ref{thm:centert}. Substituting this into \eqref{eq:w-cov} and utilizing the intrinsic cocycle relation $S U_\varepsilon^{\infty}(E,\theta) = U_\varepsilon^{\infty}(E,\theta+\alpha)C_{\varepsilon}^{\infty}(E,\theta)$, we obtain the reduced dynamical equation on $\mathbb{C}^{2m}$:
    $$
    C_{\varepsilon}^{\infty}(E,\theta) d(\theta) = e^{2\pi i x} d(\theta+\alpha).
    $$

We now invoke the reducibility hypothesis.
Since
$B(\theta+\alpha)^{-1}C_{\varepsilon}^{\infty}(E,\theta)B(\theta)=\Lambda$, if we  set $d(\theta) = B(\theta) c(\theta)$, we have
$$\Lambda  c(\theta)=e^{2\pi i x}  c(\theta+\alpha).$$
Writing $c=(c_1,\dots,c_{2m})^\top$, this implies that for each
$1\le j\le 2m$ and every $\theta$,
$$c_j(\theta+\alpha)=e^{2\pi i(\rho_j-x)}  c_j(\theta).
$$
 There exists at least one $j$ such that $c_j(\cdot)\neq0$.  By examining the Fourier series of $c_j$, there exists a unique $k \in \mathbb{Z}^d$ such that $
\rho_j - x = \langle k, \alpha \rangle \mod \mathbb{Z},
$ which in particular implies (1) and $
c_j(\theta) = a_j e^{2\pi i \langle k, \theta \rangle},
$ where $a_j \in \mathbb{C}$ is a constant. 
The same argument applies to each component $c_j$,
every nontrivial component is a single Fourier mode, while the remaining ones
vanish identically.
Consequently, $c(\cdot) \in C_h^\omega(\mathbb{T}^d, \mathbb{C}^{2m})$. Since $U_\varepsilon^{\infty}(E,\theta)$ is analytic, it follows that $w(\cdot) \in C_h^\omega(\mathbb{T}^d)$, and consequently, $\widehat{u}(\theta) =w_0(\theta)  \in C_h^\omega(\mathbb{T}^d, \mathbb{C})$ with $
|u_n| \leq \|w\|_h e^{-2\pi h |n |}.
$
\end{proof}

\subsection*{Proof of Theorem \ref{mainthe:loca}}
The starting point is Eliasson’s famous pure point result:
  \begin{theorem}[\cite{E}]\label{cor-Eliasson}
    Suppose that $\alpha\in \mathrm{DC}_d$.  
Then there exist $\varepsilon_0=\varepsilon_0(\alpha,V,W)>0$  and a Borel measurable function  $E\colon\mathbb{T}\to\mathbb{R}$
such that for all $|\varepsilon|<\varepsilon_0$, for almost every $x\in\mathbb{T}$, the values $
E\bigl(x+\langle k,\alpha\rangle\bigr),$ $k\in\mathbb{Z}^d,$
are eigenvalues of $\mathbf H_{\varepsilon W,V,\alpha,x}$. Furthermore, the corresponding eigenfunctions
$\{u^k\}_{k\in\mathbb{Z}^d}$ belong to $\ell^1(\mathbb{Z}^d)$ and form a complete basis of
$\ell^2(\mathbb{Z}^d)$.
\end{theorem}

Let $\mathcal{A}\subset\mathbb{T}$ denote the full-measure set of phases
for which the conclusion of Theorem~\ref{cor-Eliasson} holds.
Fix $x\in \mathcal{A}$. The operator $\mathbf H_{\varepsilon W,V,\alpha,x}$ has pure point spectrum with eigenvalues
$E(x+\langle k,\alpha\rangle), k\in\mathbb{Z}^d.$
Define the ``bad'' eigenvalues
$$\mathcal{E}_x:=\{E(x+\langle k,\alpha\rangle):\ E(x+\langle k,\alpha\rangle)\in\Theta\},$$
where $\Theta$ is the exceptional set appearing in Theorem~\ref{main:reduce} and 
let
 $\mathcal{E}:=\bigcup_{x\in\mathcal{A}}\mathcal{E}_x.$
Clearly, $\mathcal{E}\subset \Theta$. What's important for us is the following: 

\begin{lemma}\label{nonreducible}
    For almost every $x\in\mathcal{A}$, one has $\mathcal{E}_x = \emptyset$.
\end{lemma}

\begin{proof}
Let $d\mathcal{N}$ be the density of states and $\mu_{x,\delta_p}$ the spectral measure of
$\mathbf{H}_{\varepsilon W,V,\alpha,x}$  corresponding to the vector $\delta_p\;(p\in\mathbb{Z}^d)$.
Since $d\mathcal{N}$ is absolutely continuous (Theorem \ref{thm:IDSac}) and the exceptional set
$\Theta$ has Lebesgue measure zero (Theorem \ref{main:reduce}), we have $d\mathcal{N}(\Theta)=0$.  By definition,
$$
d\mathcal{N}(\Theta) = \int_{\mathbb{T}} \mu_{x,\delta_p}(\Theta) dx,
\qquad p\in\mathbb{Z}^d.
$$
Hence for each fixed $p$, the integrand vanishes for almost every $x$. Taking the
countable intersection over all $p$, we obtain a full‑measure set
$\mathcal{F}\subset\mathcal{A}$ such that
$\mu_{x,\delta_p}(\Theta) = 0, $ for all $ x\in\mathcal{F}, p\in\mathbb{Z}^d.$

Because $\mathcal{E}_x\subset\Theta$, it follows that
$\mu_{x,\delta_p}(\mathcal{E}_x)=0$ for all $x\in\mathcal{F}$ and all $p\in\mathbb{Z}^d$.
Now suppose, for some $x\in\mathcal{F}$, that $\mathcal{E}_x\neq\emptyset$. 
By Theorem \ref{cor-Eliasson} the operator $\mathbf{H}_{\varepsilon W,V,\alpha,x}$  has pure point
spectrum with eigenvalues $\{E(x+\langle k,\alpha\rangle)\}_{k\in\mathbb{Z}^d}$.
Consequently any $E\in\mathcal{E}_x$ is an eigenvalue and there exists $p\in\mathbb{Z}^d$
for which $\mu_{x,\delta_p}(\{E\})>0$ and thus
$\mu_{x,\delta_p}(\mathcal{E}_x)\ge\mu_{x,\delta_p}(\{E\})>0,$
contradicting $\mu_{x,\delta_p}(\mathcal{E}_x)=0$ for $x\in \mathcal{F}$. Therefore $\mathcal{E}_x=\emptyset$
for all $x\in\mathcal{F}$, i.e., for almost every $x$.
\end{proof}
By Lemma~\ref{nonreducible}, for almost every $x \in \mathbb{T}$ and all $k \in \mathbb{Z}^d$, the center cocycle at energy $E(x+\langle k,\alpha\rangle)$ is $C_h^\omega$-reducible. Proposition~\ref{prop:rigidity} then implies that the corresponding eigenfunction $u^k$ decays exponentially. Since Theorem~\ref{cor-Eliasson} ensures that the family $\{u^k\}_{k\in\mathbb{Z}^d}$ forms a complete basis of $\ell^2(\mathbb{Z}^d)$, it directly follows that $\mathbf{H}_{\varepsilon W,V,\alpha,x}$ admits a complete set of exponentially decaying eigenfunctions.
\qed

\appendix

\section{Proof of Lemma \ref{lem:unperturbed-resolvent}}\label{appendixa}
Let $R^{l}_0(z)= \bigl[r_{i,j}^0(z)\bigr]_{i,j=1}^{2l}.$
It suffices to verify $(z-A_{0}^{l}(E))R^{l}_0(z)=\mathrm{Id}$.
For $i=j=1$, since $T_1(z)=\widetilde v_l z^{l-1}$, the identity $\sum_{k=0}^{2l}\widetilde v_{l-k}z^{-(k-1)} = z^{-l+1}L^{l}_{E}(z)$ gives
$$
\begin{aligned}
((z-A_{0}^{l}(E))R^{l}_0(z))_{1,1}
&= \frac{1}{\tilde v_l}\Bigl(\tilde v_l z r_{1,1}^0(z) + \sum_{k=1}^{2l}{\tilde v_{l-k}}z^{-(k-1)}r_{1,1}^0(z)\Bigr)\\
&= \frac{z^{-l+1}L^{l}_{E}(z)}{\tilde v_l}\cdot\frac{\widetilde v_l z^{l-1}}{L^{l}_{E}(z)} = 1.
\end{aligned}
$$
For $i=1$, $j\ge 2$,
$$
\begin{aligned}
((z-A_{0}^{l}(E))R^{l}_0(z))_{1,j}
&= \frac{1}{\tilde v_l}\Biggl[
\frac{T_j(z)}{L^{l}_{E}(z)}
\Bigl(
\tilde v_l z^j + \sum_{k=1}^{j-1}\tilde v_{l-k} z^{j-k} + \sum_{k=j}^{2l}\tilde v_{l-k} z^{-(k-j)}
\Bigr)\\
&\qquad - \Bigl(\tilde v_l z^{j-1} + \sum_{k=1}^{j-1}\tilde v_{l-k} z^{j-k-1}\Bigr)\Biggr]\\
&= \frac{1}{\tilde v_l}\Bigl[\frac{T_j(z)}{L^{l}_{E}(z)}z^{-(l-j)}L^{l}_{E}(z) - z^{-(l-j)}T_j(z)\Bigr] = 0.
\end{aligned}
$$
For $i\ge 2$, we have
$$
\begin{aligned}
((z-A_{0}^{l}(E))R^{l}_0(z))_{i,j}
&= zr_{j,j}^0(z)-r_{j-1,j}^0(z) && (\text{if } i=j)\\
&= zr_{j,j}^0(z)-(zr_{j,j}^0(z)-1) = 1;\\
&= -r_{i-1,j}^0(z)+zr_{i,j}^0(z) && (\text{if } i>j)\\
&= -z^{-(i-1-j)}\frac{T_j(z)}{L^{l}_{E}(z)} + z\cdot z^{-(i-j)}\frac{T_j(z)}{L^{l}_{E}(z)} = 0;\\
&= zr_{i,j}^0(z)-r_{i-1,j}^0(z) && (\text{if } i<j)\\
&= z\bigl(z^{j-i}r_{j,j}^0(z)-z^{j-i-1}\bigr) - \bigl(z^{j-i+1}r_{j,j}^0(z)-z^{j-i}\bigr) = 0.
\end{aligned}
$$
\qed

\section*{Acknowledgements} 
Qi Zhou would like to thank A. Avila, S. Jitomirskaya  for insightful discussions, which ultimately led  to this paper. 
Xianzhe Li is supported by an AMS-Simons Travel Grant.
 Zhenfu Wang is supported by NSFC grant (124B2011) and Nankai Zhide Foundation. Jiangong You and  Qi Zhou are supported by NSFC grant (12531006, 12526201) and Nankai Zhide Foundation.

\end{document}